\documentclass[11pt]{amsart}

\usepackage{amsmath,amssymb,amsthm}
\usepackage{thmtools,enumerate}
\usepackage{mathtools}

\usepackage[german,english]{babel}
\usepackage[autostyle]{csquotes}

\usepackage[all]{xy}
\usepackage{pstricks}

\usepackage{tikz-cd}
\usetikzlibrary{babel}

\usepackage{hyperref}
\hypersetup{
	colorlinks=true,
	linkcolor=red,
	citecolor=blue}

\usepackage{esint}

\theoremstyle{plain}

\newtheorem{thm}{Theorem}[section]
\newtheorem{prop}[thm]{Proposition}
\newtheorem{lem}[thm]{Lemma}
\newtheorem{cor}[thm]{Corollary}

\newtheorem{thmA}{Theorem}

\theoremstyle{definition}

\newtheorem{dfn}[thm]{Definition}
\newtheorem{rem}[thm]{Remark}
\newtheorem{exa}[thm]{Example}

\newcommand{\N}{\mathbb{N}}
\newcommand{\Q}{\mathbb{Q}}
\newcommand{\R}{\mathbb{R}}
\newcommand{\C}{\mathbb{C}}
\newcommand{\OO}{\mathcal{O}}
\newcommand{\can}{\mathrm{can}}
\newcommand{\alg}{\mathrm{alg}}
\newcommand{\loc}{\mathrm{loc}}

\newcommand{\dbar}{{\overline{\partial}}}

\newcommand{\lla}[0]{{\langle\!\hspace{0.04cm} \!\langle}}
\newcommand{\rra}[0]{{\rangle\!\hspace{0.04cm}\!\rangle}}

\DeclareMathOperator{\codim}{codim}
\DeclareMathOperator{\mult}{mult}
\DeclareMathOperator{\Exc}{Exc}

\DeclareMathOperator{\Supp}{Supp}

\DeclareMathOperator{\diam}{diam}

\DeclareMathOperator{\Bs}{Bs}
\DeclareMathOperator{\sB}{\mathbf{B}}
\DeclareMathOperator{\PSH}{\mathrm{PSH}}
\DeclareMathOperator{\ddiv}{div}
\DeclareMathOperator{\Div}{Div}

\makeatletter
\renewcommand{\tocsection}[3]{%
  \indentlabel{\@ifnotempty{#2}{\ignorespaces#1 \makebox[1em][r]{#2}.\quad}}#3}
\renewcommand{\tocsubsection}[3]{%
  \indentlabel{\@ifnotempty{#2}{\ignorespaces#1 \indent@subsec@num#2.\quad}}#3}
\def\indent@subsec@num#1{%
  \ifx#1\@secnumber
    \@secnumber
  \else
    \expandafter\indent@subsec@num@aux\expandafter#1%
  \fi
}
\def\indent@subsec@num@aux#1.#2.{\makebox[1em][r]{#1}.#2.}
\makeatother

\begin{document}
\title[Minimal metrics and the Abundance conjecture]{Metrics with minimal singularities and\\ the Abundance conjecture}

	\author{Vladimir Lazi\'c}
	\address{Fachrichtung Mathematik, Campus, Geb\"aude E2.4, Universit\"at des Saarlandes, 66123 Saarbr\"ucken, Germany}
	\email{lazic@math.uni-sb.de}
		
	\dedicatory{To Thomas Peternell on the occasion of his 70th birthday, with admiration}

	\begin{abstract}
	The Abundance conjecture predicts that on a minimal projective klt pair $(X,\Delta)$, the adjoint divisor $K_X+\Delta$ is semiample. When $\chi(X,\OO_X)\neq0$, we give a necessary and sufficient condition for the conjecture to hold in terms of the asymptotic behaviour of multiplier ideals of currents with minimal singularities of small twists of $K_X+\Delta$. Furthermore, we prove fundamental structural properties as well as regularity and weak convergence behaviour of an important class of currents with minimal singularities: the supercanonical currents. The results of the paper indicate strongly that supercanonical currents are central to the completion of the proof of the Abundance conjecture for minimal klt pairs $(X,\Delta)$ with $\chi(X,\OO_X)\neq0$. 
	\end{abstract}

\maketitle

	\begingroup
		\hypersetup{linkcolor=black}
		\setcounter{tocdepth}{2}
			\makeatletter
				\def\l@subsection{\@tocline{2}{0pt}{2.5pc}{5pc}{}}
			\makeatother
		\tableofcontents
	\endgroup

\newpage

\section{Introduction}

The Abundance conjecture is one of the most important open problems in algebraic geometry. It predicts that on a projective klt pair $(X,\Delta)$, if the adjoint divisor $K_X+\Delta$ is nef, then it is semiample; in other words, there exist a fibration $f\colon X\to Z$ and an ample $\R$-divisor $A$ on $Z$ such that $K_X+\Delta\sim_\R f^*A$. The conjecture is classically known for curves and surfaces, whereas for threefolds it was a fantastic achievement obtained in \cite{Miy87,Miy88a,Miy88b,Kaw92,KMM94}. In arbitrary dimension, the conjecture holds for pairs of log general type \cite{Sho85,Kaw85b}, for pairs of numerical dimension $0$ \cite{Nak04}, and for varieties satisfying Miyaoka's equality \cite{IMM24}. 

In dimensions at least $4$, up to now there has only been one general result due to Lazi\'c and Peternell \cite{LP18,LP20b}, and to Gongyo and Matsumura \cite{GM17}: assuming the Minimal Model Program in lower dimensions, the divisor $K_X+\Delta$ is semiample if $\chi(X,\OO_X)\neq0$ and if the pullback of $K_X+\Delta$ to a resolution of $X$ is hermitian semipositive. Very little seems to be known about the Abundance conjecture in higher dimensions when $\chi(X,\OO_X)=0$, unless $X$ is uniruled \cite{LM21}.

The papers \cite{LP18,LP20b} show, more generally, that half of the Abundance conjecture -- the Nonvanishing conjecture -- holds when $\chi(X,\OO_X)\neq0$, if the pullback of $K_X+\Delta$ to a resolution of $X$ has a \emph{singular metric with generalised algebraic singularities}. This class of metrics, discussed in detail in \S\ref{subsec:generalisedanalytic}, is a singular generalisation of hermitian semipositive metrics and is a natural class of metrics from the point of view of the Minimal Model Program. Op.\ cit.\ indicated strongly that understanding this class of metrics is crucial for progress on the Abundance conjecture.

The quest for metrics with generalised algebraic singularities on adjoint divisors $K_X+\Delta$ is the main motivation for this paper.

The best candidates for such metrics are \emph{metrics with minimal singularities}. Singular metrics with minimal singularities on an $\R$-divisor $L$ on a compact K\"ahler manifold induce the smallest norms among all possible positively curved singular metrics on $L$, modulo certain compatibility conditions for singularities of metrics. Such metrics are notoriously difficult to work with as they are usually very transcendental and can be only implicitly described. However, they have some very good properties which we recall in Section \ref{sec:minimalsings}, which distinguish them from other singular metrics on $L$.

In this paper we investigate how metrics with minimal singularities on divisors $K_X+\Delta+\varepsilon A$ behave when $\varepsilon\downarrow0$, where $A$ is an ample divisor on $X$. We prove two main results:
\begin{enumerate}[\normalfont (a)]
\item the Abundance conjecture can be reinterpreted as a statement about good asymptotic behaviour of multiplier ideals of currents with minimal singularities, and
\item supercanonical currents are excellent candidates to prove such good behaviour of multiplier ideals, and thus complete the proof of the Abundance conjecture when $\chi(X,\OO_X)\neq0$.
\end{enumerate}

\smallskip

\paragraph{\textbf{Notation.}}

If $T$ is a closed positive current a compact K\"ahler manifold $X$, we use the notation $\mathcal I(T)_{\min}$ for the multiplier ideal of any closed positive current with minimal singularities in the cohomology class of $T$, see \S\ref{subsec:multiplier} and \S\ref{subsection:comparisonsings}.

\medskip

\noindent{\sc The first main result}

\medskip

Our first main result is that on a minimal klt pair $(X,\Delta)$ with $\chi(X,\OO_X)\neq0$, the Abundance conjecture is equivalent to an approximation property of multiplier ideals of currents with minimal singularities associated to divisors $K_X+\Delta$ and $K_X+\Delta+\frac1m A$ for $m\in\N_{>0}$, where $A$ is an ample divisor on $X$.\footnote{We prove this assuming the Minimal Model Program in lower dimensions. This is a natural and necessary condition in all current work on the Abundance conjecture, considering that we aim to prove it by induction on the dimension.} Roughly speaking, this approximation property says that the multiplier ideals of currents with minimal singularities associated to large multiples of $K_X+\Delta$ and $K_X+\Delta+\frac1m A$ are almost the same when $m\to\infty$.

This statement has two parts, given in Proposition \ref{prop:MMPimpliesapproximation} and Theorem \ref{thm:main1}. The first observation is that this approximation statement of multiplier ideals is a consequence of the Abundance conjecture: this is the content of the following proposition, whose proof is given in Section \ref{sec:goodapproximations}.

\begin{prop}\label{prop:MMPimpliesapproximation}
Let $(X,\Delta)$ be a projective klt pair such that $K_X+\Delta$ is semiample. Let $\pi\colon Y\to X$ be a log resolution of $(X,\Delta)$ and write 
$$K_Y+\Delta_Y\sim_\R\pi^*(K_X+\Delta)+E,$$
where $\Delta_Y$ and $E$ are effective $\R$-divisors without common components. Let $A$ be an ample $\R$-divisor on $Y$. Then there exist an effective divisor $D$ on $Y$ and a sequence of positive integers $\{m_\ell\}_{\ell\in\N_{>0}}$ such that $m_\ell\to\infty$ and
$$\textstyle\mathcal I\big(\ell(K_Y+\Delta_Y+\frac{1}{m_\ell} A)\big)_{\min}\subseteq \mathcal I\big(\ell(K_Y+\Delta_Y)\big)_{\min}\otimes\OO_Y(D)\quad\text{for all }\ell.$$
\end{prop}

To explain the conclusion of this proposition, note that, in its notation, we always have 
$$\textstyle\mathcal I\big(\ell(K_Y+\Delta_Y)\big)_{\min}\subseteq\mathcal I\big(\ell(K_Y+\Delta_Y+\frac{1}{m_\ell} A)\big)_{\min}\quad\text{for all }\ell$$
by Lemma \ref{lem:descendingminimal}. Therefore, Proposition \ref{prop:MMPimpliesapproximation} says that the multiplier ideals $\mathcal I\big(\ell(K_Y+\Delta_Y)\big)_{\min}$ and $\mathcal I\big(\ell(K_Y+\Delta_Y+\frac{1}{m_\ell} A)\big)_{\min}$ are almost equal.

The first main result of the paper is that for pairs $(X,\Delta)$ with $\chi(X,\OO_X)\neq0$ we have the converse to Proposition \ref{prop:MMPimpliesapproximation}.

\begin{thmA}\label{thm:main1}
Assume the existence of good minimal models for projective klt pairs in dimensions at most $n-1$.

Let $(X,\Delta)$ be a projective klt pair of dimension $n$ such that $K_X+\Delta$ is nef and $\Delta$ is a $\Q$-divisor. Let $\pi\colon Y\to X$ be a log resolution of $(X,\Delta)$ and write 
$$K_Y+\Delta_Y\sim_\Q\pi^*(K_X+\Delta)+E,$$
where $\Delta_Y$ and $E$ are effective $\Q$-divisors without common components. Let $A$ be an ample $\R$-divisor on $Y$, and assume that there exist an effective divisor $D$ on $Y$ and a sequence of positive integers $\{m_\ell\}_{\ell\in\N_{>0}}$ such that $m_\ell\to\infty$ and
$$\textstyle\mathcal I\big(\ell(K_Y+\Delta_Y+\frac{1}{m_\ell} A)\big)_{\min}\subseteq \mathcal I\big(\ell(K_Y+\Delta_Y)\big)_{\min}\otimes\OO_Y(D)\quad\text{for all }\ell.$$
If $\kappa(X,K_X+\Delta)\geq0$ or $\chi(X,\OO_X)\neq0$, then $K_X+\Delta$ is semiample.
\end{thmA}

Part \ref{part:asymptoticapproximations} of the paper is dedicated to the proof of this theorem. It follows immediately from Theorem \ref{thm:main1a}, which proves a much more precise statement.

Theorem \ref{thm:main1} and Proposition \ref{prop:MMPimpliesapproximation} together show that, when $\chi(X,\OO_X)\neq0$, the Abundance conjecture is a statement about the behaviour of multiplier ideals of currents with minimal singularities. (We stress that this does not depend on any particular \emph{choice} of currents with minimal singularities: this gives significant flexibility that we will exploit several times in the paper). This is the first main contribution of this work.

We explain briefly the strategy of the proof of Theorem \ref{thm:main1}. First we introduce and study in detail \emph{asymptotically equisingular approximations}: a sequence of closed almost positive $(1,1)$-currents $\{T_m\}_{m\in\N}$ on a compact K\"ahler manifold $X$ is an asymptotically equisingular approximation of a closed almost positive $(1,1)$-current $T$ on $X$ if there exist an effective divisor $D$ on $X$ and a sequence of positive integers $\{m_\ell\}_{\ell\in\N_{>0}}$ such that $m_\ell\to\infty$ and we have the inclusions of multiplier ideals
$$\mathcal I(\ell T_{m_\ell})\otimes\OO_X({-}D)\subseteq\mathcal I(\ell T)\subseteq \mathcal I(\ell T_{m_\ell})\otimes\OO_X(D)\quad\text{for all }\ell.$$
Note that we do not require that the currents $T_m$ converge weakly to $T$, hence asymptotically equisingular approximations would seem to be too weak for successful applications in practice. We will see, however, that they are perfectly suited to the context of the Minimal Model Program. Stronger forms of approximations appeared in connection to the regularisation techniques of Demailly \cite{Dem92,DPS01,Cao14}, but equisingular approximations considered there do not seem suitable for applications within the Minimal Model Program; they did, however, motivate the definition of asymptotically equisingular approximations, as will be apparent in Sections \ref{sec:equisingular} and \ref{sec:goodapproximations}.

In order to make asymptotically equisingular approximations useful within the context of the Minimal Model Program, we introduce in Section \ref{sec:excellentapproximations} a much stronger version of approximations of currents: \emph{excellent approximations}. We show in Theorem \ref{thm:valuationscurrents} that the existence of excellent approximations of a current with minimal singularities $T$ is equivalent to $T$ having generalised analytic singularities. We combine this information in Section \ref{sec:excellentandMMP} with the techniques from \cite{LP18,LP20b} to deduce certain strong cohomological properties of the sheaves of differential forms. Finally, in Theorem \ref{thm:main1a} we show that in the context of the Minimal Model Program, asymptotically equisingular approximations are always excellent: this allows to prove the Nonvanishing by further application of the methods from \cite{LP18,LP20b}, and then semiampleness follows from the main result of \cite{GM17}.

For completeness we remark here that when $X$ is uniruled in Theorem \ref{thm:main1}, then we know that $\kappa(X,K_X+\Delta)\geq0$ by \cite[Theorem 1.1]{LM21}. Therefore, when it comes to the Nonvanishing conjecture, the main remaining case is the case of non-uniruled varieties.

\medskip

\noindent{\sc The second main result}

\medskip

Theorem \ref{thm:main1} implies that understanding multiplier ideals of currents with minimal singularities, and especially their behaviour under perturbations, is fundamental for the proof of the Abundance conjecture. This is where supercanonical currents enter the picture.

The second goal of the paper is to study in detail a very specific choice of currents with minimal singularities: the \emph{supercanonical currents} introduced by Tsuji in \cite{Tsu07,Tsu11} and investigated in much greater generality and detail by Berman and Demailly \cite{BD12}. The origins of supercanonical currents can be traced back to the work of Narasimhan and Simha \cite{NS68}, where they were examined in the case of ample line bundles. We study supercanonical currents in detail in Section \ref{sec:supercanonical}.

We explain first the main idea behind supercanonical currents. Usually, the existence of at least one current with minimal singularities in a pseudoeffective cohomology class is shown by using a suitable $L^\infty$-condition; this is explained in Section \ref{sec:minimalsings}. This seems, however, not to be suited for use in birational geometry. In contrast, supercanonical currents are defined by an exponential $L^1$-condition, see Section \ref{sec:supercanonical}. On a technical level, this makes them adapted to proofs involving estimates in which one uses H\"older's inequality. Crucially for us, this allows to use techniques of Berman, Demailly and others to show that supercanonical currents can actually be calculated by using only algebraic data: concretely, a supercanonical current of a big line bundle $L$ depends only on the global holomorphic sections of powers of $L$, see Theorem \ref{thm:supercanbig}. A large portion of Part \ref{part:approximations} is dedicated to showing this fundamental fact. We will then be able to prove much better regularity properties of such currents compared to other currents with minimal singularities. This is especially useful in the context of the Minimal Model Program, as we will see in Theorem \ref{thm:main2}.

The paper \cite{BD12} studies supercanonical currents on a projective klt pair $(X,\Delta)$ such that $K_X+\Delta$ is big and proves several of its properties. In this paper we define supercanonical currents on any pseudoeffective line bundle, inspired by the definition in op.\ cit. The definition in Section \ref{sec:supercanonical} is somewhat more transparent than that in \cite{BD12}, and we simplify the construction by viewing it from a slightly different standpoint. This allows to give a streamlined and precise proof of the behaviour of supercanonical currents on big line bundles in Section \ref{sec:supercanbig}: one of the main new ingredients is a result on uniform bounds of norms of sections of adjoint line bundles given in Theorem \ref{thm:boundedsections}. Further explanations will be given at the beginning of Section \ref{sec:supercanbig}.

After the case of big line bundles is settled, the main problem is to analyse what happens when the line bundle $L$ is only pseudoeffective, and how supercanonical currents associated to divisors $L+\varepsilon A$ behave when $\varepsilon\downarrow0$, where $A$ is an ample divisor on $X$. Following a suggestion from \cite[Generalization 5.24]{BD12}, we show that the corresponding supercanonical currents of $L+\varepsilon A$ converge weakly to a supercanonical current of $L$, and deduce additional strong regularity properties.

Specialising to the context of the Minimal Model Program, the following is our second main result. (We use the following notation introduced in Section \ref{sec:supercanonical}: if $\theta$ is a smooth closed real $(1,1)$-form on a compact complex manifold $X$ whose cohomology class is pseudoeffective, then $T_{\theta,\can}$ denotes the supercanonical current associated to $\theta$.)

\begin{thmA}\label{thm:main2}
Let $(X,\Delta)$ be a projective klt pair such that $\Delta$ is a $\Q$-divisor and $K_X+\Delta$ is pseudoeffective. Let $\pi\colon Y\to X$ be a log resolution of $(X,\Delta)$ and write 
$$K_Y+\Delta_Y\sim_\Q\pi^*(K_X+\Delta)+E,$$
where $\Delta_Y$ and $E$ are effective $\Q$-divisors without common components. Let $A$ be an ample $\Q$-divisor on $Y$, and let $\alpha$ and $\omega$ be fixed smooth $(1,1)$-forms in the cohomology classes of $K_Y+\Delta_Y$ and $A$, respectively. Then:
\begin{enumerate}[\normalfont (a)]
\item for each $\varepsilon>0$ the supercanonical current $T_{\alpha+\varepsilon\omega,\can}$ depends only on the holomorphic global sections of multiples of $K_Y+\Delta_Y+\varepsilon A$,
\item the supercanonical currents $T_{\alpha+\varepsilon\omega,\can}$ converge weakly to the supercanonical current $T_{\alpha,\can}$ as $\varepsilon\downarrow0$.
\end{enumerate}
If additionally $K_X+\Delta$ is nef, then there exists a positive rational number $\delta$ such that:
\begin{enumerate}[\normalfont (a)]
\item[{\normalfont (c)}] the non-nef loci $\sB_-(K_Y+\Delta_Y+\varepsilon A)$ do not depend on $0\leq\varepsilon\leq\delta$, and they are equal to the non-ample loci $\sB_+(K_Y+\Delta_Y+\varepsilon A)$ for $0<\varepsilon\leq\delta$,
\item[{\normalfont (d)}] for each $0<\varepsilon\leq\delta$ the supercanonical current $T_{\alpha+\varepsilon\omega,\can}$ has continuous local potentials away from the non-nef locus $\sB_-(K_Y+\Delta_Y)$,
\item[{\normalfont (e)}] for any two rational numbers $\varepsilon_1,\varepsilon_2\in (0,\delta]$ and for any $t\in[0,1]$ the supercanonical current
$$tT_{\alpha+\varepsilon_1\omega,\can}+(1-t)T_{\alpha+\varepsilon_2\omega,\can}$$
is the current with minimal singularities in the cohomology class of the divisor $K_Y+\Delta_Y+\big(t\varepsilon_1+(1-t)\varepsilon_2\big)A$.
\end{enumerate}
\end{thmA}

Parts (a) and (b) of the theorem are very delicate and they hold more generally for pseudoeffective divisors which are not necessarily adjoint, see Theorem \ref{thm:supercanpsef} for a much more precise statement. Part (d) holds also in that more general context, albeit with a weaker estimate of the size of the regularity locus. The other statements rely crucially on the fact that we are working with adjoint divisors, and they depend on the Minimal Model Program, see Theorem \ref{thm:localPL}.

The main aim of Theorem \ref{thm:main2} is to gain precise information on the behaviour of multiplier ideals associated to supercanonical currents under perturbations by an ample divisor, in order to combine it with Theorem \ref{thm:main1} to obtain the proof of the Abundance conjecture for minimal projective klt pairs $(X,\Delta)$ with $\chi(X,\OO_X)\neq0$. The more detailed  results from Sections \ref{sec:supercanbig} and \ref{sec:proofofMain2} indicate how this might be achieved, see Theorem \ref{thm:supercanbig}\eqref{enu:13} and Theorem \ref{thm:supercanpsef}\eqref{enu:d}.

The algebraicity statements (a) and (b) as well as the regularity statement (d) of Theorem \ref{thm:main2} are very strong and, combined with Theorem \ref{thm:main1}, we expect them to be crucial for the completion of the proof of the Abundance conjecture for minimal projective klt pairs $(X,\Delta)$ with $\chi(X,\OO_X)\neq0$. 

\smallskip

\paragraph{\textbf{On the organisation of the paper.}}

This work contains as many ingredients from complex birational geometry as it does from pluripotential theory. I have attempted to make it accessible to both birational and complex geometers. This possibly resulted in the inclusion of proofs of some results which might be considered standard or classical by some readers.

\medskip

{\footnotesize
\paragraph{\textbf{Acknowledgements.}}
It is my great pleasure to dedicate this paper to Thomas Peternell. This present work builds on our joint quest towards abundance-related problems that we started almost a decade ago. He has provided constant support and has been a source of of wonderful mathematical ideas. This paper would not have been possible without him.

I had several conversations online about a very preliminary version of the ideas presented here with Jean-Pierre Demailly in the summer of 2021, partly together with Thomas Peternell. Jean-Pierre clarified several things about his Bergman kernel techniques and the paper \cite{BD12}. These conversations have had a very big impact on this paper.

I am very grateful to Nikolaos Tsakanikas for discussions about the content of Section \ref{sec:localPL} and for extensive comments which improved the presentation of the paper. I am grateful to Vincent Guedj and Zhixin Xie for useful comments and suggestions.

I gratefully acknowledge support by the Deutsche Forschungsgemeinschaft (DFG, German Research Foundation) – Project-ID 286237555 – TRR 195 and Project-ID 530132094.}

\newpage

\part{Preliminaries}

\section{Preliminaries: pluripotential theory}

Much of the material discussed here can be found in \cite{Dem12,Kli91,GZ17} or in the introductory sections of \cite{Bou02,Bou04}. The notes \cite{BT06,Vu21} present many of the foundational results with more details or clarity. The presentation in \cite{GZ05} is exceptionally clear.

We collect some definitions and results for the benefit of the reader and to settle the notation and terminology. In this paper we use the convention that $d^c=\frac{1}{2\pi i}(\partial-\dbar)$, so that $dd^c=\frac{i}{\pi}\partial\dbar$. We denote by 
$$B(x,r)=\{x\in\C^n\mid \|x\|<r\}\quad\text{and}\quad S(x,r)=\{x\in\C^n\mid \|x\|=r\}$$
the open ball and the sphere of radius $r$ and with centre $x$ in $\C^n$. All manifolds in the paper are connected. The notation $\fint$ is used for the averaged integral, i.e.\ for the integral divided by the volume of the set over which the integration is made.

\subsection{Bott-Chern cohomology}

If $X$ is a complex manifold, we define the Bott-Chern $(1,1)$-cohomology space $H^{1,1}_\mathrm{BC}(X,\C)$ as the quotient of the space of $d$-closed smooth $(1,1)$-forms modulo the $dd^c$-exact smooth $(1,1)$-forms, and we denote by $H^{1,1}_\mathrm{BC}(X,\R)$ the space of its real points. It can be shown by a partition of unity argument that $H^{1,1}_\mathrm{BC}(X,\C)$ is isomorphic to the quotient of the space of $d$-closed $(1,1)$-currents modulo the $dd^c$-exact $(1,1)$-currents. If additionally $X$ is compact and K\"ahler, then $H^{1,1}_\mathrm{BC}(X,\C)$ is isomorphic to the Dolbeault cohomology group $H^{1,1}(X,\C)$.

If $T$ is a closed $(1,1)$-current on a complex manifold $X$, we denote by $\{T\}$ its class in $H^{1,1}_\mathrm{BC}(X,\C)$. If $T$ is a real closed $(1,1)$-current, then $\{T\}\in H^{1,1}_\mathrm{BC}(X,\R)$, and the representatives of $\{T\}$ are the closed currents of the form $T+dd^c\varphi$, where $\varphi$ is a real current of degree $0$. If $T$ is a representative of a class $\alpha\in H^{1,1}_\mathrm{BC}(X,\C)$, we write $T\in\alpha$; if $T'\in\alpha$ is another representative, we also write $T\equiv T'$.

\subsection{Almost positive currents}\label{subsection:almostpositive}

Let $X$ be a complex manifold of dimension $n$. A continuous $(n-1,n-1)$-form $\varphi$ on $X$ is \emph{positive} if it can be written locally as a finite non-negative linear combination of forms of type 
$$(i\alpha_1\wedge \overline\alpha_1)\wedge \ldots\wedge(i\alpha_{n-1}\wedge \overline\alpha_{n-1}),$$
where $\alpha_i$ are $(1,0)$-forms. Positivity of forms is a pointwise property which does not depend on local coordinates. 

A $(1,1)$-current $T$ on $X$ is \emph{positive} if $T(\varphi)$ is a positive measure for every smooth positive $(n-1,n-1)$-form
$\varphi$, and we write $T\geq0$. A positive $(1,1)$-current is always real. If $T$ and $T'$ are two $(1,1)$-currents on $X$, we write $T\geq T'$ if $T-T'\geq0$. If $\varphi=i\sum h_{jk}dz_j\wedge d\overline{z_k}$ is a real continuous $(1,1)$-form, then $\varphi$ is positive if and only if $\big(h_{jk}(x)\big)$ is a positive semidefinite hermitian matrix for all $x\in X$­.

If now $D$ is an irreducible analytic subset of pure codimension $1$ in $X$, then we denote by $[D]$ the \emph{current of integration} on the regular part of $D$: this is a closed positive $(1,1)$-current. If we have an effective $\R$-divisor $G=\delta_1G_1+\dots+\delta_rG_r$ on $X$, then we call the closed positive $(1,1)$-current $[G]:=\delta_1[G_1]+\dots+\delta_r[G_r]$ the \emph{current of integration on $G$}. If there is no danger of confusion, we drop the brackets and write simply $G$ for the current of integration on $G$.

A real $(1,1)$-current $T$ on $X$ is \emph{almost positive} if $T\geq\gamma$ for some real continuous $(1,1)$-form $\gamma$ on $X$.

If $f\colon Y\to X$ is a surjective holomorphic map between complex manifolds and if $T$ is a closed almost positive $(1,1)$-current on $X$, then one can easily define its \emph{pullback} $f^*T$ to $Y$ such that $\{f^*T\}=f^*\{T\}$, see \cite[2.2.3]{Bou04}.

We will need the following easy result.

\begin{lem}\label{lem:diagonalisation}
Let $X$ be a compact complex manifold of dimension $n$ and let $\omega$ be a smooth positive definite $(1,1)$-form on $X$. Let $\{\theta_j\}_{j\in J}$ be a collection of real $(1,1)$-forms on $X$ whose coefficients are locally uniformly bounded on $X$. Then there exists a constant $C>0$ such that $\theta_j+C\omega\geq0$ for all $j$.
\end{lem}

\begin{proof}
Fix a point $x\in X$ and a coordinate neighbourhood $U_x$ centred at $x$ such that the coefficients of all $\theta_j$ are uniformly bounded on $U_x$. Then we may find finitely many real smooth $(1,1)$-forms $\widetilde\theta_1,\dots,\widetilde\theta_r$ on $U_x$ such that each $\theta_j$ is a convex linear combination of the forms $\widetilde\theta_k$. By the spectral theorem for hermitian operators, for each $k$ there exists a linear change of local coordinates $f_{k,x}\colon U_x\to U_x$ such that
$$f_{k,x}^*\widetilde\theta_k=\frac{i}{2}\sum_{\ell=1}^n\lambda_{k,\ell,x}dz_\ell\wedge d\overline{z_\ell}\quad\text{and}\quad f_{k,x}^*\omega=\frac{i}{2}\sum_{\ell=1}^n dz_\ell\wedge d\overline{z_\ell}\quad\text{at }x.$$
Then it is clear that, by possibly shrinking $U_x$, there exists a constant $C_x$ on $U_x$ such that $f_{k,x}^*\widetilde\theta_k+C_xf_{k,x}^*\omega\geq0$ on $U_x$ for all $1\leq k\leq r$. Therefore, $\widetilde\theta_k+C_x\omega\geq0$ on $U_x$ for all $1\leq k\leq r$ by \cite[Proposition III.1.17]{Dem12}, hence $\theta_j+C_x\omega\geq0$ on $U_x$ for all $j\in J$. We conclude by the compactness of $X$.
\end{proof}

\subsection{Plurisubharmonic functions}

In this subsection $X$ is a complex manifold of dimension $n$. A function $\varphi \colon X \to [-\infty, +\infty)$ is \emph{plurisubharmonic} or \emph{psh} if it is upper semicontinuous, locally integrable, and satisfies the mean value inequality
$$ f^*\varphi (0) \leq \fint_\Delta f^*\varphi\, dV_\Delta $$
for any holomorphic mapping $f\colon \Delta \to X$ from the open unit disk $\Delta \subseteq \C$. Every plurisubharmonic function is \emph{subharmonic}, i.e.\ it satisfies the mean value inequality
$$ f^*\varphi(0) \leq \fint_B f^*\varphi\, dV_B $$
for any open embedding $f\colon B \to X$ of the open unit ball $B\subseteq \C^n$.

Psh functions on $X$ are locally bounded from above and belong to $L^p_\loc(X)$ for any $1\leq p<\infty$. If additionally $X$ is compact, then any psh function on $X$ is constant.

A subset $A$ of $X$ is \emph{locally pluripolar} if it is locally contained in the pole set $\{u=-\infty\}$ of a psh function $u$. Since each psh function is locally integrable, the set $A$ is of Lebesgue measure zero and the complement $X\setminus A$ is dense in $X$.

A closed $(1,1)$-current $T$ on $X$ is positive if and only if for each $x\in X$ there exists an open subset $x\in U\subseteq X$ such that $T$ can be locally written as $T=dd^c \varphi$ for a psh function $\varphi$ on $U$. The function $\varphi$ is a \emph{local potential} of $T$ on $U$.

We will often need the following well-known properties of subharmonic functions.

\begin{lem}\label{lem:strongusc}
Let $\Omega\subseteq\C^n$ be a domain.
\begin{enumerate}[\normalfont (a)]
\item Let $\varphi$ be a subharmonic function on $\Omega$ and let $A\subseteq\Omega$ be a set of Lebesgue measure zero. Then
$$\limsup\limits_{z'\to z,\,z'\in\Omega\setminus A}\varphi(z')=\varphi(z)\quad\text{for every }z\in\Omega.$$
\item Let $\varphi$ be a subharmonic function on $\Omega$. Then
$$\displaystyle \varphi(z)=\lim\limits_{r\to0}\fint_{B(z,r)}\varphi(z)dV\quad\text{for every }z\in\Omega.$$
\item Let $\varphi$ and $\psi$ be subharmonic functions on $\Omega$ and assume that $\varphi\leq\psi$ almost everywhere. Then $\varphi\leq\psi$.
\end{enumerate}
\end{lem}

\begin{proof}
Part (c) follows immediately from (a). We will show (a) and (b) simultaneously. For a fixed $z\in\Omega$, the mean value inequality on balls $B(z,r)\subseteq\Omega$ and the upper semicontinuity of $\varphi$ give
\begin{align*}
\varphi(z)&\leq\lim_{r\to0}\fint_{B(z,r)}\varphi(z)dV=\lim_{r\to0}\fint_{B(z,r)\setminus A}\varphi(z)dV\\
&\leq\lim_{r\to0}\sup_{B(z,r)\setminus A}\varphi(z')=\limsup_{z'\to z,\,z'\in\Omega\setminus A}\varphi(z')\leq\varphi(z).
\end{align*}
This finishes the proof.
\end{proof}

\subsection{Quasi-psh functions}\label{subsec:quasi-psh}

As mentioned above, a psh function on a compact complex manifold is always constant. A more suitable notion on compact complex manifolds is that of \emph{quasi-plurisubharmonic} or \emph{quasi-psh} functions: a function $\varphi\colon X\to[-\infty,+\infty)$ on a complex manifold $X$ is quasi-psh if it is locally equal to the sum of a psh function and a smooth function. Equivalently, $\varphi$ is quasi-psh if it is locally integrable and upper semicontinuous, and there exists a smooth closed real $(1,1)$-form $\theta$ on $X$ such that $\theta+dd^c\varphi\geq 0$ in the sense of currents. A good introduction to quasi-psh functions is in \cite{GZ05}.

Now, if $\theta$ is a real continuous $(1,1)$-form on $X$ and if $\varphi$ is a quasi-psh function on $X$ such that $\theta+dd^c\varphi\geq 0$ in the sense of currents, then we say that $\varphi$ is \emph{$\theta$-psh} and we denote the set of all $\theta$-psh functions by $\PSH(X,\theta)$. The weak topology on the set $\{\theta+dd^c\varphi\mid\varphi\in\PSH(X,\theta)\}$ corresponds to the $L^1_\loc(X)$-topology on $\PSH(X,\theta)$. The set 
$$\big\{\varphi\in\PSH(X,\theta)\mid \sup\nolimits_X\varphi=0\big\}$$
is compact in this topology, as we will see in Theorem \ref{thm:compactnessquasipsh}. 

If $X$ is additionally compact, if $\theta$ is a smooth closed real $(1,1)$-form on $X$ and if $T$ is a closed almost positive $(1,1)$-current in $\{\theta\}$, then there exists a closed real smooth $(1,1)$-form $\gamma$ such that $\gamma+T\geq0$, and clearly $\gamma+T\in\{\gamma+\theta\}$. By the $dd^c$-lemma there exists $\varphi\in\PSH(X,\gamma+\theta)$, which is unique up to an additive constant,  such that $\gamma+T=(\gamma+\theta)+dd^c\varphi$, hence
$$T=\theta+dd^c\varphi.$$
By adopting the terminology from \cite{BEGZ10}, such a function $\varphi$ is called a \emph{global potential} of $T$; global potentials depend, up to an additive constant, on the choice of $\theta$, but not of $\gamma$. 

A subset of $X$ is \emph{pluripolar} if it is contained in the pole set $\{\varphi=-\infty\}$ of a quasi-psh function $\varphi$ on $X$.

We will need the following easy consequence of Lemma \ref{lem:strongusc}(c).

\begin{cor}\label{cor:strongusc}
Let $X$ be a complex manifold and let $\varphi$ and $\psi$ be quasi-psh functions on $X$. If $\varphi\leq\psi$ holds almost everywhere, then $\varphi\leq\psi$.
\end{cor}

\begin{proof}
Let $\theta_1$ and $\theta_2$ be real smooth $(1,1)$-forms on $X$ such that $\varphi_1\in\PSH(X,\theta_1)$ and $\varphi_2\in\PSH(X,\theta_2)$. Fix a point $x\in X$. As in the proof of Theorem \ref{thm:compactnessquasipsh}(a) below, there exist an open neighbourhood $U$ of $x$ and a smooth closed form $\omega$ on $U$ such that $\omega\geq\theta_1$ and $\omega\geq\theta_2$ on $U$. Hence, $\varphi_1$ and $\varphi_2$ are $\omega$-psh on $U$. If $\xi$ is a local potential of $\omega$ on $U$, then $\xi+\varphi_1$ and $\xi+\varphi_2$ are psh and $\xi+\varphi_1\leq\xi+\varphi_2$ almost everywhere on $U$. We conclude by Lemma \ref{lem:strongusc}(c).
\end{proof}

\subsection{Upper semicontinuous regularisation}\label{subsec:upper}

Let $\Omega\subseteq\C^n$ be an open subset and let $u$ be a function on $\Omega$. We define its \emph{upper semicontinuous regularisation} as
$$u^*(z)=\lim_{\varepsilon\to0}\sup_{B(z,\varepsilon)}u\quad\text{for }z\in\Omega.$$
Then $u^*$ is the smallest upper semicontinuous function which is $\geq u$. This notion extends easily to quasi-psh functions on a complex manifold.

Consider a family $\{u_\alpha\}$ of psh functions on $\Omega$ which is locally uniformly bounded from above, and
set $u:=\sup_\alpha u_\alpha$. Then the function $u^*$ is psh and we have $u^*=u$ almost everywhere, see \cite[Theorem I.5.7]{Dem12}. We will need very often the following extension of this and other important compactness results to quasi-psh functions.

\begin{thm}\label{thm:compactnessquasipsh}
Let $X$ be a complex manifold and let $\theta$ be a continuous real $(1,1)$-form on $X$.
\begin{enumerate}[\normalfont (a)]
\item  Consider a family $\{\varphi_\alpha\}$ of $\theta$-psh functions on $X$ which is locally uniformly bounded from above, and
set $\varphi:=\sup_\alpha \varphi_\alpha$. Then $\varphi^*\in\PSH(X,\theta)$ and $\varphi^*=\varphi$ almost everywhere.
\item Assume additionally that $X$ is compact. Let $\{\varphi_j\}$ be a sequence of $\theta$-psh functions on $X$ which are uniformly bounded from above. Then either the sequence $\{\varphi_j\}$ converges uniformly to ${-}\infty$ or it has a subsequence which converges in $L^1_\loc(X)$ and almost everywhere to a function in $\PSH(X,\theta)$.
\item Let $\{\varphi_j\}$ be a sequence of $\theta$-psh functions on $X$ which are locally uniformly bounded from above. Then the function $\Big(\limsup\limits_{j\to\infty}\varphi_j\Big)^*$ is $\theta$-psh.
\item Let $\{\varphi_j\}$ be a sequence of $\theta$-psh functions on $X$ which are locally uniformly bounded from above. If the sequence is decreasing, then either it converges uniformly to ${-}\infty$ or it converges in $L^1_\loc(X)$ to the $\theta$-psh function $\lim\limits_{j\to\infty}\varphi_j$. If the sequence is increasing, then it converges in $L^1_\loc(X)$ and almost everywhere to the $\theta$-psh function $\Big(\lim\limits_{j\to\infty}\varphi_j\Big)^*$.
\item Assume additionally that $X$ is compact. Let $\theta'$ be a positive continuous $(1,1)$-form on $X$, and consider a sequence of real numbers $\varepsilon_j\downarrow0$. For each positive integer $j$, let $\varphi_j\in\PSH(X,\theta+\varepsilon_j\theta')$ and assume that all $\varphi_j$ are uniformly bounded from above. Then either the sequence $\{\varphi_j\}$ converges uniformly to ${-}\infty$ or it has a subsequence which converges in $L^1_\loc(X)$ and almost everywhere to a function in $\PSH(X,\theta)$.
\end{enumerate}
\end{thm}

\begin{proof}
For (a) we extract the proof from \cite[Proposition 2.1.3]{Bou02}. Fix a point $x\in X$. By the spectral theorem for hermitian operators and by the continuity of $\theta$, for each $\varepsilon>0$ there exists real numbers $\lambda_j$ and a neighbourhood $U_\varepsilon$ of $x$ with local coordinates $z=(z_1,\dots,z_n)$ such that, if we set $q(z):=\sum \lambda_j|z_j|^2$, we have
$$dd^c\big(q(z)-\varepsilon|z|^2\big)\leq\theta\leq dd^c\big(q(z)+\varepsilon|z|^2\big).$$
Then on $U_\varepsilon$ each function $q(z)+\varepsilon|z|^2+\varphi_\alpha(z)$ is psh, hence so is the function 
$$\big(\sup\nolimits_\alpha\{q+\varepsilon|z|^2+\varphi_\alpha\}\big)^*=q+\varepsilon|z|^2+\varphi^*$$
by the paragraph before the theorem. Therefore, 
$$dd^c\varphi^*+\theta+2\varepsilon dd^c|z|^2\geq dd^c\varphi^*+dd^cq(z)+\varepsilon dd^c|z|^2\geq0$$
on $U_\varepsilon$. Letting $\varepsilon\to0$ we obtain $dd^c\varphi^*+\theta\geq0$ at $x$. Since $x$ was arbitrary, this shows that $\varphi$ is $\theta$-psh.

The proof of (b) is similar, using the local result for psh functions \cite[Proposition I.5.9]{Dem12}; see \cite[Proposition 2.5.7]{Vu21} for details.

Part (c) follows similarly as (a) from the local result for psh functions \cite[Proposition 2.9.17]{Kli91}.

Part (d) follows from (b) and (c).

Now we prove (e). Assume that the sequence $\{\varphi_j\}$ does not converge uniformly to ${-}\infty$. For positive integers $k\geq k'$ we have $\theta+\varepsilon_{k'}\theta'\geq\theta+\varepsilon_k\theta'$, and hence 
$$\varphi_k\in\PSH(X,\theta+\varepsilon_{k'}\theta').$$
Therefore, by (b) there exists a subsequence $\{\varphi_{j_1}\}$ of $\{\varphi_j\}$ which converges in $L^1_\loc(X)$ and almost everywhere to a function
$$\varphi\in\PSH(X,\theta+\varepsilon_1\theta').$$
We will be done if we show that $\varphi\in\PSH(X,\theta)$, and for this it suffices to prove that $\varphi\in\PSH(X,\theta+\varepsilon_i\theta')$ for all $i$, since $\theta$ is the weak limit of the sequence $\{\theta+\varepsilon_i\theta'\}$. To this end, by (b) we inductively have that, for all $i\geq2$, there exists a subsequence $\{\varphi_{j_i}\}$ of $\{\varphi_{j_{i-1}}\}$ which converges in $L^1_\loc(X)$ and almost everywhere to a function $\eta_i\in\PSH(X,\theta+\varepsilon_i\theta')$, thus $\eta_i=\varphi$ almost everywhere. But then $\eta_i=\varphi$ by Corollary \ref{cor:strongusc}, and in particular, $\varphi\in\PSH(X,\theta+\varepsilon_i\theta')$. This finishes the proof.
\end{proof}

\subsection{Positivity of classes}

Let $X$ be compact complex manifold, let $\omega$ be a fixed smooth positive $(1,1)$-form on $X$, and consider a cohomology class $\alpha\in H^{1,1}_\mathrm{BC}(X,\R)$. Then $\alpha$ is:
\begin{enumerate}[\normalfont (i)]
\item \emph{pseudoeffective} if there exists a closed positive $(1,1)$-current $T\in\alpha$;
\item \emph{nef} if for each $\varepsilon>0$ there exists a smooth form $\theta_\varepsilon\in\alpha$ such that $\theta_\varepsilon\geq{-}\varepsilon\omega$;
\item \emph{big} if there exist $\varepsilon>0$ and a closed $(1,1)$-current $T\in\alpha$ such that $T\geq\varepsilon\omega$.
\end{enumerate}
These definitions do not depend on the choice of $\omega$, and they correspond to the usual notions from algebraic geometry when $X$ is projective and $\alpha$ is an algebraic class.

\subsection{Lelong numbers}

Let $\Omega\subseteq\C^n$ be an open subset and let $\varphi$ be a psh function on $\Omega$. The \emph{Lelong number} of $\varphi$ at a point $x\in\Omega$ is
$$ \nu(\varphi,x):=\lim_{r\to0}\frac{\sup_{B(x,r)}\varphi}{\log r}; $$
this is equivalent to other definitions in the literature by \cite[Example III.6.9]{Dem12}. Thus, if $\nu(\varphi,x)>0$, then $\varphi(x)={-}\infty$, but the converse does not always hold. The Lelong number $\nu(\varphi,x)$ does not depend on the choice of local coordinates around $x$. For psh functions $u$ and $v$ on $\Omega$ and for each point $x\in \Omega$ we have
$$ \nu(u+v,x)=\nu(u,x)+\nu(v,x); $$
note that $u+v$ is a psh function by Example \ref{exa:quasi-psh}(d) below.

Let now $T$ be a closed positive $(1,1)$-current on a complex manifold $X$. Then locally at a point $x\in X$ we can write $T=dd^c\varphi$ for a psh function $\varphi$, and we define the Lelong number of $T$ at $x$ as
$$\nu(T,x):=\nu(\varphi,x);$$
this does not depend on the choice of $\varphi$. If $Y$ is an analytic subset of $X$ and if $x\in X$, then a result of Thie states that $\nu(Y,x)$ is equal to the multiplicity of $Y$ at $x$.

We will need the following result which compares the Lelong numbers under pullbacks \cite[Corollary 4]{Fav99}, see also \cite[Th\'eor\`eme 5.1]{Kis00}.

\begin{thm}\label{thm:Favre}
Let $f\colon Y\to X$ be a surjective holomorphic map between compact complex manifolds. Then there exists a constant $C>0$ such that for every closed positive $(1,1)$-current $T$ on $X$ and for all points $y\in Y$ and $x:=f(y)\in X$ we have
$$\nu(T,x)\leq\nu(f^*T,y)\leq C\nu(T,x).$$
\end{thm}

If $T$ is a closed almost positive $(1,1)$-current on $X$ and if $x$ is a point in $X$ with local coordinates $z=(z_1,\dots,z_n)$ around $x$, then there exists a positive constant $C$ such that $T+Cdd^c|z|^2\geq0$ locally around $x$. Then we define the Lelong number $\nu(T,x)$ as $\nu(T+Cdd^c|z|^2,x)$; this does not depend on the choice of $C$.

For $c\geq0$ define the \emph{Lelong upperlevel sets} as 
$$E_c(T):=\{x\in X\mid \nu(T,x)\geq c\}.$$
Then a fundamental theorem of \cite{Siu74} states that for each $c>0$, the set $E_c(T)$ is a proper analytic subset of $X$. Thus, for any analytic subset $Y$ of $X$ we may define the \emph{generic Lelong number of $T$ along $Y$} as 
$$\nu(T,Y):=\inf\{\nu(T,x)\mid x\in Y\},$$
which is equal to $\nu(T,x)$ for a very general point $x\in Y$.

\subsection{Divisorial valuations}

Let $X$ be compact complex manifold. Following \cite[B.5 and B.6]{BBJ21}, a \emph{prime divisor over $X$} denotes a prime divisor $E\subseteq X'$, where $\mu\colon X'\to X$ is a resolution. We say that two prime divisors $E_1\subseteq X_1$ and $E_2\subseteq X_2$ over $X$ are \emph{equivalent} if there exists a common resolution $X'$ of $X_1$ and $X_2$ such that the strict transforms of $E_1$ and $E_2$ on $X'$ coincide. When $X$ is projective, a prime divisor over $X$ is the same thing as a geometric divisorial valuation on $X$ by \cite[Lemma 2.45]{KM98}.

Let $T$ be a closed positive $(1,1)$-current on $X$. If $E$ is a prime divisor on a resolution $f\colon Y\to X$, we denote
$$\nu(T,E):=\nu(f^*T,E).$$
If $E'$ is another prime divisor over $X$ equivalent to $E$, then $\nu(T,E)=\nu(T,E')$. If $D$ is an $\R$-divisor on $X$, then we define the \emph{multiplicity of $D$ along $E$} by
$$\mult_E D:=\mult_E f^*D.$$

\subsection{Siu decomposition}\label{subsec:Siudecomposition}
If $X$ is a complex manifold and if $T$ is a closed positive $(1,1)$-current on $X$, then there exist at most countably many codimension $1$ irreducible analytic subsets $D_k$ such that $T$ has the \emph{Siu decomposition}
$$T=R+\sum\nu(T,D_k)\cdot D_k,$$
where $R$ is a closed positive $(1,1)$-current such that $\codim_X E_c(R)\geq2$ for each $c>0$. In this paper we call $\sum\nu(T,D_k)\cdot D_k$ the \emph{divisorial part} and $R$ the \emph{residual part} of (the Siu decomposition of) $T$.

Now assume that $T$ is closed almost positive $(1,1)$-current, and let $\gamma$ be a continuous form on $X$ such that $T\geq\gamma$. Then one can construct the Siu decomposition $T=\sum\nu(T,D_k)\cdot D_k+R$ of $T$ similarly as above, where now $R$ is a closed almost positive $(1,1)$-current satisfying $R\geq\gamma$.

With notation as above, if $\pi\colon Y\to X$ is a resolution and if 
$$\pi^*T=R_Y+\sum\nu(\pi^*T,D'_\ell)\cdot D'_\ell$$
is the Siu decomposition of $\pi^*T$, then it is clear that each $D'_\ell$ is a component of $\pi^*D_k$ for some $k$, or it is a $\pi$-exceptional divisor. In particular, if the divisorial part of $T$ is an $\R$-divisor, then so is the divisorial part of $\pi^*T$.

\subsection{Singular metrics}

Let $L$ be a holomorphic line bundle on a complex manifold $X$. A \emph{singular hermitian metric} or simply a \emph{singular metric} $h$ on $L$ is a metric given in every trivialisation $\theta\colon L|_{\Omega}\to \Omega \times \C$ by 
$$ h(\xi,\xi):=|\theta (\xi)|^2e^{-2\varphi(x)}\quad\text{for }x\in \Omega,\ \xi \in L_x \text{ and } \varphi\in L^1_{\loc}(\Omega).$$
We also denote $|\cdot|_h:=h(\cdot\,,\cdot)^{1/2}$. The function $\varphi$ is called the \emph{local weight} of $h$ with respect to the trivialisation $\theta$. The \emph{curvature current} $\Theta_h(L):=dd^c\varphi$ is globally defined and lies in $\{L\}\in H^{1,1}_\mathrm{BC}(X,\R)$. The curvature current $\Theta_h(L)$ of $h$ is \emph{semipositive} if it is positive in the sense of currents. 

Now fix a smooth metric $h_\infty$ on $L$. Then there exists a locally integrable function $\varphi$ on $X$ such that $h=h_\infty e^{-2\varphi}$, and we call $\varphi$ the \emph{global weight} of $h$ with respect to the reference metric $h_\infty$. Then we have
$$\Theta_h(L)=\Theta_{h_\infty}(L)+dd^c\varphi.$$
Conversely, for any closed $(1,1)$-current $T\in\{L\}$ there exists a degree $0$ current $\varphi$ such that $T=\Theta_{h_\infty}(L)+dd^c\varphi$. When additionally $T$ is almost positive, then $\varphi\in L^1_{\loc}(X)$, hence every almost positive current $T\in\{L\}$ is the curvature current of a singular hermitian metric on $L$, and the global weight $\varphi$ is a quasi-psh function on $X$.

We now mention several examples of quasi-psh functions which are relevant for this paper.

\begin{exa}\label{exa:quasi-psh}\mbox{}
\begin{enumerate}[\normalfont (a)]
\item Let $\Omega\subseteq\C^n$ be an open subset and let $f_1,\dots,f_m$ be holomorphic functions on $\Omega$. Then $\log\big(|f_1|^2+\dots+|f_m|^2\big)$ is psh on $\Omega$ by \cite[Example I.5.12]{Dem12}.
\item More generally, let $L$ be a holomorphic line bundle with a continuous metric $h$ on a complex manifold $X$, and consider global holomorphic sections $\sigma_1,\dots,\sigma_m$ of $L$. Then the function $\varphi\colon X\to[-\infty,+\infty)$ given by $$\varphi:=\frac{1}{2}\log\big(|\sigma_1|^2_h+\dots+|\sigma_m|^2_h\big)$$
is quasi-psh on $X$: indeed, let $\theta$ be a local continuous weight of $h$ on some trivialisation of $L$. Then locally and by (a) we have
$$dd^c(\theta+\varphi)=\frac{1}{2}dd^c\log\big(|\sigma_1|^2+\dots+|\sigma_m|^2\big)\geq0,$$
hence the function $\theta+\varphi$ is psh. In particular, globally we have $\Theta_h(L)+dd^c\varphi\geq0$, hence the curvature current of the singular metric $he^{-2\varphi}$ on $L$ is semipositive. The metric $he^{-2\varphi}$ and the current $\Theta_h(L)+dd^c\varphi$ do not depend on the choice of $h$.
\item In the context of (b), let $\sigma$ be a global holomorphic section of $L$ and let $\sum m_iD_i$ be the zero-divisor of $f$. Then we have the \emph{global Lelong-Poincar\'e equation}
$$\Theta_h(L)+dd^c\log|\sigma|_h = \sum m_iD_i,$$
understood in the sense of currents, see \cite[Theorem III.2.15 and Section V.13]{Dem12}.
\item Let $\Omega\subseteq\C^n$ be an open subset, let $u_1,\dots,u_r$ be psh functions on $\Omega$ and let $\chi\colon[-\infty,+\infty)^r\to[-\infty,+\infty)$ be a convex function which is non-decreasing in each coordinate. Then by \cite[Theorem I.5.6]{Dem12}, the function $\chi(u_1,\dots,u_r)$ is psh on $\Omega$. In particular, the functions $u_1+\dots+u_r$, $\max\{u_1,\dots,u_r\}$ and $e^{u_1}+\dots+e^{u_r}$ are psh on $\Omega$.
\item Let $X$ be a complex manifold, let $\theta$ be a continuous real $(1,1)$-form on $X$, and let $\{\varphi_j\}$ be a sequence of $\theta$-psh functions on $X$ which are locally uniformly bounded from above. If $\sum\varepsilon_j$ is a convergent series of positive real numbers, then the function $\sum\varepsilon_j\varphi_j$ is $\theta$-psh. Indeed, there exists a constant $C$ such that $\varphi_j':=\varphi_j-C\leq0$ for all $j$. Then each partial sum $\Phi_k:=\sum_{j\leq k}\varepsilon_j\varphi_j'$ is $\theta$-psh by (d), and the sequence $\{\Phi_k\}$ is decreasing, hence $\sum\varepsilon_j\varphi_j'=\lim\limits_{k\to\infty}\Phi_k\in\PSH(X,\theta)$ by Theorem \ref{thm:compactnessquasipsh}(d); note that the limit is not ${-}\infty$ since the union of pluripolar sets of all $\varphi_j$ is of Lebesgue measure zero in $X$. Since $\sum\varepsilon_j\varphi_j=\sum\varepsilon_j\varphi_j'+C\sum\varepsilon_j$, we obtain the claim.
\end{enumerate}
\end{exa}

We will need later the following remark, which we extract from the proof of \cite[Theorem 1.1]{Vu19}.

\begin{rem}\label{rem:pluripolarcountable}
Let $X$ be a compact complex manifold, let $\theta$ be a continuous real $(1,1)$-form on $X$, and let $\{\varphi_j\}$ be a sequence of $\theta$-psh functions on $X$. Then the union of pluripolar loci of all $\varphi_j$ is again a pluripolar set. In order to see this, first note that by subtracting a constant from each $\varphi_j$, we may assume that $\varphi_j\leq0$ for all $j$. Then by Example \ref{exa:quasi-psh}(e) the function $\varphi:=\sum j^{-2}\varphi_j$ is $\theta$-psh, and the union of pluripolar loci of all $\varphi_j$ is contained in the set $\{\varphi={-}\infty\}$.
\end{rem}

\subsection{Multiplier ideals}\label{subsec:multiplier}

If $\varphi$ is a quasi-psh function on a complex manifold $X$, the \emph{multiplier ideal sheaf} $\mathcal I(\varphi)\subseteq \OO_X$ is defined by 
$$ \mathcal I(\varphi)(U):=\{f\in \OO_X(U)\mid |f|e^{-\varphi}\in L^2_{\loc}(U)\} $$ 
for every open set $U\subseteq X$. Note that we set $|f(x)|e^{-\varphi(x)}=0$ at points $x \in X$ where $f(x) = 0$ and $\varphi(x) = {-}\infty$. The sheaf $\mathcal I(\varphi)$ is a coherent ideal sheaf on $X$. 

If now $h$ is a singular metric on a holomorphic line bundle $L$ on $X$ whose curvature current $\Theta_h(L)$ is almost positive, then its associated global weight $\varphi$ (with respect to some fixed smooth metric on $L$) is quasi-psh, and we define $\mathcal I(h):=\mathcal I(\varphi)$. This does not depend on the choice of the smooth metric on $L$.

Finally, if $T$ is a closed almost positive $(1,1)$-current on $X$, then any of its associated global potentials $\varphi$ (see \S\ref{subsec:quasi-psh}) is quasi-psh, and we define $\mathcal I(T):=\mathcal I(\varphi)$. This does not depend on the choice of $\varphi$. If $\nu(T,x)<1$ at a point $x\in X$, then $\mathcal I(T)_x=\OO_{X,x}$ by Skoda's lemma \cite[Lemma 5.6]{Dem01}.

The following is a fundamental result, proved first in \cite[3.3]{GZ15}; similar results with easier proofs are in \cite[Theorem 1.1]{Lem17} and \cite[Corollary 1.1 and Remark 1.4]{Hie14}.

\begin{thm}\label{thm:GuanZhou}
Let $\varphi$ and $\{\varphi_i\}_{i\in\N}$ be psh functions on an open set $U$ of a complex manifold $X$. Assume that $\varphi_i\leq\varphi$ for all $i$, and that the sequence $\{\varphi_i\}$ converges to $\varphi$ in $L^1_\loc(X)$. Then for every $U'\Subset U$ there exists a positive integer $i_0$ such that $\mathcal I(\varphi_i)|_{U'}=\mathcal I(\varphi)|_{U'}$ for all $i\geq i_0$.
\end{thm}

We will need the following important result, Theorem \ref{thm:DEL+DP}. In order to state it, we need a piece of notation: Assume that $T$ is a closed positive $(1,1)$-current on a complex manifold $X$, which can be written as a sum
$$T=\sum_{i=1}^\infty\lambda_i D_i,$$
where $\lambda_i\geq0$ for all $i$ and each $D_i$ is a prime divisor on $X$; in other words, the residual part of the Siu decomposition of $T$ is zero. Then $\lfloor T\rfloor$ denotes the closed positive $(1,1)$-current
$$\lfloor T\rfloor:=\sum_{i=1}^\infty\lfloor\lambda_i\rfloor D_i.$$

\begin{thm}\label{thm:DEL+DP}
Let $X$ be a complex manifold.
\begin{enumerate}[\normalfont (a)]
\item Let $T_1$ and $T_2$ be two closed almost positive $(1,1)$-currents on $X$. Then
$$\mathcal I(T_1+T_2)\subseteq\mathcal I(T_1)\cdot\mathcal I(T_2).$$

\item Let $T_1$ and $T_2$ be two closed almost positive $(1,1)$-currents on $X$. If for $x\in X$ we have $\nu(T_1,x)=0$, then
$$\mathcal I(T_1+T_2)_x=\mathcal I(T_2)_x.$$

\item If $G$ is an effective $\R$-divisor on $X$ with simple normal crossings support, then
$$ \mathcal I(G)=\OO_X({-}\lfloor G\rfloor). $$

\item If $G$ is a closed positive $(1,1)$-current on $X$ whose residual part is zero, then $\lfloor G\rfloor$ is a divisor on $X$ and
$$ \mathcal I(G)\subseteq\OO_X({-}\lfloor G\rfloor). $$

\item Let $T$ be a closed almost positive $(1,1)$-current on $X$ and let
$$ T = R+D $$
be its Siu decomposition, where $R$ is the residual part and $D$ is the divisorial part. 
Then $\lfloor D\rfloor$ is a divisor on $X$, we have
$$ \textstyle \mathcal I(T)\subseteq\OO_X(-\lfloor D\rfloor), $$
and this inclusion is an equality on a Zariski open subset $U$ with the property that $\codim_X(X\setminus U)\geq2$. 
\end{enumerate}
\end{thm}

\begin{rem}\label{rem:sheaves}
Here and elsewhere in the paper, if $D$ is an integral divisor on a complex manifold, then $\OO_X(D)$ denotes the subsheaf of the sheaf of meromorphic functions on $X$ whose divisor of zeroes and poles is precisely $D$. Thus, if $D$ is effective, then $\OO_X({-}D)$ is the sheaf of germs of holomorphic functions on $X$ which vanish along $D$; in particular, we have $\OO_X({-}D)\subseteq\OO_X$, and if $D'$ is another integral divisor on $X$, then  we have  $\OO_X({-}D)\subseteq\OO_X({-}D')$ if and only if $D'\leq D$.
\end{rem}

\begin{proof}[Proof of Theorem \ref{thm:DEL+DP}]
Part (a) is \cite[Theorem 2.6(ii)]{DEL00}, and part (b) is \cite[Lemma 2.14]{LP20b}. Part (c) is well known \cite[Remark 5.9]{Dem01}. 

For (d), first note that the closed positive $(1,1)$-current $G':=\lfloor G\rfloor$ is a divisor since the Lelong upperlevel set $E_1(G)$ is analytic. Then $G-G'$ is also a closed positive $(1,1)$-current on $X$, and by (a) we have
\begin{equation}\label{eq:inclusions}
\mathcal I(G)\subseteq\mathcal I(G-G')\cdot\mathcal I(G')\subseteq\mathcal I(G').
\end{equation}
Let $V$ be the maximal Zariski open subset of $X$ such that $G'|_V$ is a smooth divisor. Then $\codim_X(X\setminus V)\geq2$, and by (c) we have $\mathcal I(G')|_V=\OO_V({-}G')$. Since $\mathcal I(G')$ is torsion free and $\OO_X({-}G')$ is a line bundle, it follows that $\mathcal I(G')\subseteq\OO_X({-}G')$, which together with \eqref{eq:inclusions} implies (d).

Part (e) is \cite[Proposition 3.2]{DP03}; since the notation and context is slightly different, we provide the proof for the benefit of the reader. The current $D':=\lfloor D\rfloor$ is a divisor since $E_1(T)$ is an analytic subset of $X$. Then $T-D'$ is also a closed almost positive $(1,1)$-current on $X$, and by (a) and (d) we have
$$ \mathcal I(T)\subseteq\mathcal I(T-D')\cdot\mathcal I(D')\subseteq\mathcal I(D')\subseteq \OO_X({-}D'), $$
which gives the second claim in (e).

Now we show the last claim in (e). If $D_i$ are the components of $D$ and if $D_{i,\mathrm{sing}}$ is the singular locus of $D_i$ for each $i$, set
$$Z:=\bigcup_i D_{i,\mathrm{sing}}\cup\bigcup_{k,\ell}(D_k\cap D_\ell)\cup\bigcup_{c>0}E_c(R).$$
Then $Z$ is the union of at most countably many analytic subsets of $X$ of codimension at least $2$, and it suffices to show that
\begin{equation}\label{eq:atapoint}
\mathcal I(T)_x=\OO_X({-}D')_x\quad\text{for all }x\in X\setminus Z,
\end{equation}
since the locus in $X$ where the coherent sheaves $\mathcal I(T)$ and $\OO_X(-\lfloor D\rfloor)$ differ is an analytic subset of $X$. To that end, fix $x\in X\setminus Z$. Assume first that $x$ does not belong to any component of $D$. Then $\nu(T,x)=0$ by the definition of $Z$, hence by Skoda's lemma we have $\mathcal I(T)_x=\OO_{X,x}$, and clearly also $\OO_X(-\lfloor D\rfloor)_x=\OO_{X,x}$, which shows \eqref{eq:atapoint} in this first case. Finally, assume that $x$ belongs to a component $\Gamma$ of $D$ and set $R_1:=R+\big(D-\nu(T,\Gamma)\cdot\Gamma\big)$. Then by the definition of $Z$ we have $\nu(R_1,x)=0$, thus (b) and (c) yield
\begin{align*}
\mathcal I(T)_x&=\mathcal I\big(R_1+\nu(T,\Gamma)\cdot\Gamma\big)_x=\mathcal I\big(\nu(T,\Gamma)\cdot\Gamma\big)_x\\
&=\OO_X\big({-}\lfloor\nu(T,\Gamma)\rfloor\cdot\Gamma\big)_x=\OO_X({-}D')_x,
\end{align*}
which gives \eqref{eq:atapoint} also in this second case, and finishes the proof.
\end{proof}

We will also need the following consequence of the change of variables formula.

\begin{lem}\label{lem:Cao}
Let $X$ be a complex manifold and let $f\colon Y\to X$ be a resolution of $X$. Let $\varphi_1$ and $\varphi_2$ be two quasi-psh functions on $X$ such that $\mathcal I(\varphi_1)\subseteq \mathcal I(\varphi_2)$. If $A:=K_Y-f^*K_X$, then 
$$\mathcal I(f^*\varphi_1)\otimes\OO_X(-A)\subseteq \mathcal I(f^*\varphi_2).$$
\end{lem}

\begin{proof}
This is \cite[Lemma 2.2]{Cao14}; note that there is a typo in that statement: the divisor $E$ in op.\ cit.\ should be defined as $E=K_{\widetilde{X}}-\pi^*K_X$.
\end{proof}

\subsection{Currents with analytic singularities}

A closed almost positive $(1,1)$-current $T$ on a compact complex manifold $X$, and any of its global potentials $\varphi$, are said to have \emph{analytic singularities} if there exist a coherent ideal sheaf $\mathcal I$ and a constant $c>0$ such that, locally on $X$, we have
$$ \varphi=c\log(|f_1|^2+\dots+|f_k|^2)+u, $$
where $u$ is smooth and $f_1,\dots,f_k$ are local generators of $\mathcal I $. The current $T$ is smooth outside of the co-support of $\mathcal I$. Now, if $\pi\colon Y\to X$ is a resolution of $X$ which factors through the blowup of the scheme $V(\mathcal I )$, there exists an effective divisor $D$ on $Y$ such that $\pi^{-1}\mathcal I = \OO_Y({-}D)$, and the Siu decomposition of $\pi^*T$ has the form
$$\pi^*T=\theta+cD,$$
where $\theta$ is a smooth $(1,1)$-form. If $T\geq\gamma$ for some smooth form $\gamma$, then $\theta\geq\pi^*\gamma$.

\subsection{Currents with generalised analytic singularities}\label{subsec:generalisedanalytic}

We need a generalisation of the concept of currents with analytic singularities introduced in \cite{LP20b}. A closed almost positive $(1,1)$-current $T$ on a compact complex manifold $X$, and any of its global potentials $\varphi$, are said to have \emph{generalised analytic singularities} if there exists a resolution $\pi\colon Y \to X$ such that the Siu decomposition of $\pi^*T$ has the form
$$\pi^*T=\Theta+D,$$
where $\Theta$ is a closed almost positive $(1,1)$-current whose all Lelong numbers are zero and $D$ is an effective $\R$-divisor on $Y$. In that case we say that the current $T$ \emph{descends} to $Y$. If $D$ is additionally a $\Q$-divisor, we say that $T$ has \emph{generalised algebraic singularities}.

Clearly, if a closed almost positive $(1,1)$-current has analytic singularities, then it has generalised analytic singularities.

Let $f\colon Z\to Y$ be a further resolution, and set $g:=\pi\circ f$. Then the current $f^*\Theta$ has all Lelong numbers zero by Theorem \ref{thm:Favre}, hence the Siu decomposition of $g^*T$ has the form
$$g^*T=f^*\Theta+f^*D.$$
Thus, if $f$ is a sufficiently high resolution, then we may assume that the support of the divisorial part $f^*D$ has simple normal crossings.

\subsection{Scalar products and norms}

Let $X$ be a complex manifold of dimension $n$ with a hermitian metric $\omega$, and $L$ be a hermitian line bundle on $X$ with a singular metric $h$. If $u$ and $v$ are $L$-valued $(p,q)$-forms with measurable coefficients, then $|u|_{h,\omega}$ denotes the pointwise norm on $\bigwedge^{p,q}T_X^*\otimes L$ induced by the hermitian metric on $T_X$ whose fundamental form is $\omega$ and by $h$, $\langle u,v\rangle_{h,\omega}$ is the corresponding scalar product, and $dV_\omega:=\omega^n/n!$ is the volume form associated to $\omega$; cf.\ \cite[p.\ 99]{Hoer65}. Set
$$\lla u,v \rra_{h,\omega}:=\int_X \langle u, v\rangle_{h,\omega}\, dV_\omega\quad\text{and}\quad \|u\|_{h,\omega}:=\lla u,u \rra_{h,\omega}^{1/2}.$$
If $L_{p,q}^2(X,L)_{h,\omega}$ is the set of $L$-valued $(p,q)$-forms with measurable coefficients such that $\|u\|_{h,\omega}<\infty$, then $L_{p,q}^2(X,L)_{h,\omega}$ is a Hilbert space with the scalar product $\lla \cdot\, ,\cdot \rra_{h,\omega}$.

If $\sigma$ is a global holomorphic section of the line bundle $\OO_X(K_X)\otimes L$, then we may view it as a smooth $L$-valued $(n,0)$-form and we write $\|\sigma\|_{h,\omega}$ for the corresponding norm. If $h_{K_X}$ is the smooth metric on $\OO_X(K_X)$ induced by the hermitian metric on $T_X$ whose fundamental form is $\omega$, and if $g:=h_{K_X}h$ is the induced metric on $\OO_X(K_X)\otimes L$, then we also write $\|\sigma\|_g:=\|\sigma\|_{h,\omega}$.

We will need the following remark in the proof of Theorem \ref{thm:boundedsections}.

\begin{rem}\label{rem:vanishat0}
With notation as above, fix a smooth metric $h_0$ on the line bundle $\OO_X(K_X)\otimes L$ and assume that $|\cdot|_{h,\omega}=|\cdot|_{h_0}e^{-\varphi}$, where $\varphi$ is a locally integrable function on $X$ which is bounded from above by a constant $C$. Assume that there exists a coordinate ball $U\subseteq X$, an integrable function $\theta\colon U\to\R\cup\{{-}\infty\}$ and a section $s\in C^\infty\big(U,\OO_X(K_X)\otimes L\big)$ such that
$$\int_U |s|^2_{h,\omega}e^{-2\theta}dV_\omega<\infty,$$
but the function $e^{-2\theta}$ is not locally integrable around some point $x\in U$. Then we claim that $s(x)=0$. Indeed, assume that $s(x)\neq 0$, and pick a small ball $x\in V\subseteq U$ such that $M:=\min\{|s(x)|_{h_0}\mid x\in V\}>0$. Then
$$Me^{-2C}\int_V e^{-2\theta}dV_\omega\leq\int_U |s|_{h_0}^2e^{-2\varphi}e^{-2\theta}dV_\omega=\int_U |s|^2_{h,\omega}e^{-2\theta}dV_\omega<\infty,$$
hence $\int_V e^{-2\theta}dV_\omega<\infty$, a contradiction which implies the claim.
\end{rem}

\subsection{H\"ormander's estimates}

We will need the following result which follows by expanding on the techniques of H\"ormander $L^2$ estimates \cite{Hoer65,Hoer90}. The most general result of this form is in \cite[Th\'eor\`eme 5.1 and Lemme 3.2]{Dem82}, where it was proved for complete K\"ahler varieties; see also \cite[Corollary 5.3]{Dem01}. In this paper we only need it for projective manifolds, in which case the proof is much simpler, see \cite[Theorem 3.1]{Dem92b} or \cite[Lecture 5, Theorem 1.1]{Ber10}

\begin{thm}\label{thm:Hoermander}
Let $X$ be a compact K\"ahler manifold with a K\"ahler form $\omega$. Let $L$ be a line bundle on $X$ with a singular metric $h$ such that $\Theta_h(L)\geq\varepsilon\omega$. Then for every form $v\in L^2_{p,q}(X,L)_{h,\omega}$ with $q\geq1$ and $\dbar v=0$ there exists a form $u\in L^2_{p,q-1}(X,L)_{h,\omega}$ such that
$$ \dbar u=v\quad\text{and}\quad \|u\|^2_{h,\omega}\leq\frac{1}{2\pi q\varepsilon}\|v\|^2_{h,\omega}.$$
\end{thm}

\section{Preliminaries: birational geometry}

A \emph{fibration} is a projective surjective morphism with connected fibres between two normal varieties.

We write $D \geq 0$ for an effective $\R$-divisor $D$ on a normal variety $X$. If $f\colon X\to Y$ is a surjective morphism of normal varieties and if $D$ is an $\R$-divisor on $X$, then $D$ is \emph{$f$-exceptional} if $\codim_Y f(\Supp D) \geq 2$.

If $X$ is a normal projective variety and if $D$ is an $\R$-Cartier $\R$-divisor on $X$, we denote $|D|_\R:=\{D'\geq 0 \mid D'\sim_\R D\}$. 

A \emph{pair} $(X,\Delta)$ consists of a normal variety $X$ and a Weil $\R$-divisor $\Delta\geq0$ such that the divisor $K_X+\Delta$ is $\R$-Cartier. The standard reference for the foundational definitions and results on the singularities of pairs and the Minimal Model Program (MMP) is \cite{KM98}, and we use these freely in this paper. We recall additionally that flips for klt pairs exist by \cite[Corollary 1.4.1]{BCHM}. We use the MMP with scaling of an ample (or just big) divisor as described in \cite[Remark 3.10.10]{BCHM}.

\begin{rem}\label{rem:uniruled}
We will need the following observation in Section \ref{sec:excellentandMMP}: if $(X,\Delta)$ is a $\Q$-factorial pair such that $X$ is not uniruled, then $K_X+\Delta$ is pseudoeffective. Indeed, let $\pi\colon Y\to X$ be a resolution of $X$. Then $Y$ is not uniruled, hence the divisor $K_Y$ is pseudoeffective by \cite[Corollary 0.3]{BDPP}. Then the divisor $K_X\sim_\R\pi_*K_Y$ is pseudoeffective, and the claim is immediate.
\end{rem}

\subsection{Models}\label{subsec:models}

We recall the definition of negative maps, of minimal models and of good minimal models.

\begin{dfn}
Let $X$ and $Y$ be $\Q$-factorial varieties, and let $D$ be an $\R$-divisor on $X$. A birational contraction $f\colon X\dashrightarrow Y$ is \emph{$D$-non-positive} (respectively \emph{$D$-negative}) if there exists a resolution $(p,q)\colon W\to X\times Y$ of the map $f$ such that 
$$p^*D\sim_\R q^*f_*D+E,$$
where $E\geq0$ is a $q$-exceptional $\R$-divisor (respectively, $E\geq0$ is a $q$-exceptional $\R$-divisor and $\Supp E$ contains the proper transform of every $f$-exceptional divisor).
	
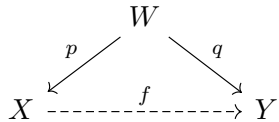
\begin{center}
		\begin{tikzcd}
			& W \arrow[dl, "p" swap] \arrow[dr, "q"] && \\
			X \arrow[rr, dashed, "f" ] && Y
		\end{tikzcd}
	\end{center} 
If $f$ is $D$-negative and additionally $f_*D$ is nef, the map $f$ is a \emph{minimal model} for $D$. If moreover $f_*D$ is semiample, the map $f$ is a \emph{good minimal model} for $D$, or simply a \emph{good model} for $D$.
\end{dfn}

We use these notions almost exclusively for divisors of the form $D=K_X+\Delta$, where $(X,\Delta)$ is a klt pair. Then we talk of minimal and good models of a klt pair $(X,\Delta)$.

Note that if $(X,\Delta)$ is a klt pair, then it has a good model if and only if there exists a Minimal Model Program with scaling of an ample divisor which terminates with a good model of $(X,\Delta)$; this follows from the proof of \cite[Lemma 2.1]{Laz24}.

\subsection{Nakayama--Zariski and Boucksom--Zariski functions}\label{subsec:nakayama}

There are two ways to assign asymptotic functions to pseudoeffective classes: the algebraic construction from \cite{Nak04} and the analytic construction from \cite{Bou04}. They coincide on projective manifolds, but we will need both constructions in this paper.

Let $X$ be a $\Q$-factorial projective variety and let $\Gamma$ be a prime divisor on $X$. Nakayama \cite{Nak04} defined \emph{$\sigma_\Gamma$-functions} on the pseudoeffective cone of $X$; this was originally done when $X$ is smooth, but the definition works well in the $\Q$-factorial setting \cite[Lemma 2.12]{LX23}. We explain briefly their construction. If $D$ is a big $\R$-divisor on $X$, set
$$ \sigma_\Gamma (D) := \inf \{ \mult_\Gamma \Delta \mid 0 \leq\Delta \sim_\R D \}; $$
and if $D$ is a pseudoeffective $\R$-divisor on $X$ and $A$ is an ample $\R$-divisor on $X$, define
$$ \sigma_\Gamma (D) := \lim_{\varepsilon\downarrow 0} \sigma_\Gamma (D+\varepsilon A); $$
this does not depend on the choice of $A$ and is compatible with the definition above for big divisors. Moreover, $\sigma_\Gamma(D)$ only depends on the numerical class of $D$, hence $\sigma_\Gamma$ is well-defined on the pseudoeffective cone of $X$. Each function $\sigma_\Gamma$ is homogeneous of degree $1$, convex and lower semicontinuous on the cone of pseudoeffective divisors on $X$, and it is continuous on the cone of big divisors on $X$.
   
Set
$$ N_\sigma (D) := \sum_\Gamma \sigma_\Gamma(D)\cdot \Gamma\quad\text{and}\quad P_\sigma:=D-N_\sigma(D), $$
where the formal sum runs through all prime divisors $\Gamma$ on $X$. Both $N_\sigma(D)$ and $P_\sigma(D)$ are $\R$-divisors on $X$, and the decomposition $ D = P_\sigma (D) + N_\sigma (D) $ is the \emph{Nakayama--Zariski decomposition} of $D$.

If $X$ is a compact Kähler manifold and if $\Gamma$ is an analytic prime divisor on $X$, Boucksom \cite{Bou04} defined \emph{$\nu(\cdot,\Gamma)$-functions} on the cone of pseudoeffective classes in $H^{1,1}(X,\R)$, and he showed that they coincide with Nakayama's $\sigma_\Gamma$-functions when one considers algebraic classes. To avoid possible confusion with Lelong numbers, we will denote these Boucksom's functions also by $\sigma_\Gamma$. We explain briefly their construction, adopting for the moment the concept of currents with minimal singularities which will be dealt with in detail in Section \ref{sec:minimalsings}.

Let $\alpha$ be a pseudoeffective class in $H^{1,1}(X,\R)$. After fixing a reference K\"ahler form $\omega$, and if $T_{\min,\varepsilon}$ is a current with minimal singularities in the class $\alpha+\varepsilon\{\omega\}$ for a positive real number $\varepsilon$, set
$$\sigma_\Gamma(\alpha):=\inf_{x\in\Gamma}\sup_{\varepsilon>0}\nu(T_{\min,\varepsilon},x);$$
this does not depend on the choice of $\omega$, and one has $\sigma_\Gamma(\alpha)=\nu(T_{\min},\Gamma)$ when $\alpha$ is a big class and $T_{\min}\in \alpha$ is a current with minimal singularities.

\begin{rem}
Even though the notation is slightly different, the definition above is equivalent to that from \cite{Bou04}. We explain this briefly now. If $\alpha$ is a class in $H^{1,1}(X,\R)$ and if $\omega$ is a K\"ahler form on $X$, then \cite[\S2.8]{Bou04} introduces $\alpha[\gamma]$ as the set of closed almost positive $(1,1)$-currents $T\in\alpha$ such that $T\geq\gamma$. Then \cite[\S3.1]{Bou04} defines
$$\sigma_\Gamma(\alpha):=\inf_{x\in\Gamma}\sup_{\varepsilon>0}\nu(\widetilde T_{\min,\varepsilon},x),$$
where for each $\varepsilon>0$, $\widetilde T_{\min,\varepsilon}$ is the current with minimal singularities in $\alpha[{-}\varepsilon\omega]$; this is defined analogously as for pseudoeffective classes in Section \ref{sec:minimalsings}. Now, since $\omega$ is closed, one shows easily that $T_{\min,\varepsilon}=\widetilde T_{\min,\varepsilon}+\varepsilon\omega$, which yields that the definition from \cite{Bou04} is equivalent to the one given in this paper.
\end{rem}

The following lemma is well known and we include the proof for completeness.

\begin{lem}\label{lem:NsigmaMMP}
Let $(X,\Delta)$ be a projective log canonical pair such that $K_X+\Delta$ is pseudoeffective, $\Delta$ is a $\Q$-divisor and such that $(X,\Delta)$ has a minimal model. If $f\colon Y\to X$ is a resolution, then $N_\sigma\big(f^*(K_X+\Delta)\big) $ is a $\Q$-divisor.
\end{lem}

\begin{proof}
Let $\varphi\colon (X,\Delta)\dashrightarrow (X',\Delta')$ be a minimal model of $(X,\Delta)$ and let $(p,q)\colon W\to X\times X'$ be a resolution of indeterminacies of $\varphi$ such that $W$ is smooth. We may assume that $p$ factors through $f$; let $w\colon W\to Y$ be the resulting map.
	\begin{center}
		\begin{tikzcd}
			Y \arrow[d, "f" swap] & W \arrow[l, "w" swap] \arrow[dl, "p" swap] \arrow[dr, "q"] && \\
			X \arrow[rr, dashed, "\varphi" ] && X'
		\end{tikzcd}
	\end{center} 
Then there exists an effective $q$-exceptional $\Q$-divisor $E$ on $W$ such that
$$p^*(K_X+\Delta)\sim_\Q q^*(K_{X'}+\Delta')+E.$$
Then $N_\sigma\big(p^*(K_X+\Delta)\big)=N_\sigma\big(q^*(K_{X'}+\Delta')\big)+E=E$ by \cite[Lemma 2.4]{LP20a}, hence by \cite[Lemma 2.13]{LX23} we have
$$ N_\sigma\big(f^*(K_X+\Delta)\big)=w_*N_\sigma\big(p^*(K_X+\Delta)\big)=w_*E, $$
which proves the lemma.
\end{proof}

\subsection{Stable, diminished and augmented base loci}

A good reference for basic results on the \emph{asymptotic base loci} treated in this subsection is \cite{ELMNP}, see also \cite[Section 2]{TX23}.

If $X$ is a normal projective variety and if $D$ is a pseudoeffective $\R$-Cartier $\R$-divisor on $X$, the \emph{stable base locus} of $D$ is
$$\sB(D):=\bigcap_{D'\in |D|_\R}\Supp D';$$
this is a closed subset of $X$.

\begin{rem}\label{rem:stableQ}
If $D$ is a $\Q$-divisor, by \cite[Lemma 3.5.3]{BCHM} this is equivalent to saying that $\sB(D)=\bigcap\limits_{k\in\N}\Bs|kD|$, hence by \cite[Proposition 2.1.21]{Laz04} we have $\sB(D)=\Bs|kD|$ for all $k$ sufficiently divisible.
\end{rem}

The \emph{diminished base locus} of $D$ is
$$\sB_-(D):=\bigcup_{A\text{ ample on }X}\sB(D+A);$$
this only depends on the numerical equivalence class of $D$ and is a countable union of closed subsets of $X$. If $X$ is additionally $\Q$-factorial, then $N_\sigma(D)$ is the divisorial part of $\sB_-(D)$, see \cite[Lemma 2.17]{LX23}. This locus is sometimes called the \emph{non-nef locus} of $D$; we use both names for this locus interchangeably.

The \emph{augmented base locus} of $D$ is
$$\sB_+(D):=\bigcap_{A\text{ ample on }X}\sB(D-A);$$
it only depends on the numerical equivalence class of $D$ and is a closed subset of $X$. This locus is sometimes called the \emph{non-ample locus} of $D$; we use both names for this locus interchangeably.

\begin{rem}\label{rem:augmented}
We have
\begin{equation}\label{eq:inclusionsloci}
\sB_-(D)\subseteq\sB(D)\subseteq\sB_+(D).
\end{equation}
Further, by \cite[Proposition 1.5]{ELMNP} we have $\sB(D)_+=\sB(D-A)$ for any ample $\R$-divisor $A$ whose numerical class is of sufficiently small norm. From this it is easy to deduce that $D$ is ample if and only if $\sB_+(D)=\emptyset$. By \cite[Lemma 1.14]{ELMNP} we have
$$\sB_-(D)=\bigcup_{A\text{ ample on }X}\sB_+(D+A).$$
\end{rem}

\subsection{Finite generation}

We review now several facts about finitely generated multigraded rings and the existence of minimal models, which will be used in Section \ref{sec:localPL}.

If $X$ is a normal projective variety and if $D$ is a $\Q$-Cartier $\Q$-divisor on $X$, we define the global sections of $D$ by 
$$ H^0(X,D)=\{f\in k(X)\mid \ddiv f+D \geq 0 \}; $$
note that clearly $H^0(X,D)=H^0(X,\lfloor D\rfloor)$. If $D_1,\dots,D_r$ are $\Q$-Cartier $\Q$-divisor on $X$, we define the corresponding \emph{divisorial ring} as
$$ \mathfrak R:=R(X;D_1, \dots, D_r):=\bigoplus_{(n_1,\dots, n_r)\in \N^r} H^0(X,n_1D_1+\dots + n_rD_r). $$
The \emph{support} of $\mathfrak R$, denoted by $\Supp\mathfrak{R}$, is the convex hull of all integral divisors $D$ in the cone $\sum_{i=1}^r\R_+ D_i\subseteq \Div_\R(X)$ such that $H^0(X,D)\neq0$. 

The following result gives the most important example of a finitely generated divisorial ring. The first part of Theorem \ref{thm:models} was proved in \cite[Corollary 1.1.9]{BCHM} and \cite[Theorem A]{CL12a}; see also \cite[Theorem 2]{CL13} and Remark \ref{rem:fingen} for the formulation we adopt in this paper. The second part is a special case of \cite[Theorem 5.4]{KKL16}, and can be also deduced from \cite[Theorem F]{BCHM}.

\begin{thm}\label{thm:models}
Let $X$ be a $\Q$-factorial projective variety and let $\Delta_i$ be $\Q$-divisors on $X$ such that each pair $(X,\Delta_i)$ is klt for $i=1,\dots,r$. Assume that for each $i$ that $\Delta_i$ is big or that $K_X+\Delta_i$ is big. Then the ring
$$ \mathfrak R=R(X;K_X+\Delta_1,\dots,K_X+\Delta_r) $$
is finitely generated. Moreover, $\Supp \mathfrak R$ is a rational polyhedral cone and there is a finite rational polyhedral subdivision $\Supp \mathfrak R=\bigcup \mathcal{C}_k$ with the property that for each $k$ there exist a $\Q$-factorial projective variety $X_k$ and a birational contraction $\varphi_k\colon X\dashrightarrow X_k$ such that $\varphi_k$ is a minimal model for every klt pair $(X,B_k)$ with $K_X+B_k\in\mathcal C_k$.
\end{thm}

\begin{rem}\label{rem:fingen}
Even though the formulation is slightly different, Theorem \ref{thm:models} follows easily from \cite[Theorem 2]{CL13}. Indeed, without loss of generality we may assume that $K_X+\Delta_i$ are big for $i\leq k$ and that $\Delta_i$ are big for $i>k$. For each $i\leq k$ let $E_i$ be an effective $\Q$-divisor such that $K_X+\Delta_i\sim_\Q E_i$, and pick a rational number $\varepsilon>0$ such that $(X,\Delta_i+\varepsilon E_i)$ is klt for each $i\leq k$. Then by \cite[Theorem 2]{CL13} the ring
$$R(X;K_X+\Delta_1+\varepsilon E_1,\dots,K_X+\Delta_k+\varepsilon E_k,K_X+\Delta_{k+1},\dots,K_X+\Delta_r).$$ 
Since $K_X+\Delta_i+\varepsilon E_i\sim_\Q(1+\varepsilon)(K_X+\Delta_i)$ for $i\leq k$, the ring $\mathfrak R$ is finitely generated by \cite[Lemma 2.25]{CL12a}.
\end{rem}

\section{Auxiliary results}

\subsection{(Pluri)subharmonic functions}
In this paper we need very precise properties of (pluri)subharmonic functions. We start with the following easy lemma.

\begin{lem}\label{lem:limit}
Let $\Omega\subseteq\C^n$ be a domain, let $x\in\Omega$ and let $f\colon \Omega\to\R\cup\{{-}\infty\}$ be an upper semicontinuous function. Let $\{a_m\}$ and $\{b_m\}$ be sequences of positive real numbers such that $\lim\limits_{m\to\infty}a_m=1$ and $\lim\limits_{m\to\infty}b_m=0$, and denote $c_m:=\sup_{B(x,1/m)}f$. Then:
\begin{enumerate}[\normalfont (a)]
\item $\lim\limits_{m\to\infty}a_mc_m\leq f(x)$,
\item $\lim\limits_{m\to\infty}b_mc_m\leq 0$,
\item if $f$ is subharmonic, then $\lim\limits_{m\to\infty}a_mc_m=f(x)$,
\item if $f$ is psh, then $\lim\limits_{m\to\infty}\frac1m c_m=0$.
\end{enumerate}
\end{lem}

\begin{proof}
Note that the sequence $\{c_m\}$ is decreasing, hence converging to a value in $\R\cup\{{-}\infty\}$. Therefore,
$$\lim_{m\to\infty}a_mc_m=\lim_{m\to\infty}c_m=\limsup_{x'\to x}f(x')\leq f(x),$$
which shows (a). When $f$ is subharmonic, then the last inequality is an equality by Lemma \ref{lem:strongusc}(a), which gives (c).

If $\lim\limits_{m\to\infty}c_m\in\R$, then $\lim\limits_{m\to\infty}b_mc_m=0$; otherwise we have $c_m<0$ for all $m\gg0$, thus (b) follows.

For (d), note that
$$\lim_{m\to\infty}\frac{c_m}{m}=\lim_{m\to\infty}\frac{c_m}{\log m}\frac{\log m}{m}.$$
Since $\lim\limits_{m\to\infty}\frac{c_m}{{-}\log m}=\nu(f,x)$ and $\lim\limits_{m\to\infty}\frac{\log m}{m}=0$, the claim follows.
\end{proof}

The following two approximation results are much deeper, and they will be crucial in Part \ref{part:approximations}.

\begin{lem}\label{lem:convergencepsh}
Let $X$ be a complex manifold and let $\alpha$ be a continuous $(1,1)$-form on $X$.  Let $\{\varphi_n\}$ be a sequence of $\alpha$-psh functions which are locally uniformly bounded from above and which converge in $L^1_\loc(X)$ to a function $\varphi\in\PSH(X,\alpha)$. Then for every sequence of points $\{x_n\}$ in $X$ which converges to a point $x\in X$ we have
$$\varphi(x)\geq\limsup_{n\to\infty}\varphi_n(x_n).$$
\end{lem}

\begin{proof}
Set $a:=\limsup\limits_{n\to\infty}\varphi_n(x_n)$. Then by passing to subsequences of $\{\varphi_n\}$ and $\{x_n\}$ we may assume that $\varphi_n$ converges to $\varphi$ almost everywhere and
$$a=\lim_{n\to\infty}\varphi_n(x_n).$$
As in the proof of Theorem \ref{thm:compactnessquasipsh}(a), locally around $x$ there exists a smooth closed form $\omega\geq\alpha$. By replacing $\alpha$ by $\omega$ and $X$ by a small neighbourhood around $x$, we may assume that $\alpha$ is smooth and closed.

Fix a small coordinate ball $B(x,2r)$ in $X$ such that the functions $\varphi_n$ are uniformly bounded from above on $B(x,2r)$ and let $\theta$ be a smooth potential of $\alpha$ on $B(x,2r)$. Then the functions $\theta+\varphi$ and all $\theta+\varphi_n$ are psh on $B(x,2r)$. We may assume that $x_n\in B(x,r)$, so that $B(x_n,r)\subseteq B(x,2r)$, and let $\chi_A$ denote the characteristic function of a set $A\subseteq B(x,2r)$. Then the sequence $\{e^{\theta+\varphi_n}\chi_{B(x_n,r)}\}$ is uniformly bounded from above on $B(x,2r)$, and converges almost everywhere to $e^{\theta+\varphi}\chi_{B(x,r)}$, hence by Lebesgue's dominated convergence theorem and by the mean value inequality we have
\begin{align}
\fint_{B(x,r)}e^{\theta+\varphi}dV_\omega&=\lim_{n\to\infty}\fint_{B(x_n,r)}e^{\theta+\varphi_n}dV_\omega \label{eq:0abc}\\
&\geq \lim_{n\to\infty}e^{\theta(x_n)+\varphi_n(x_n)}=e^{\theta(x)+a}.\notag
\end{align}
By letting $r\to0$ in \eqref{eq:0abc} we conclude by Lemma \ref{lem:strongusc}(b) that 
$$e^{\theta(x)+\varphi(x)}\geq e^{\theta(x)+a},$$
which gives the desired inequality.
\end{proof}

\begin{lem}\label{lem:approximatedense}
Let $\varphi$ be a subharmonic function on a domain $\Omega\subseteq\C^n$ and let $A\subseteq\Omega$ be a set of Lebesgue measure zero such that $\Omega\setminus A$ is dense in $\Omega$. Then there exists a countable set $D\subseteq \Omega\setminus A$ which is dense in $\Omega$, such that for every $z\in\Omega$ there exists a sequence $\{z_q\}$ in $D$ with
$$\lim\limits_{q\to\infty}z_q=z\quad\text{and}\quad \lim\limits_{q\to\infty}\varphi(z_q)=\varphi(z).$$
\end{lem}

\begin{proof}
Denote by $\|\cdot\|$ the euclidean norm on $\C^n$. Set 
$$\mathcal C:=\{(y,r)\in\Q^{2n}\times\Q\mid B(y,2r)\subseteq\Omega\},$$
where we view $\Q^{2n}$ as a subset of $\C^n$. For each $(y,r)\in\mathcal C$, let $\widetilde{z}_{y,r}$ be a point in $\overline{B(y,r)}$ such that $\varphi(\widetilde{z}_{y,r})=\max\big(\varphi|_{\overline{B(y,r)}}\big)$. Then by applying Lemma \ref{lem:strongusc}(a) to the point $\widetilde{z}_{y,r}$ we obtain that there exists a point
\begin{equation}\label{eq:0ee}
z_{y,r}\in (\Omega\setminus A)\cap B(y,2r)
\end{equation}
such that 
\begin{equation}\label{eq:0e}
\varphi(z_{y,r})\geq\varphi(\widetilde{z}_{y,r})-r=\max\big(\varphi|_{\overline{B(y,r)}}\big)-r.
\end{equation}

Then we claim that the countable set
$$D:=\{z_{y,r}\mid (y,r)\in\mathcal C\}\subseteq\Omega\setminus A$$
is dense in $\Omega$. Indeed, consider a point $w\in\Omega$. Let $m_0$ be a positive integer such that $B(w,2^{-m_0})\subseteq\Omega$, and for each $m\geq m_0$ pick points
$$w_m\in\Q^{2n}\cap B(w,2^{-m-1}).$$
Then $(w_m,2^{-m-2})\in\mathcal C$ by the definition of $\mathcal C$, hence $$z_{w_m,2^{-m-2}}\in B(w_m,2^{-m-1})\cap D$$
by \eqref{eq:0ee} and by the definition of $D$. Therefore, $z_{w_m,2^{-m-2}}\in B(w,2^{-m})$ for any $m\geq m_0$, which proves that $D$ is dense in $\Omega$.

Now, fix $z\in\Omega$. To finish the proof it suffices to show that for each $\varepsilon>0$ there exists
$$z'\in D\cap B(z,\varepsilon)\quad\text{with}\quad|\varphi(z)-\varphi(z')|<\varepsilon.$$
Assume otherwise. Then there exists $\varepsilon>0$ such that $B(z,\varepsilon)\subseteq\Omega$ and such that for all $z'\in D\cap B(z,\varepsilon)$ we have $|\varphi(z)-\varphi(z')|\geq \varepsilon$. By Lemma \ref{lem:strongusc}(a), this implies that there exists a rational number $0<\delta\leq\varepsilon/3$ such that
\begin{equation}\label{eq:0f}
\text{for all }z'\in D\cap B(z,3\delta)\text{ we have }\varphi(z')\leq\varphi(z)-\varepsilon.
\end{equation}
Pick a point $z_0\in\Q^{2n}\cap B(z,\delta)$. Then the point $z_{z_0,\delta}\in D\cap B(z_0,2\delta)$ constructed as above belongs to $D\cap B(z,3\delta)$, and by \eqref{eq:0e} we have
$$\varphi(z_{z_0,\delta})\geq\max\big(\varphi|_{\overline{B(z_0,\delta)}}\big)-\delta\geq\varphi(z)-\delta,$$
which contradicts \eqref{eq:0f}. This concludes the proof.
\end{proof}

We will, in fact, need the following global version of the previous lemma, which follows from Lemma \ref{lem:approximatedense} by compactness.

\begin{cor}\label{cor:approximatedense}
Let $\varphi$ be a quasi-psh function on a compact complex manifold $X$ and let $A\subseteq X$ be a set of Lebesgue measure zero such that $X\setminus A$ is dense in $X$. Then there exists a countable set $D\subseteq X\setminus A$ which is dense in $X$, such that for every $z\in X$ there exists a sequence $\{z_q\}$ in $D$ with
$$\lim\limits_{q\to\infty}z_q=z\quad\text{and}\quad \lim\limits_{q\to\infty}\varphi(z_q)=\varphi(z).$$
\end{cor}

\subsection{Estimates of sections of line bundles}

The following two lemmas will be essential in Part \ref{part:approximations}.

\begin{lem}\label{lem:holomorphicpsh}
Let $U\subseteq \C^n$ be a domain and let $\{L_j\}_{j\in J}$ be a collection of $\Q$-divisors on $U$. For each $j\in J$, let $h_j$ be a smooth metric on $L_j$ with the associated curvature $\Theta_j$, and assume that the $(1,1)$-forms $\Theta_j$ are uniformly bounded on $U$. Then for each $x\in X$ there exist constants $C>0$ and $r_0>0$ such that for every ball $B(x,r)\subseteq U$ with $r\leq r_0$ and for each $\sigma\in H^0\big(B(x,r),mL_j\big)$ with $j\in J$ and $m\in\N$ such that $mL_j$ is Cartier,
\begin{enumerate}[\normalfont (a)]
\item the function
$$\log|\sigma(z)|_{h_j^m}+mC|z-x|^2$$
is psh on $B(x,r)$, and
\item we have
$$ |\sigma(x)|^2_{h_j^m}\leq e^{2m Cr^2}\fint_{B(x,r)}|\sigma|^2_{h_j^m}dV_\omega. $$
\end{enumerate}
\end{lem}

\begin{proof}
Fix $x\in X$. By the proof of Lemma \ref{lem:diagonalisation} applied to the standard K\"ahler metric on $\C^n$, there exist constants $C>0$ and $r_0>0$ such that ${-}\Theta_j+Cdd^c |z-x|^2\geq0$ on $B(x,r_0)$ for all $j\in J$. For $j\in J$, for $m\in\N$ such that $mL_j$ is Cartier, for $r\leq r_0$ and for $\sigma\in H^0\big(B(x,r),mL_j\big)$ we have $m\Theta_j+dd^c\log|\sigma|_{h_j^m}\geq0$ on $B(x,r)$ by Example \ref{exa:quasi-psh}(b), hence
$$dd^c\log|\sigma|_{h_j^m}+mCdd^c |z-x|^2\geq0\quad \text{on } B(x,r).$$
This shows (a). 

Now, for each $j\in J$ consider the smooth metric $g_j:=h_je^{C|z-x|^2}$ on $L_j|_{B(x,r_0)}$. Since
$$|\sigma|^2_{g_j^m}=e^{2\log|\sigma|_{h_j^m}+2mC|z-x|^2},$$
the function $|\sigma|^2_{g_j^m}$ is psh on $B(x,r)$ by (a) and by Example \ref{exa:quasi-psh}(d), hence the mean value inequality at the point $x$ gives
$$ |\sigma(x)|^2_{h_j^m}=|\sigma(x)|^2_{g_j^m}\leq\fint_{B(x,r)}|\sigma|^2_{g_j^m}dV_\omega \leq e^{2m Cr^2}\fint_{B(x,r)}|\sigma|^2_{h_j^m}dV_\omega, $$
which finishes the proof.
\end{proof}

\begin{lem}\label{lem:compactuniform}
Let $X$ be a complex compact manifold and let $L$ be a line bundle on $X$ with a continuous metric $h$. Let $V\subseteq H^0(X,L)$ be a compact subset with respect to a norm $\|\cdot\|$ on $H^0(X,L)$. Consider sections $\{\sigma_j\}_{j\in\N}$ in $V$ such that $\lim\limits_{j\to\infty}\sigma_j=\sigma_0$ in the norm $\|\cdot\|$. Then the following holds.
\begin{enumerate}[\normalfont (a)]
\item The sections $\sigma_j$ converge uniformly to $\sigma_0$ in the metric $h$, i.e.\ for every $\varepsilon>0$ there exists a positive integer $N$ such that $|\sigma_0(z)-\sigma_j(z)|_h\leq\varepsilon$ for all $j\geq N$ and all $z\in X$.
\item For any sequence of points $\{x_j\}$ in $X$ such that $\lim\limits_{j\to\infty}x_j=x_0$ we have
$$\lim\limits_{j\to\infty}|\sigma_j(x_j)|_h=|\sigma_0(x_0)|_h.$$
\end{enumerate}
\end{lem}

\begin{proof}
Fix a basis $e_1,\dots,e_n$ of $H^0(X,L)$, and write 
$$\sigma_j=\sum\limits_{i=1}^n\alpha_{j,i}e_j\quad\text{for some }\alpha_{j,i}\in\C.$$
Then $\lim\limits_{j\to\infty}\alpha_{j,i}=\alpha_{0,i}$ by assumption. Fix $\varepsilon>0$. Then, by the continuity of $h$, for each $z_0\in X$ there exists $r_{z_0}>0$ and a positive integer $N_{z_0}$ such that
$$\sum_{i=1}^n|\alpha_{0,i}-\alpha_{j,i}|\cdot|e_i(z)|_h\leq\varepsilon$$
for all $j\geq N_{z_0}$ and $z\in B(z_0,r_{z_0})$, hence the triangle inequality gives
\begin{equation}\label{eq:6}
|\sigma_0(z)-\sigma_j(z)|_h\leq\varepsilon
\end{equation}
for all $j\geq N_{z_0}$ and $z\in B(z_0,r_{z_0})$. By compactness we can find finitely many points $z_1,\dots,z_k\in X$ such that the balls $B(z_i,r_{z_i})$ cover $X$. If we set $N:=\max\{N_{z_i}\mid 1\leq i\leq k\}$, then $|\sigma_0(z)-\sigma_j(z)|_h\leq\varepsilon$ for all $j\geq N$ and all $z\in X$ by \eqref{eq:6}, which shows (a).

To show (b), fix $\varepsilon>0$. Then there exists a positive integer $N_1$ such that 
$$\Big||\sigma_0(x_0)|_h-|\sigma_0(x_j)|_h\Big|\leq\varepsilon\quad\text{for all }j\geq N_1$$
On the other hand, by (a) there exists a positive integer $N_2$ such that $|\sigma_0(z)-\sigma_j(z)|_h\leq\varepsilon$ for all $j\geq N_2$ and all $z\in X$. In particular,
$$\Big||\sigma_0(x_j)|_h-|\sigma_j(x_j)|_h\Big|\leq|\sigma_0(x_j)-\sigma_j(x_j)|_h\leq\varepsilon\quad\text{for every }j\geq N_2.$$
Therefore, for all $j\geq\max\{N_1,N_2\}$ we have
$$\Big||\sigma_0(x_0)|_h-|\sigma_j(x_j)|_h\Big|\leq 2\varepsilon,$$
which finishes the proof.
\end{proof}

\subsection{Special quasi-psh functions}

The following results construct particular quasi-psh functions which will be needed in Part \ref{part:approximations}.

\begin{lem}\label{lem:logarithmic}
Let $X$ be a smooth projective variety and let $D$ be a pseudoeffective $\Q$-divisor on $X$. Let $A$ be an ample $\Q$-divisor on $X$, let $\omega\in\{A\}$ be a positive smooth form and let $\alpha\in\{D+A\}$ be a smooth form. Then for each rational number $\varepsilon\in(0,1)$ there exists a quasi-psh function $\psi_\varepsilon$ on $X$ which has logarithmic singularities with poles along $\sB(D+\varepsilon A)$ such that
$$\alpha+dd^c\psi_\varepsilon\geq(1-\varepsilon)\omega.$$
\end{lem}

\begin{proof}
Fix a rational number $\varepsilon\in(0,1)$, let $h$ be a smooth metric on $D$ such that $\Theta_h(D)=\alpha-\omega\in\{D\}$, and let $h_A$ be the smooth metric on $A$ such that $\omega=\Theta_{h_A}(A)$. The $\Q$-divisor $D+\varepsilon A$ is big, hence by Remark \ref{rem:stableQ} there exists a positive integer $m$ such that
$$\sB(D+\varepsilon A)=\Bs|m(D+\varepsilon A)|.$$
Let $\sigma_1,\dots,\sigma_k$ be a basis of the vector space $H^0\big(X,m(D+\varepsilon A)\big)$, and set
$$\psi_\varepsilon=\frac{1}{2m}\log\sum_{i=1}^k|\sigma_i|^2_{h^mh_A^{m\varepsilon}}.$$
Then $(\alpha-\omega)+\varepsilon\omega+dd^c\psi_\varepsilon\geq0$ by Example \ref{exa:quasi-psh}(b) and $\psi_\varepsilon$ clearly has poles along $\Bs|m(D+\varepsilon A)|$. This concludes the proof.
\end{proof}

\begin{cor}\label{cor:logarithmic}
Let $X$ be a smooth projective variety with a K\"ahler form $\omega$, let $D$ be a big $\Q$-divisor on $X$ and let $\alpha\in\{D\}$ be a smooth form. Then there exists a rational number $\varepsilon\in(0,1)$ and a quasi-psh function $\psi$ on $X$ which has logarithmic singularities with poles along $\sB_+(D)$ such that
$$\alpha+dd^c\psi\geq \varepsilon\omega.$$
\end{cor}

\begin{proof}
Let $A$ be an ample $\Q$-divisor on $X$ such that the $\Q$-divisor $D-A$ is big, and let $\omega'\in\{A\}$ be a positive smooth form. Then by Lemma \ref{lem:diagonalisation} there exists a positive constant $C$ such that $C\omega'\geq\omega$, hence by replacing $\omega$ by $\omega'$, we may assume that $\omega\in\{A\}$.

By Remark \ref{rem:augmented} there exists a rational number $\varepsilon\in(0,1)$ such that
$$\sB_+(D)=\sB(D-\varepsilon A)=\sB\big((D-A)+(1-\varepsilon)A\big).$$
Then the result follows from Lemma \ref{lem:logarithmic} applied to the $\Q$-divisor $D-A$, the ample $\Q$-divisor $A$ and the rational number $1-\varepsilon$.
\end{proof}

\newpage

\part{Currents with minimal singularities}

\section{Singularities of currents}\label{sec:minimalsings}

In this section $X$ is always a compact complex manifold. Good sources for the foundational material on currents with minimal singularities are \cite{DPS01,Bou04}.

\subsection{Comparison of singularities}\label{subsection:comparisonsings}

Let $\varphi_1$ and $\varphi_2$ be quasi-psh functions on a compact complex manifold $X$. We say that $\varphi_1$ is \emph{less singular} than $\varphi_2$, and write $\varphi_1\preceq\varphi_2$, if there exists a constant $C$ such that $\varphi_2\leq\varphi_1+C$. We denote by $\varphi_1\approx\varphi_2$ the induced equivalence relation, i.e.\ we say that $\varphi_1$ and $\varphi_2$ have \emph{equivalent singularities} if $\varphi_1\preceq\varphi_2\preceq\varphi_1$. 

If $T_1$ and $T_2$ are two closed almost positive $(1,1)$-currents on $X$ with corresponding global potentials $\varphi_1$ and $\varphi_2$, we say that $T_1$ is \emph{less singular} than $T_2$, and write $T_1\preceq T_2$, if $\varphi_1\preceq\varphi_2$; and similarly for $T_1\approx T_2$. This does not depend on the choice of global potentials. It is immediate that any two closed almost positive $(1,1)$-currents with equivalent singularities have the same Lelong numbers.

\begin{rem}\label{rem:preceqproperties}
The relation $\preceq$ behaves well with respect to multiplication by positive constants and sums of currents. More precisely, let $\varphi_1, \varphi_2$ and $\varphi_3$ be quasi-psh functions on a compact complex manifold $X$ and let $\lambda$ be a positive real number. If $\varphi_1\preceq \varphi_2$, then it follows immediately that $\lambda \varphi_1\preceq\lambda \varphi_2$ and $\varphi_1+\varphi_3\preceq \varphi_2+\varphi_3$. Conversely, if $\varphi_1+\varphi_3\preceq \varphi_2+\varphi_3$, then $\varphi_1\preceq \varphi_2$: this is clear away from the pole set $\{\varphi_1+\varphi_2+\varphi_3={-}\infty\}$, hence it holds everywhere on $X$ by Corollary \ref{cor:strongusc}. Similar statements hold for currents, and are proved by considering their global potentials.
\end{rem}

Now, let $\varphi_1$ and $\varphi_2$ be quasi-psh functions on a compact complex manifold $X$ such that $\varphi_1\preceq\varphi_2$. Then it is immediate to check that
$$ \mathcal I(\varphi_2)\subseteq\mathcal I(\varphi_1).$$
In particular, if $\varphi_1$ and $\varphi_2$ have equivalent singularities, then they have the same multiplier ideal. 

\subsection{Minimal singularities}

Let $\alpha$ be a closed real continuous $(1,1)$-form on $X$ whose class $\{\alpha\}\in H^{1,1}_\mathrm{BC}(X,\R)$ is pseudoeffective. A minimal element $\varphi_{\min}\in\PSH(X,\alpha)$ with respect to the relation $\preceq$ is called a \emph{global potential with minimal singularities in $\PSH(X,\alpha)$}, and the corresponding current
$$T_{\min}=\alpha+dd^c\varphi_{\min}$$
is a \emph{current with minimal singularities in $\{\alpha\}$}; such a global potential and a current always exist by the next paragraph. Note that $T_{\min}\in\{\alpha\}$ is unique up to equivalence of singularities, but is in general not unique, see \cite[Proposition 5.2]{LX24}. One checks immediately that for each point $x\in X$ we have
\begin{equation}\label{eq:Tmin}
\nu(T_{\min},x)=\min_{T\in\alpha}\nu(T,x).
\end{equation}
It is also clear by Remark \ref{rem:preceqproperties} that for each positive number $\lambda$, the current $\lambda T_{\min}$ has minimal singularities in the class $\{\lambda\alpha\}$. By \S\ref{subsection:comparisonsings}, all currents with minimal singularities in a fixed cohomology class have the same multiplier ideal.

To show the existence of global potentials with minimal singularities, following the notation from \cite{GZ05,BEGZ10} we consider the upper envelope 
$$V_\alpha=\sup\big\{\varphi\in\PSH(X,\alpha)\mid \sup\nolimits_X\varphi=0\big\}.$$
The function $V_\alpha$ is again $\alpha$-psh. Indeed, by Theorem \ref{thm:compactnessquasipsh}(a) we have $V_\alpha^*\in\PSH(X,\alpha)$, and clearly $V_\alpha\leq V_\alpha^*$ by the definition of upper semicontinuous regularisations. But then
$$V_\alpha^*\in\big\{\varphi\in\PSH(X,\alpha)\mid \sup\nolimits_X\varphi=0\big\},$$
hence $V_\alpha^*\leq V_\alpha$ by the definition of $V_\alpha$. Thus, $V_\alpha=V_\alpha^*$.

The functions $V_\alpha$ are good for showing some existence results, such as the one above, and they have good regularity properties on the non-ample locus when $\alpha$ is a big class, see \cite{DT21} and the references therein. However, they seem to be too general to be useful in birational geometry. That is the reason why we will consider different global potentials with minimal singularities in this paper: \emph{supercanonical potentials}, studied in Section \ref{sec:supercanonical}.

The main reason why functions $V_\alpha$ are useful is that, as showed above, they themselves belong to the envelope $\big\{\varphi\in\PSH(X,\alpha)\mid \sup\nolimits_X\varphi=0\big\}$ (supercanonical potentials do not satisfy this property and this is one of the main issues in dealing with them). To demonstrate how this is used in practice, we prove the following result noted already in \cite[Lemma 5.2]{BEGZ10}.

\begin{lem}\label{lem:descendingminimal}
Let $X$ be a compact K\"ahler manifold with a K\"ahler form $\omega$, and let $\alpha$ be a real continuous $(1,1)$-form on $X$ whose class $\{\alpha\}\in H^{1,1}(X,\R)$ is pseudoeffective. Denote $\alpha_t:=\alpha+t\omega$ for $t\geq0$. Then the functions $V_{\alpha_t}$ decrease pointwise to $V_\alpha$ as $t\to0$. In particular, the positive currents $\alpha_t+dd^cV_{\alpha_t}$ converge weakly to $\alpha+dd^cV_\alpha$ as $t\to0$, and for real numbers $0\leq t_1\leq t_2$ we have $\mathcal I(V_{\alpha_{t_1}})\subseteq\mathcal I(V_{\alpha_{t_2}})$.
\end{lem}

\begin{proof}
Since $\omega\geq0$, we have $V_{\alpha_{t'}}\in\big\{\varphi\in\PSH(X,\alpha_t)\mid \sup\nolimits_X\varphi=0\big\}$ when $t'\leq t$, hence $V_{\alpha_{t'}}\leq V_{\alpha_t}$ by the definition of $V_{\alpha_t}$; this also shows the statement on multiplier ideals in the lemma. Therefore, the limit $V_0:=\lim\limits_{t\to0}V_{\alpha_t}$ exists and clearly $V_\alpha\leq V_0\leq0$. Further, by Theorem \ref{thm:compactnessquasipsh}(e) we have $V_0\in\PSH(X,\alpha)$ and the functions $V_{\alpha_t}$ converge to $V_0$ in $L^1_\loc(X)$. Thus, $V_0\leq V_\alpha$ by the definition of $V_\alpha$, and so $V_0=V_\alpha$, as desired.
\end{proof}

\subsection{Minimal singularities under pullbacks and sums}

Currents with minimal singularities are stable under pullback:

\begin{prop}\label{pro:pullbackmincurrent} 
Let $\pi\colon Y\to X$ be a surjective morphism with connected fibres between compact complex manifolds and let $\theta\in H^{1,1}_\mathrm{BC}(X,\R)$ be a pseudoeffective class. Then a closed positive $(1,1)$-current $T\in\theta$ has minimal singularities if and only if the current $f^*T\in f^*\theta$ has minimal singularities.
\end{prop}

The proof is in \cite[Proposition 1.12]{BEGZ10}, see also \cite[Proposition 5.1]{LX24}.

\begin{rem}\label{rem:smallerTmin}
Let $T$ be a closed positive $(1,1)$-current on a compact complex manifold $X$ which is a current with minimal singularities in the class $\{T\}$, and let $T_1\leq T$ be a closed positive $(1,1)$-current on $X$. Then $T_1$ is a current with minimal singularities in the class $\{T_1\}$. Indeed, denote $T_2:=T-T_1\geq0$. If $S$ is any closed positive $(1,1)$-current in $\{T_1\}$, then $S+T_2\in\{T\}$. By the definition of currents with minimal singularities we have $T_1+T_2=T\preceq S+T_2$. But then $T_1\preceq S$ by Remark \ref{rem:preceqproperties}, as desired.
\end{rem}

We will need later the following results.

\begin{lem}\label{lem:Tminsigma}
Let $\pi\colon Y\to X$ be a surjective morphism with connected fibres from a smooth complex projective variety to a normal complex projective variety. Let $D$ be a pseudoeffective $\R$-divisor on $X$ and let $E$ be an effective $\pi$-exceptional $\R$-divisor on $Y$.
\begin{enumerate}[\normalfont (a)]
\item For each closed positive current $T\in\{\pi^*D+E\}$ we have $T\geq E$.
\item If a current $S\in\{\pi^*D\}$ has minimal singularities, then the current $S+E\in\{\pi^*D+E\}$ has minimal singularities.
\item If a current $T\in\{\pi^*D+E\}$ has minimal singularities, then $T-E\in\{\pi^*D\}$ is a positive current with minimal singularities.
\end{enumerate}
\end{lem}

\begin{proof}
Part (a) has the same proof as \cite[Lemma 2.15]{LP20b}; alternatively, the proof can be extracted from that of \cite[Corollary 4.4]{LX24}, by replacing there the reference \cite[Proposition III.5.7 and Lemma III.5.14]{Nak04} by either \cite[Lemma 2.16]{GL13} or \cite[Lemma 2.4]{LP20a}. Note that some of those results are stated for $\Q$-divisors, but the proofs work for $\R$-divisors.

For (b), consider a current $S'\in\{\pi^*D+E\}$. Then by (a) we have that $S'-E$ is a positive current in $\{\pi^*D\}$, hence $S\preceq S'-E$ since $S$ has minimal singularities. Therefore, $S+E\preceq S'$ by Remark \ref{rem:preceqproperties}, which shows that $S+E$ has minimal singularities.

To show (c), note that $T-E\geq0$ by (a). We conclude by Remark \ref{rem:smallerTmin}.
\end{proof}

\begin{lem}\label{lem:minsingeasy}
Let $X$ be a compact complex manifold and let $\alpha$ and $\beta$ be smooth forms whose classes in $H^{1,1}_\mathrm{BC}(X,\R)$ are pseudoeffective.
\begin{enumerate}[\normalfont (a)]
\item If $\varphi\in\PSH(X,\alpha)$ is bounded, then $\varphi$ has minimal singularities.
\item Assume that there exist $\varphi_\alpha\in\PSH(X,\alpha)$ and $\varphi_\beta\in\PSH(X,\beta)$ with minimal singularities such that the function $\varphi_\alpha+\varphi_\beta\in\PSH(X,\alpha+\beta)$ has minimal singularities. Then for all functions $\varphi_\alpha'\in\PSH(X,\alpha)$ and $\varphi_\beta'\in\PSH(X,\beta)$ with minimal singularities, the function $\varphi_\alpha'+\varphi_\beta'\in\PSH(X,\alpha+\beta)$ has minimal singularities.
\item Let $D_1$ and $D_2$ be semiample $\Q$-divisors on $X$, and let $T_1\in\{D_1\}$ and $T_2\in\{D_2\}$ be currents with minimal singularities. Then for each $0\leq t\leq 1$, the current $tT_1+(1-t)T_2\in\{tD_1+(1-t)D_2\}$ has minimal singularities.
\end{enumerate}
\end{lem}

\begin{proof}
We first show (a). By assumption there exists a constant $C_\varphi$ such that $\varphi\geq C_\varphi$. If $\varphi'\in\PSH(X,\alpha)$, then it is bounded from above, hence there exists a constant $C_{\varphi'}$ such that $\varphi'\leq C_{\varphi'}$. Therefore, $\varphi'\leq\varphi+C_{\varphi'}-C_\varphi$, so that $\varphi\preceq\varphi'$ and consequently, $\varphi$ has minimal singularities.

For (b), by the definition of minimal singularities we have $\varphi_\alpha\approx\varphi_\alpha'$ and $\varphi_\beta\approx\varphi_\beta'$. Then $\varphi_\alpha+\varphi_\beta\approx\varphi_\alpha'+\varphi_\beta'$ by Remark \ref{rem:preceqproperties}, which gives (b).

Finally, we show (c). By passing to multiples, we may assume that $D_1$ and $D_2$ are integral basepoint free divisors. Then by Example \ref{exa:quasi-psh}(b) there exist smooth positive $(1,1)$-forms $\alpha_1\in\{D_1\}$ and $\alpha_2\in\{D_2\}$, which have minimal singularities by (a). Then again by (a), for each $0\leq t\leq 1$, the current $t\alpha_1+(1-t)\alpha_2\in\{tD_1+(1-t)D_2\}$ has minimal singularities, and we conclude by (b).
\end{proof}

\begin{rem}\label{rem:semiample}
If $D$ is a semiample $\Q$-divisor on a compact complex manifold $X$ and if $T\in\{D\}$ is a current with minimal singularities, then all Lelong numbers of $T$ are zero. Indeed, as in the proof of Lemma \ref{lem:minsingeasy} there exists a smooth positive $(1,1)$-form $\alpha\in\{D\}$, which clearly has all Lelong numbers zero. We conclude by \eqref{eq:Tmin}.
\end{rem}

\subsection{Siu decomposition of currents with minimal singularities}

One of the advantages of working with currents with minimal singularities is that the divisorial part of their Siu decomposition always has finitely many components. This was observed already in \cite[Lemma 4.11]{LM23}; here we provide a slightly different proof. In the following lemma $\rho(X)$ denotes the Picard number of a compact complex manifold $X$.

\begin{lem}\label{lem:Siu}
Let $X$ be a compact complex manifold, let $\theta\in H^{1,1}_\mathrm{BC}(X,\R)$ be a pseudoeffective class, and let $T_{\min}\in\alpha$ be a current with minimal singularities. Let
$$T_{\min}=D+R$$
be the Siu decomposition of $T_{\min}$, where $D$ is its divisorial part and $R$ is its residual part. Then $D$ is an $\R$-divisor which has at most $\rho(X)$ components.
\end{lem}

\begin{proof}
Write
$$D=\sum_{i\in I}\lambda_i D_i,$$
where $D_i$ are prime divisors on $X$, and assume for contradiction that $\#I>\rho(X)$. If we denote $M=\{1,\dots,\rho(X)+1\}$, then there exist $m\in M$ and real numbers $\lambda_i'$ such that 
\begin{equation}\label{eq:1}
D_m\equiv\sum_{i\in M\setminus\{m\}}\lambda_i' D_i.
\end{equation}
Choose a positive real number $\varepsilon$ such that $\lambda_m>\varepsilon$ and $\lambda_i+\varepsilon\lambda_i'>0$ for all $i\in M\setminus\{m\}$. Then we have
$$\sum_{i>\rho(X)+1}\lambda_i D_i+R\geq 0 \quad\text{and}\quad\sum_{i\in M\setminus\{m\}}(\lambda_i+\varepsilon\lambda_i')D_i+(\lambda_m-\varepsilon)D_m\geq0,$$
hence
$$T:=\sum_{i\in M\setminus\{m\}}(\lambda_i+\varepsilon\lambda_i')D_i+(\lambda_m-\varepsilon)D_m+\sum_{i>\rho(X)+1}\lambda_i D_i+R\geq 0,$$
and note that, by \eqref{eq:1}, we have
$$T\equiv T+\varepsilon\Big(D_m-\sum\nolimits_{i\in M\setminus\{m\}}\lambda_i' D_i\Big)=T_{\min}.$$
But then
$$\nu(T,D_m)=\lambda_m-\varepsilon<\lambda_m=\nu(T_{\min},D_m),$$
a contradiction which proves the lemma.
\end{proof}

\section{Supercanonical currents}\label{sec:supercanonical}

In this section we introduce a special kind of currents with minimal singularities: \emph{supercanonical currents}. As mentioned in the introduction and as we will see in Lemma \ref{lem:supercanonical}, these are defined by an exponential $L^1$-condition. This property will yield in Part \ref{part:approximations} that supercanonical currents on big line bundles can be defined only by using global holomorphic sections of their multiples. This \emph{algebraicity} is the main reason why supercanonical currents should be fundamental for applications within the MMP, and this is spelled out in Theorem \ref{thm:main2} and in other results in Sections \ref{sec:supercanbig} and \ref{sec:proofofMain2}.

Supercanonical currents for $\Q$-divisors of the form $K_X+\Delta$, where $(X,\Delta)$ is a projective klt pair, were defined in \cite{BD12}. We use a similar, but somewhat simpler version of that definition in order to extend it to all pseudoeffective classes.

Supercanonical currents are defined in the following lemma, whose proof follows closely, for the most part, the presentation in \cite[Section 5]{BD12}.

\begin{lem}\label{lem:supercanonical}
Let $X$ be a compact complex manifold and let $\alpha$ be a smooth $(1,1)$-form on $X$ whose class $\{\alpha\}\in H_\mathrm{BC}^{1,1}(X,\R)$ is pseudoeffective. Let
$$\textstyle \mathcal S_\alpha:=\big\{\varphi\in\PSH(X,\alpha)\mid \int_X e^{2\varphi}dV_\omega\leq 1\big\},$$
and define the \emph{supercanonical potential} $\varphi_{\alpha,\can}$ associated to $\alpha$ as
$$\varphi_{\alpha,\can}(x):=\sup_{\varphi\in\mathcal S_\alpha}\varphi(x) \quad\text{for }x\in X.$$
Then:
\begin{enumerate}[\normalfont (a)]
\item $\mathcal S_\alpha\neq\emptyset$,
\item all $\varphi\in\mathcal S_\alpha$ are uniformly bounded from above on $X$,
\item $\varphi_{\alpha,\can}(x)=\max\limits_{\varphi\in\mathcal S_\alpha}\varphi(x)$ for $x\in X$,
\item $\varphi_{\alpha,\can}\in\PSH(X,\alpha)$,
\item the current
$$T_{\alpha,\can}=\alpha+dd^c\varphi_{\alpha,\can}$$ 
is a closed positive $(1,1)$-current with minimal singularities in $\{\alpha\}$, called the \emph{supercanonical current} associated to $\alpha$.
\end{enumerate}
\end{lem}

\begin{proof}
\emph{Step 1.}
Consider any $\varphi_0\in\PSH(X,\alpha)$. As $\varphi_0$ is bounded from above, there exists a constant $C_0$ such that $\int_X e^{2\varphi_0}dV_\omega\leq 2C_0$, hence $\varphi_0-\log C_0\in\mathcal S_\alpha$. This gives (a). 

\medskip

\emph{Step 2.}
By compactness of $X$ there exist finitely many coordinate balls $U_i=B(z_i,2r_i)$ for $z_i\in X$ such that the balls $V_i=B(z_i,r_i)$ cover $X$ and for each $i$ there is a smooth function $\theta_i$ on $\overline{U_i}$ such that $\alpha|_{U_i}=dd^c\theta_i$. Denote $M_{i,\min}:=\inf_{U_i}e^{2\theta_i}$ and $M_{i,\max}:=\sup_{U_i}e^{2\theta_i}$.

Let $\varphi\in\mathcal S_\alpha$ and let $x\in V_i$ for some $i$. Then $\theta_i+\varphi|_{U_i}$ is plurisubharmonic on $U_i$, hence so is $e^{\theta_i+\varphi|_{U_i}}$, and we have $B(x,r_i)\subseteq U_i$. The mean value inequality and the assumption $\varphi\in\mathcal S_\alpha$ give
\begin{align*}
e^{2\varphi(x)}M_{i,\min}&\leq e^{2(\theta_i(x)+\varphi(x))}\leq\frac{n!}{\pi^nr_i^{2n}}\int_{B(x,r_i)}e^{2(\theta_i+\varphi)}dV_\omega\\
&\leq\frac{n!}{\pi^nr_i^{2n}}\int_{U_i}e^{2\varphi}e^{2\theta_i}dV_\omega\leq \frac{n!}{\pi^nr_i^{2n}}M_{i,\max},
\end{align*}
hence
$$\varphi(x)\leq \frac12\log\Big(\frac{n!}{\pi^nr_i^{2n}}M_{i,\max}/M_{i,\min}\Big).$$
This shows (b).

\medskip

\emph{Step 3.}
The function $\varphi_{\alpha,\can}$ is well defined by (b), and set $\Phi=(\varphi_{\alpha,\can})^*$. Then $\Phi\in\PSH(X,\alpha)$ by Theorem \ref{thm:compactnessquasipsh}(a), and we claim that $\Phi=\varphi_{\alpha,\can}$.

To that end, fix $x\in X$. We may assume that $\Phi(x)\neq{-}\infty$, since otherwise the claim is clear. Then there exists a sequence $\{x_n\}$ of points in $X$ such that $x_n\to x$ and
$$\Phi(x)=\limsup\limits_{z\to x}\varphi_{\alpha,\can}(z)=\lim_{n\to\infty}\varphi_{\alpha,\can}(x_n),$$
hence, by the definition of $\varphi_{\alpha,\can}$, there exists a sequence of functions $\{\varphi_n\}$ in $\mathcal S_\alpha$ such that
\begin{equation}\label{eq:lim2}
\Phi(x)=\lim_{n\to\infty}\varphi_n(x_n).
\end{equation}
By (b) and by Theorem \ref{thm:compactnessquasipsh}(b), after passing to a subsequence we may assume that the sequence $\{\varphi_n\}$ converges in $L^1_{\loc}(X)$ and almost everywhere to a function $\widetilde\varphi\in\PSH(X,\alpha)$, and then $\widetilde\varphi\in\mathcal S_\alpha$ by Fatou's lemma. In particular, we have 
\begin{equation}\label{eq:7dd}
\widetilde\varphi(x)\leq\varphi_{\alpha,\can}(x)\leq\Phi(x)
\end{equation}
by the definition of $\varphi_{\alpha,\can}$. On the other hand, we have $\widetilde\varphi(x)\geq \Phi(x)$ by Lemma \ref{lem:convergencepsh} and by \eqref{eq:lim2}, which, together with \eqref{eq:7dd}, finishes the proof of the claim and of (d). The same proof also shows that $\varphi_{\alpha,\can}(x)=\widetilde\varphi(x)$, which gives (c).

\medskip

\emph{Step 5.}
Finally, we show (e). Consider any $\varphi_1\in\PSH(X,\alpha)$. Then as in Step 1 there exists a constant $C_1$ such that $\varphi_1-C_1\in\mathcal S_\alpha$, hence $\varphi_{\alpha,\can}\preceq\varphi_1$ by the definition of $\varphi_{\alpha,\can}$. This finishes the proof.
\end{proof}

The following lemma proves the first easy properties of supercanonical currents.

\begin{lem}\label{lem:supercanfirstproperties}
Let $X$ be a compact complex manifold, and let $\alpha$ and $\beta$ be smooth real $(1,1)$-forms on $X$ whose cohomology classes are pseudoeffective. With notation from Lemma \ref{lem:supercanonical} the following holds.
\begin{enumerate}[\normalfont (a)]
\item For each $0\leq\varepsilon\leq1$ we have
$$\varepsilon\varphi_{\alpha,\can}+(1-\varepsilon)\varphi_{\beta,\can}\leq \varphi_{\varepsilon\alpha+(1-\varepsilon)\beta,\can}.$$
\item There exists a constant $C$ such that for each $0\leq\varepsilon\leq1$ and each $\varphi\in\mathcal S_{\alpha+\varepsilon\beta}$ we have $\varphi\leq C$.
\item If $\beta\geq0$, then for each $0\leq\varepsilon_1\leq\varepsilon_2\leq1$ we have
$$\mathcal S_{\alpha+\varepsilon_1\beta}\subseteq\mathcal S_{\alpha+\varepsilon_2\beta}\quad\text{and}\quad\varphi_{\alpha+\varepsilon_1\beta,\can}\leq \varphi_{\alpha+\varepsilon_2\beta,\can}.$$
\end{enumerate}
\end{lem}

\begin{proof}
Part (c) follows immediately from the inclusion $\PSH(X,\alpha+\varepsilon_1\beta)\subseteq\PSH(X,\alpha+\varepsilon_2\beta)$ and from the definition of supercanonical potentials, so we concentrate on (a) and (b).

For (a), fix $0\leq\varepsilon\leq1$. For fixed $\varphi_\alpha\in\mathcal S_\alpha$ and $\varphi_\beta\in\mathcal S_\beta$, it is immediate that
$$\varepsilon\varphi_\alpha+(1-\varepsilon)\varphi_\beta\in\PSH\big(X,\varepsilon\alpha+(1-\varepsilon)\beta\big),$$
and by H\"older's inequality we have
$$\int_X e^{2(\varepsilon\varphi_\alpha+(1-\varepsilon)\varphi_\beta)}dV_\omega\leq\bigg(\int_X e^{2\varphi_\alpha}dV_\omega\bigg)^\varepsilon\bigg(\int_X e^{2\varphi_\beta}dV_\omega\bigg)^{1-\varepsilon}\leq1.$$
Therefore, we have $\varepsilon\varphi_\alpha+(1-\varepsilon)\varphi_\beta\in\mathcal S_{\varepsilon\alpha+(1-\varepsilon)\beta}$, hence
$$\varepsilon\varphi_\alpha+(1-\varepsilon)\varphi_\beta\leq\varphi_{\varepsilon\alpha+(1-\varepsilon)\beta,\can}.$$
Then (a) follows by taking the pointwise supremum over all $\varphi_\alpha\in\mathcal S_\alpha$ and $\varphi_\beta\in\mathcal S_\beta$.

Next we show (b). The proof is analogous to that of Lemma \ref{lem:supercanonical}(b). By compactness of $X$ there exist finitely many coordinate balls $U_i=B(z_i,2r_i)$ for $z_i\in X$ such that the balls $W_i=B(z_i,r_i)$ cover $X$ and for each $i$ there are smooth functions $\theta_i$ and $\xi_i$ on $\overline{U_i}$ such that $\alpha|_{U_i}=dd^c\theta_i$ and $\beta|_{U_i}=dd^c\xi_i$. Denote
$$M_{i,\min}:=\inf_{\varepsilon\in[0,1]}\inf_{U_i}e^{2(\theta_i+\varepsilon\xi_i)}\quad\text{and}\quad M_{i,\max}:=\sup_{\varepsilon\in[0,1]}\sup_{U_i}e^{2(\theta_i+\varepsilon\xi_i)}.$$
Now, fix $0\leq\varepsilon\leq1$ and $\varphi\in\mathcal S_{\alpha+\varepsilon\beta}$, and let $x\in W_i$ for some $i$. Then $\theta_i+\varepsilon\xi_i+\varphi|_{U_i}$ is plurisubharmonic on $U_i$, hence so is $e^{\theta_i+\varepsilon\xi_i+\varphi|_{U_i}}$, and we have $B(x,r_i)\subseteq U_i$. Then as in Step 2 of the proof of Lemma \ref{lem:supercanonical}, the mean value inequality and the assumption $\varphi\in\mathcal S_{\alpha+\varepsilon\beta}$ give
$$e^{2\varphi(x)}\leq \frac{n!}{\pi^nr_i^{2n}}M_{i,\max}/M_{i,\min}.$$
This shows (b).
\end{proof}

\newpage

\part{Asymptotically equisingular approximations}\label{part:asymptoticapproximations}

In Part \ref{part:asymptoticapproximations} we prove the first main result of this paper, Theorem \ref{thm:main1}. We first study in detail the different instances of approximations of currents which are relevant for this paper, in an increasing order of complexity: \emph{asymptotically equisingular approximations}, \emph{good approximations} and finally \emph{excellent approximations}. One of the main technical results of this part is Corollary \ref{cor:samemodel}, which essentially says that in the context of the MMP, the approximation by currents with minimal singularities is asymptotically equisingular if and only if it is excellent. This is one of the main ingredients in the proof of Theorem \ref{thm:main1}.

\section{Asymptotically equisingular approximations}\label{sec:equisingular}

In this section we introduce the weakest form of approximations of currents relevant for this paper.

\begin{dfn}\label{dfn:asymptoticallyequising}
Let $T$ be a closed almost positive $(1,1)$-current on a compact complex manifold $X$. A sequence of closed almost positive $(1,1)$-currents $\{T_m\}_{m\in\N}$ on $X$ is an \emph{asymptotically equisingular approximation} of $T$ if there exist an effective divisor $D$ on $X$ and a sequence of positive integers $\{m_\ell\}_{\ell\in\N_{>0}}$ such that $m_\ell\to\infty$ and
$$\mathcal I(\ell T_{m_\ell})\otimes\OO_X({-}D)\subseteq\mathcal I(\ell T)\subseteq \mathcal I(\ell T_{m_\ell})\otimes\OO_X(D)\quad\text{for all }\ell.$$
\end{dfn}

\begin{rem}
In the definition we use the convention from Remark \ref{rem:sheaves}: in particular, all sheaves in Definition \ref{dfn:asymptoticallyequising} are understood as subsheaves of the sheaf of meromorphic functions on $X$.
\end{rem}

Definition \ref{dfn:asymptoticallyequising} is inspired by equisingular approximations from \cite[Definition 2.3]{Cao14} and \cite[Definition 4.1.3]{Dem15}, although equisingular approximations from op.\ cit.\ seem to be a too restrictive notion to consider in the context of the Minimal Model Program. Note that in Definition \ref{dfn:asymptoticallyequising} we do not require that the sequence $\{T_m\}_{m\in\N}$ converges weakly to $T$, hence the concept of asymptotically equisingular approximations seems to be a very weak one (we will see in Theorem \ref{thm:demailly} that any positive $(1,1)$-current on a compact K\"ahler manifold has such an approximation). The word \emph{approximation} might a priori be misleading, but is in some sense justified by the following lemma.

\begin{lem}\label{lem:valuationscurrents}
Let $X$ be a compact complex manifold and let $T$ be a closed almost positive $(1,1)$-current on $X$ with an asymptotically equisingular approximation $\{T_m\}_{m\in\N}$. Then there exists a sequence of positive integers $\{m_\ell\}_{\ell\in\N_{>0}}$ with $m_\ell\to\infty$, such that for each prime divisor $E$ over $X$ we have 
$$\nu(T,E)=\lim\limits_{\ell\to\infty}\nu(T_{m_\ell},E).$$
\end{lem}

\begin{proof}
Fix a prime divisor $E$ over $X$, let $\pi\colon Y\to X$ be a modification such that $E$ is a prime divisor on $Y$, and let
$$A:=K_Y-\pi^*K_X$$
denote the ramification divisor on $Y$. Let
$$\pi^*T=R+D\quad\text{and}\quad \pi^*T_m=R_m+D_m$$
be the Siu decompositions of $\pi^*T$ and each $\pi^*T_m$, respectively.

By the definition of asymptotically equisingular approximations, there exist an effective integral divisor $G$ on $X$ and a sequence of positive integers $\{m_\ell\}_{\ell\in\N_{>0}}$ such that $m_\ell\to\infty$ and
\begin{equation}\label{eq:10a}
\mathcal I(\ell T_{m_\ell})\otimes\OO_X({-}G)\subseteq\mathcal I(\ell T)\subseteq \mathcal I(\ell T_{m_\ell})\otimes\OO_X(G)\quad\text{for all }\ell.
\end{equation}
By Theorem \ref{thm:DEL+DP}(a)(d) we have
$$\mathcal I(\ell T_{m_\ell}+G)\subseteq\mathcal I(\ell T_{m_\ell})\otimes\OO_X({-}G),$$
which together with the first inclusion in \eqref{eq:10a} gives
$$\mathcal I(\ell T_{m_\ell}+G)\subseteq\mathcal I(\ell T).$$
Then Lemma \ref{lem:Cao} implies
\begin{equation}\label{eq:10}
\mathcal I(\ell\pi^*T_{m_\ell}+\pi^*G)\otimes\OO_Y({-}A)\subseteq \mathcal I(\ell\pi^*T).
\end{equation}
Similarly we obtain
\begin{equation}\label{eq:11}
\mathcal I(\ell\pi^*T+\pi^*G)\otimes\OO_Y({-}A)\subseteq \mathcal I(\ell\pi^*T_{m_\ell}).
\end{equation}
Now, note that for each $\ell$, the $(1,1)$-current $\ell D+\pi^*G$ is the divisorial part of the Siu decomposition of $\ell\pi^*T+\pi^*G$, and $\ell D_{m_\ell}+\pi^*G$ is the divisorial part of the Siu decomposition of $\ell\pi^*T_{m_\ell}+\pi^*G$. Thus, by Theorem \ref{thm:DEL+DP}(e) there exists an analytic open subset $U\subseteq Y$ such that $\codim_Y(Y\setminus U)\geq2$ and
\begin{align*}
\mathcal I(\ell\pi^*T+\pi^*G)|_U&=\OO_{U}(-\lfloor \ell D\rfloor-\pi^*G\big),\\
\mathcal I(\ell\pi^*T_{m_\ell}+\pi^*G)|_U&=\OO_{U}(-\lfloor \ell D_{m_\ell}\rfloor-\pi^*G\big),
\end{align*}
and
$$\mathcal I(\ell\pi^*T)|_U=\OO_{U}(-\lfloor \ell D\rfloor\big)\quad\text{and}\quad\mathcal I(\ell\pi^*T_{m_\ell})|_U=\OO_{U}(-\lfloor \ell D_{m_\ell}\rfloor\big).$$ 
This together with \eqref{eq:10} and \eqref{eq:11} gives
$$ \OO_{U}(-\lfloor \ell D_{m_\ell}\rfloor-\pi^*G-A\big)\subseteq \OO_{U}(-\lfloor \ell D\rfloor\big) $$
and
$$ \OO_{U}(-\lfloor \ell D\rfloor-\pi^*G-A\big)\subseteq\OO_{U}(-\lfloor \ell D_{m_\ell}\rfloor\big). $$
Considering the order of vanishing along $E$, these inclusions imply
\begin{align*}
\lfloor\ell\nu(T,E)\rfloor & -\mult_E \pi^*G-\mult_E A\leq \lfloor\ell\nu(T_{m_\ell},E)\rfloor\\
&\leq \lfloor\ell\nu(T,E)\rfloor+\mult_E \pi^*G+\mult_E A.
\end{align*}
We conclude by dividing these inequalities by $\ell$ and letting $\ell\to\infty$.
\end{proof}

\section{Good approximations}\label{sec:goodapproximations}

In the context of the MMP, asymptotically equisingular approximations carry a priori too little information on the currents involved.  We need first a stronger notion.

\begin{dfn}\label{dfn:goodapprox}
Let $T$ be a closed almost positive $(1,1)$-current on a compact complex manifold $X$. A sequence of closed almost positive $(1,1)$-currents $\{T_m\}_{m\in\N}$ on $X$ is a \emph{good approximation} of $T$ if:
\begin{enumerate}[\normalfont (i)]
\item $\{T_m\}_{m\in\N}$ is an asymptotically equisingular approximation of $T$, and
\item all $T_m$ have generalised analytic singularities.
\end{enumerate}
\end{dfn}

The definition of good approximations is motivated by the following result hidden in \cite[Section 1]{Dem15}, which shows that every closed positive $(1,1)$-current on a compact K\"ahler manifold always has at least one good approximation. Theorem \ref{thm:demailly} is not necessary for the remainder of the paper, but we include it for the sake of completeness: it demonstrates how the concept of good approximations is inspired by the approximation techniques of Demailly \cite{Dem92}.

\begin{thm}\label{thm:demailly}
Let $T$ be a closed positive $(1,1)$-current on a compact K\"ahler manifold $X$. Then there exists a good approximation of $T$.
\end{thm}

\begin{proof}
We recall the argument from \cite[Section 1]{Dem15}. Fix a smooth form $\alpha\in\{T\}$ and let $\varphi$ be a quasi-psh function on $X$ such that $T=\alpha+dd^c\varphi$. By subtracting a constant from $\varphi$ we may assume that $\varphi\leq0$. By the Bergman kernel approximation technique \cite[Theorem 1.2]{Dem15}, there exist quasi-psh functions with analytic singularities $\varphi_m\leq0$ and a constant $C>0$ such that:
\begin{enumerate}[\normalfont (i)]
\item $\alpha + dd^c \varphi_m\geq {-}\varepsilon_m\omega$, where $\lim\limits_{m\to+\infty}\varepsilon_m = 0$,
\item $\varphi_m\geq\varphi-\frac{C}{m}$ for every $m$.
\end{enumerate}
Setting $T_m:=\alpha+dd^c\big(\frac{m+1}{m}\varphi_m\big)$, we claim that $\{T_m\}_{m\in\N}$ is the desired sequence. Indeed, we only need to show $\{T_m\}_{m\in\N}$ is an asymptotically equisingular approximation. To that end, fix $\ell>0$. We have $\mathcal I\big(\frac{m+1}{m}\ell\varphi\big)\subseteq\mathcal I(\frac{m+1}{m}\ell\varphi_m)$ by (ii), and since $\mathcal I\big(\frac{m+1}{m}\ell\varphi\big)=\mathcal I(\ell\varphi)$ for $m\gg0$ by Theorem \ref{thm:GuanZhou}, we conclude that
$$\mathcal I(\ell T)\subseteq\mathcal I(\ell T_m)\quad\text{for }m\gg0.$$
Conversely, by \cite[Corollary 1.12]{Dem15} for $\lambda=\ell$ and $\lambda'=\ell\frac{m+1}{m}$, we have
$$\textstyle\mathcal I\big(\ell T_m\big)\subseteq\mathcal I(\ell T)\quad\text{for }m\gg0,$$
which proves the result.
\end{proof}

The following proposition and corollary are the crucial MMP results on which the rest of the arguments in Part \ref{part:asymptoticapproximations} rely. They relate the Minimal Model Program and good approximations.

\begin{prop}\label{pro:MMPminimalmetrics}
Let $(X,\Delta)$ be a projective klt pair such that $K_X+\Delta$ is pseudoeffective and $X$ is smooth. Assume that $(X,\Delta)$ has a good minimal model. Let $T_{\min}$ be a current with minimal singularities in $\{K_X+\Delta\}$. Then the following holds.
\begin{enumerate}[\normalfont (a)]
\item The current $T_{\min}$ has generalised analytic singularities, and moreover, it has generalised algebraic singularities if $\Delta$ is a $\Q$-divisor.
\item If $\varphi\colon (X,\Delta)\dashrightarrow (X',\Delta')$ is a $(K_X+\Delta)$-non-positive birational contraction such that the divisor $K_{X'}+\Delta'$ is semiample, and if $W\to X$ is a resolution of indeterminacies of $\varphi$ which is smooth, then $T_{\min}$ descends to $W$.
\item If $\pi\colon Y\to X$ is a resolution such that the Siu decomposition of $\pi^*T_{\min}$ has the form
$$\pi^*T_{\min} = R+D,$$
where the residual part $R$ has all Lelong numbers zero and $D$ is the divisorial part, then
$$D=N_\sigma\big(\pi^*(K_X+\Delta)\big).$$
\end{enumerate}
\end{prop}

\begin{proof}
We first show (a). As mentioned in \S\ref{subsec:models}, we may run a $(K_X+\Delta)$-MMP $\varphi\colon (X,\Delta)\dashrightarrow (X',\Delta')$ with scaling of an ample divisor which terminates with a good minimal model $(X',\Delta')$. Let $(p,q)\colon W\to X\times X'$ be a resolution of indeterminacies of $\varphi$ such that $W$ is smooth.
	\begin{center}
		\begin{tikzcd}
			& W \arrow[dl, "p" swap] \arrow[dr, "q"] && \\
			X \arrow[rr, dashed, "\varphi" ] && X'
		\end{tikzcd}
	\end{center} 
By the Negativity lemma \cite[Lemma 3.39]{KM98}, there exists an effective $q$-exceptional $\R$-divisor $E$ on $W$ such that
\begin{equation}\label{eq:negativity}
p^*(K_X+\Delta)\sim_\R q^*(K_{X'}+\Delta')+E_W.
\end{equation}
The current $p^*T_{\min}\in\{p^*(K_X+\Delta)\}$ has minimal singularities by Proposition \ref{pro:pullbackmincurrent}. By \eqref{eq:negativity} and by Lemma \ref{lem:Tminsigma}(c), the current
$$R_W:=p^*T_{\min}-E_W\in\{q^*(K_{X'}+\Delta')\}$$
is a positive current with minimal singularities. As $q^*(K_{X'}+\Delta')$ is semiample, Remark \ref{rem:semiample} gives that $R_W$ has all Lelong numbers zero. In particular,
$$p^*T_{\min}=R_W+E_W$$
is the Siu decomposition of $p^*T_{\min}$ and therefore, the current $T_{\min}$ has generalised analytic singularities. If $\Delta$ is a $\Q$-divisor, then clearly so is $E_W$. This shows (a).

The proof of (b) is the same as that of (a).

Now we show (c). With notation as in the proof of (a) above, we may assume that $p$ factors through $\pi$; let $w\colon W\to Y$ be the resulting map. Then
\begin{equation}\label{eq:5a}
E_W=w^*D
\end{equation}
by the discussion in \S\ref{subsec:generalisedanalytic}.
 By \cite[Lemma 2.4]{LP20a} we have
\begin{equation}\label{eq:5a1}
N_\sigma\big(p^*(K_X+\Delta)\big)=E_W.
\end{equation}
Then by \eqref{eq:5a}, \eqref{eq:5a1} and \cite[Lemma 2.13]{LX23} we obtain
$$ N_\sigma\big(\pi^*(K_X+\Delta)\big)=w_*N_\sigma\big(p^*(K_X+\Delta) \big)=w_*E_W=D, $$
as desired.
\end{proof}

\begin{cor}\label{cor:samemodel}
Let $(X,\Delta)$ be a projective klt pair such that $K_X+\Delta$ is pseudoeffective and $X$ is smooth. Let $A\geq0$ be a big $\R$-divisor on $X$ such that the pair $(X,\Delta+A)$ is klt, and for each $\varepsilon>0$ let $T_{\varepsilon,\min}$ be a current with minimal singularities in $\{K_X+\Delta+\varepsilon A\}$.
\begin{enumerate}[\normalfont (a)]
\item Assume that $(X,\Delta)$ has a minimal model. Then there exists a positive rational number $\delta$ and a resolution $\pi\colon Y\to X$ such that for each $0<\varepsilon\leq\delta$ the current $T_{\varepsilon,\min}$ has generalised analytic singularities and it descends to $Y$. If $\varepsilon\in\Q$ and $\Delta$ is a $\Q$-divisor, then the current $T_{\varepsilon,\min}$ has generalised algebraic singularities.
\item Assume that $(X,\Delta)$ has a good minimal model. Then there exists a positive rational number $\delta$ and a resolution $\pi\colon Y\to X$ such that for each $0\leq\varepsilon\leq\delta$ the current $T_{\varepsilon,\min}$ has generalised analytic singularities and it descends to $Y$. If $\varepsilon\in\Q$ and $\Delta$ is a $\Q$-divisor, then the current $T_{\varepsilon,\min}$ has generalised algebraic singularities.
\end{enumerate}
\end{cor}

\begin{proof}
We first show (a). By \cite[Theorem 1.7]{HH20} we may run a $(K_X+\Delta)$-MMP $\varphi\colon (X,\Delta)\dashrightarrow (X',\Delta')$ with scaling of an ample divisor which terminates with a minimal model $(X',\Delta')$, and denote $A':=\varphi_*A$. Then there exists $0\leq\delta_0\ll1$ such that $\varphi$ is also a partial $(K_X+\Delta+\varepsilon A)$-MMP for all $0\leq\varepsilon\leq\delta_0$. By \cite[Proposition 2.12]{LT22a} there exists $0<\delta\ll\delta_0$ such that, if we run a $(K_{X'}+\Delta'+\delta A')$-MMP with scaling of an ample divisor, then it is a $(K_{X'}+\Delta')$-trivial MMP.

Since $K_{X'}+\Delta'+\delta A'$ is big, this last $(K_{X'}+\Delta'+\delta A')$-MMP terminates by \cite[Corollary 1.4.2]{BCHM}; denote the resulting map by $\rho\colon X'\dashrightarrow X''$, and set $A'':=\rho_*A'$. Then for each $0<\varepsilon\leq\delta$, the map $\rho\circ\varphi\colon X\dashrightarrow X''$ is a $(K_X+\Delta+\varepsilon A)$-MMP such that $(X'',\Delta''+\varepsilon A'')$ is a minimal model of $(X,\Delta+\varepsilon A)$, and in fact, it is a good minimal model of $(X,\Delta+\varepsilon A)$ by the Basepoint free theorem \cite[Theorem 3.3]{KM98}. Let $(\pi,\pi'')\colon Y\to X\times X''$ be a resolution of indeterminacies of $\rho\circ\varphi$ such that $Y$ is smooth.
	\begin{center}
		\begin{tikzcd}
			& Y \arrow[dl, "\pi" swap] \arrow[dr, "\pi''"] && \\
			X \arrow[r, dashed, "\varphi" ] & X' \arrow[r, dashed, "\rho" ] & X''
		\end{tikzcd}
	\end{center} 
Then for each $0<\varepsilon\leq\delta$ the current $T_{\varepsilon,\min}$ has generalised analytic singularities by Proposition \ref{pro:MMPminimalmetrics}(a), and it descends to $Y$ by Proposition \ref{pro:MMPminimalmetrics}(b). The statement on generalised algebraic singularities follows also from Proposition \ref{pro:MMPminimalmetrics}(a). This shows (a).

If the pair $(X,\Delta)$ has a good minimal model, then as mentioned in \S\ref{subsec:models} we may run a $(K_X+\Delta)$-MMP $\varphi\colon (X,\Delta)\dashrightarrow (X',\Delta')$ with scaling of an ample divisor which terminates with a minimal model $(X',\Delta')$. Then we repeat the proof of (a) above verbatim. The only thing to notice is that the divisor $K_{X''}+\Delta''$ is semiample since the map $\rho$ is $(K_{X'}+\Delta')$-trivial, hence the current $T_{0,\min}$ has generalised analytic singularities by Proposition \ref{pro:MMPminimalmetrics}(a), and it descends to $Y$ by Proposition \ref{pro:MMPminimalmetrics}(b). This finishes the proof.
\end{proof}

We can now prove Proposition \ref{prop:MMPimpliesapproximation} announced in the introduction. The heart of the proof is in the following result.

\begin{prop}\label{prop:MMPimpliesapproximation0}
Let $(X,\Delta)$ be a projective klt pair such that $K_X+\Delta$ is pseudoeffective and $X$ is smooth. Assume that $(X,\Delta)$ has a good minimal model. Let $A\geq0$ be a big $\R$-divisor on $X$ such that the pair $(X,\Delta+A)$ is klt, and for each $\varepsilon>0$ let $T_{\varepsilon,\min}$ be a current with minimal singularities in $\{K_X+\Delta+\varepsilon A\}$. Then the sequence $\{T_{\frac1m,\min}\}_{m\in\N_{>0}}$ is an asymptotically equisingular approximation of $T_{0,\min}$.
\end{prop}

\begin{proof}
By Corollary \ref{cor:samemodel}(b) there exists a positive rational number $\delta$ and a birational model $\pi\colon Y\to X$ such that for each $0\leq\varepsilon\leq\delta$ the current $T_{\varepsilon,\min}$ has generalised analytic singularities and it descends to $Y$. Possibly by blowing up further, we may assume additionally that
$$\pi^*\big(N_\sigma(K_X+\Delta)+N_\sigma(A)\big)\cup\Exc(\pi)$$
is a divisor with simple normal crossings support. Define $F$ to be the reduced divisor whose support is $\pi^*\big(N_\sigma(K_X+\Delta)+N_\sigma(A)\big)\cup\Exc(\pi)$.

For each $0\leq\varepsilon\leq\delta$, let $D_\varepsilon$ be the divisorial part of the Siu decomposition of $\pi^*T_{\varepsilon,\min}$. Then
\begin{equation}\label{eq:Depsilon}
D_\varepsilon=N_\sigma\big(\pi^*(K_X+\Delta+\varepsilon A)\big) \quad \text{for }0\leq\varepsilon\leq\delta
\end{equation}
by Proposition \ref{pro:MMPminimalmetrics}(c), hence by \cite[Lemma III.1.7(2)]{Nak04} we have
\begin{equation}\label{eq:6aa}
\lim\limits_{\varepsilon\to0}D_\varepsilon=D_0.
\end{equation}

On the other hand, by \eqref{eq:Depsilon}, by the convexity of Nakayama-Zariski functions and by \cite[Theorem III.5.16]{Nak04}, for each $0\leq\varepsilon\leq\delta$ there exists an effective $\pi$-exceptional divisor $E_\varepsilon$ on $Y$ such that
$$D_\varepsilon\leq N_\sigma\big(\pi^*(K_X+\Delta)\big)+\varepsilon N_\sigma(\pi^*A)=\pi^*N_\sigma(K_X+\Delta)+\varepsilon \pi^*N_\sigma(A)+E_\varepsilon,$$
hence
\begin{equation}\label{eq:SuppF}
\Supp D_\varepsilon\subseteq \Supp F\quad\text{for all }0\leq\varepsilon\leq\delta.
\end{equation}
Then by \eqref{eq:6aa}, for each positive integer $\ell$ we may choose a positive integer $m_\ell\gg\ell$ such that
\begin{equation}\label{eq:difference}
\ell D_{1/m_\ell}-F\leq \ell D_0\leq \ell D_{1/m_\ell}+F.
\end{equation}
As the divisors $D_\varepsilon$ have simple normal crossings support by \eqref{eq:SuppF} for $0\leq\varepsilon\leq\delta$, by Theorem \ref{thm:DEL+DP}(b)(c) we have
\begin{equation}\label{eq:SNCmultiplier}
\mathcal I(\ell\pi^*T_{\varepsilon,\min})=\OO_Y\big({-}\lfloor \ell D_\varepsilon\rfloor\big)\quad\text{for all }0\leq\varepsilon\leq\delta.
\end{equation}
Therefore, by \eqref{eq:difference} and \eqref{eq:SNCmultiplier} and since $F$ is an integral divisor, we conclude that
$$ \mathcal I(\ell \pi^*T_{\frac{1}{m_\ell},\min})\otimes\OO_Y({-}F)\subseteq\mathcal I(\ell\pi^* T_{0,\min}) \subseteq \mathcal I(\ell\pi^* T_{\frac{1}{m_\ell},\min})\otimes\OO_Y(F)$$
for all $\ell$.

Now, there exists an effective divisor $G$ on $X$ such that $F\leq\pi^*G$, hence the inclusions above give
$$ \mathcal I(\ell \pi^*T_{\frac{1}{m_\ell},\min})\otimes\OO_Y({-}\pi^*G)\subseteq\mathcal I(\ell \pi^*T_{0,\min})\subseteq \mathcal I(\ell\pi^* T_{\frac{1}{m_\ell},\min})\otimes\OO_Y(\pi^*G)$$
for all $\ell$. Tensoring these inclusions with $\OO_Y(K_Y-\pi^*K_X)$, pushing forward by $\pi$ and applying \cite[Proposition 5.8]{Dem01} we obtain
$$\mathcal I(\ell T_{\frac{1}{m_\ell},\min})\otimes\OO_X({-}G)\subseteq\mathcal I(\ell T_{0,\min})\subseteq \mathcal I(\ell T_{\frac{1}{m_\ell},\min})\otimes\OO_X(G)$$
for all $\ell$. This finishes the proof.
\end{proof}

Finally, we have:

\begin{proof}[Proof of Proposition \ref{prop:MMPimpliesapproximation}]
By \cite[Corollary 2.2]{Laz24} the pair $(Y,\Delta_Y)$ has a good minimal model. We conclude immediately by Proposition \ref{prop:MMPimpliesapproximation0}.
\end{proof}

\section{Excellent approximations}\label{sec:excellentapproximations}

In order to exploit the MMP fully, we need to consider a yet stronger type of approximations.

\begin{dfn}\label{dfn:excellentapprox}
Let $T$ be a closed almost positive $(1,1)$-current on a compact complex manifold $X$. A sequence of closed almost positive $(1,1)$-currents $\{T_m\}_{m\in\N}$ on $X$ is an \emph{excellent approximation} of $T$ if:
\begin{enumerate}[\normalfont (i)]
\item $\{T_m\}_{m\in\N}$ is a good approximation of $T$,
\item all $T_m$ descend to the same birational model $\pi\colon Y\to X$, and
\item there exists an effective divisor $B$ on $Y$ such that $\Supp B_m\subseteq\Supp B$ for each $m$, where $B_m$ is the divisorial part of the Siu decomposition of $\pi^*T_m$.
\end{enumerate}
\end{dfn}

The main reason why excellent approximations are very useful is contained in the following result.

\begin{thm}\label{thm:valuationscurrents}
Let $X$ be a compact complex manifold. Let $T$ be a closed almost positive $(1,1)$-current on $X$ such that the divisorial part of its Siu decomposition contains only finitely many components. Then the following are equivalent:
\begin{enumerate}[\normalfont (a)]
\item $T$ is a current with generalised analytic singularities,
\item there exists an excellent approximation $\{T_m\}_{m\in\N}$ of $T$.
\end{enumerate}
\end{thm}

\begin{proof}
If $T$ has generalised analytic singularities, then trivially the currents $T_m:=T$ for $m\in\N$ form an excellent approximation of $T$.

\medskip

Conversely, assume that there exists an excellent approximation $\{T_m\}$ of $T$. By the discussion in \S\ref{subsec:Siudecomposition} we have that the divisorial part of the Siu decomposition of the pullback of $T$ to any resolution of $X$ also contains only finitely many components. Then there exists a modification $\pi\colon Y\to X$ from a compact complex manifold $Y$ such that all $T_m$ descend to $Y$, and the Siu decompositions of $\pi^*T$ and $\pi^*T_m$ have the form
\begin{equation}\label{eq:2}
\pi^*T=R+\sum_{i\in I} \lambda_i D_i
\end{equation}
and
\begin{equation}\label{eq:3}
\pi^*T_m=R_m+\sum_{i\in I} \lambda_{i,m} D_i,
\end{equation}
where $R$ and $R_m$ are the residual parts, and:
\begin{enumerate}[\normalfont (i)]
\item $\lambda_i=\nu(\pi^*T,D_i)$ and $\lambda_{i,m}=\nu(\pi^*T_m,D_i)$ for each $m$ and $i$,
\item the index set $I$ is finite, and
\item each $R_m$ has all Lelong numbers zero.
\end{enumerate}
It suffices to show that all Lelong numbers of $R$ are zero.

To that end, pick a point $y\in Y$. Take a resolution $\mu\colon Z\to Y$ which factors through the blowup of $Y$ at $y$ and let $E$ be the corresponding prime divisor on $Z$. By Lemma \ref{lem:valuationscurrents} there exists a sequence of positive integers $\{m_\ell\}_{\ell\in\N_{>0}}$ with $m_\ell\to\infty$, such that 
\begin{equation}\label{eq:3a}
\lim_{\ell\to\infty}\nu(T_{m_\ell},E)=\nu(T,E)\quad\text{and}\quad \lim\limits_{\ell\to\infty}\lambda_{i,m_\ell}=\lambda_i.
\end{equation}
On the other hand, by \eqref{eq:2} and \eqref{eq:3} we have
\begin{equation}\label{eq:3c}
\nu(T,E)=\nu(\mu^*\pi^*T,E)=\sum_{i\in I} \lambda_i \mult_E\mu^*D_i+\nu(\mu^*R,E)
\end{equation}
and
\begin{equation}\label{eq:3b}
\nu(T_{m_\ell},E)=\nu(\mu^*\pi^*T_{m_\ell},E)=\sum_{i\in I} \lambda_{i,m_\ell} \mult_E\mu^*D_i+\nu(\mu^*R_{m_\ell},E).
\end{equation}
Therefore, letting $\ell\to\infty$ in \eqref{eq:3b} and using \eqref{eq:3a} and \eqref{eq:3c} yields
$$\lim\limits_{\ell\to\infty}\nu(\mu^*R_{m_\ell},E)=\nu(\mu^*R,E).$$
Since $\nu(\mu^*R_{m_\ell},E)=0$ for all $\ell$ by Theorem \ref{thm:Favre}, we obtain $\nu(\mu^*R,E)=0$, hence $\nu(R,y)=0$ by Theorem \ref{thm:Favre} again, as desired.
\end{proof}

\section{Excellent approximations and the MMP}\label{sec:excellentandMMP}

In this section we connect excellent approximations and the Minimal Model Program.

We first need the following extension of \cite[Theorem 4.3]{LP18} to klt pairs. The proof is almost the same as that of \cite[Theorem 4.3]{LP18} and \cite[Theorem 4.1]{LP20b}; the assumptions in those results are however different. In order to avoid confusion and for completeness, we include the proof here.

\begin{thm}\label{thm:nonvanishingForms}
Assume the existence of good minimal models for projective klt pairs in dimensions at most $n-1$. 

Let $(X,\Delta)$ be a projective $\Q$-factorial klt pair of dimension $n$ such that $X$ is not uniruled and $\Delta$ is a $\Q$-divisor. Let $t$ be a positive integer such that $M:=t(K_X+\Delta)$ is Cartier, and let $\pi\colon Y\to X$ be a resolution of $X$. Assume that for some positive integer $p$ we have  
$$ H^0\big(Y,(\Omega^1_Y)^{\otimes p} \otimes \OO_Y(m\pi^*M)\big) \neq 0$$
for infinitely many integers $m$. Then $\kappa (X,K_X+\Delta) \geq 0$. 
\end{thm} 

\begin{proof}
We note first that the $\Q$-divisor $K_X+\Delta$ is pseudoeffective by Remark \ref{rem:uniruled}. 
If $K_X+\Delta\equiv0$, then $K_X+\Delta\sim_\Q0$ by \cite[Corollary V.4.9]{Nak04}. Therefore, from now on we may assume that $M\not\equiv0$. We apply \cite[Lemma 4.1]{LP18} with $\mathcal E := (\Omega^1_Y)^{\otimes p} $ and $\mathcal L := \pi^*\OO_X(M)$. Then there exist a positive integer $r$, a saturated line bundle $\mathcal M$ in $\bigwedge^r\mathcal E$, an infinite set $\mathcal S\subseteq\N$ and integral divisors $N_m\geq0$ for $m\in\mathcal S$ such that 
$$ \OO_Y(N_m) \simeq \mathcal M\otimes \mathcal L^{\otimes m}\quad\text{for all }m\in\mathcal S. $$
Since $Y$ is not uniruled by assumptions, the divisor $K_Y$ is pseudoeffective by \cite[Corollary 0.3]{BDPP}, hence \cite[Proposition 4.2]{LP18} implies that there exist a positive integer $\ell$ and a pseudoeffective divisor $F$ such that
$$ N_m+ F \sim m\pi^*M+\ell K_Y. $$
By pushing forward this relation to $X$ we get
$$\pi_*N_m+\pi_*F\sim_\Q mM +\ell K_X,$$ 
and hence
$$\pi_*N_m+(\pi_*F+\ell\Delta)\sim_\Q (mt+\ell)(K_X+\Delta).$$ 
Noting that $\pi_*N_m$ is effective and that $\pi_*F+\ell\Delta$ is pseudoeffective, we conclude by \cite[Theorem 3.3]{LP18}.
\end{proof} 

Now we can deduce a criterion for Nonvanishing, related to the existence of excellent approximations of currents with minimal singularities.

\begin{thm}\label{thm:LP18}
Let $(X,\Delta)$ be a projective klt pair of dimension $n$ such that $X$ is smooth and not uniruled. Then $K_X+\Delta$ is pseudoeffective by Remark \ref{rem:uniruled}, and let $T_{\min}$ be a closed positive $(1,1)$-current with minimal singularities in $\{K_X+\Delta\}$. Assume that there exists an excellent approximation $\{T_m\}_{m\in\N}$ of $T_{\min}$.
\begin{enumerate}[\normalfont (a)]
\item Then $T_{\min}$ has generalised analytic singularities.
\item Assume the existence of good minimal models for projective klt pairs in dimensions at most $n-1$. If $\Delta$ is a $\Q$-divisor and $\kappa(X,K_X+\Delta) = {-} \infty$, then for every resolution $\pi\colon Y\to X$ with the property that $T_{\min}$ descends to $Y$ and that the divisorial part $D$ of the Siu decomposition of $\pi^*T_{\min}$ has simple normal crossings support, we have
$$ H^p\big(Y,\OO_Y(K_Y+\ell\pi^*(K_X+\Delta)-\lfloor \ell D\rfloor)\big) = 0 $$
for all $p$ and all $\ell>0$ sufficiently divisible. Moreover, if $T_{\min}$ has generalised algebraic singularities, then $D$ is a $\Q$-divisor.
\end{enumerate}
\end{thm}

\begin{proof}
Part (a) follows from Lemma \ref{lem:Siu} and Theorem \ref{thm:valuationscurrents}. 

After part (a) is settled, the rest of the argument for (b) is hidden in the proofs of \cite[Corollary 4.5]{LP18} and \cite[Theorem 5.1]{LP20b} and we reproduce the details here. By (a) and by \S\ref{subsec:generalisedanalytic} there exists a resolution $\pi\colon Y\to X$ such that the Siu decomposition of $\pi^*T_{\min}$ has the form
$$\pi^*T_{\min} = R+D,$$
where the residual part $R$ has all Lelong numbers zero and $D$ is the divisorial part of the decomposition with simple normal crossings support. By Theorem \ref{thm:DEL+DP}(b)(c) we have
\begin{equation}\label{eq:metric1}
\mathcal I(\ell\pi^* T_{\min})=\OO_Y\big({-}\lfloor \ell D\rfloor \big)\quad\text{for all }\ell\geq0.
\end{equation}
Since we assume that $\kappa(X,K_X+\Delta) = {-} \infty$, we conclude by Theorem \ref{thm:nonvanishingForms} that for all $p\geq 0$ and for all $\ell>0$ sufficiently divisible we have
$$ H^0\big(Y,\Omega^p_Y \otimes \pi^*\OO_X(\ell(K_X+\Delta))\big)=0,$$
and thus
$$H^0\big(Y,\Omega^p_Y\otimes \pi^*\OO_X(\ell(K_X+\Delta))\otimes\mathcal I(\ell\pi^* T_{\min})\big) = 0.$$
Then \cite[Theorem 0.1]{DPS01} implies that for all $p\geq 0$ and for all $\ell>0$ sufficiently divisible we have
$$H^p\big(Y,\OO_Y(K_Y+\ell\pi^*(K_X+\Delta))\otimes\mathcal I(\ell\pi^* T_{\min})\big) = 0,$$
which together with \eqref{eq:metric1} finishes the proof.
\end{proof}

\section{Proof of Theorem \ref{thm:main1}}\label{sec:proofofMainThm1}

We now have all the ingredients to prove the first main result of the paper. As announced in the introduction, we can actually show the following much more precise version.

\begin{thm}\label{thm:main1a}
Let $(X,\Delta)$ be a projective klt pair of dimension $n$ such that $K_X+\Delta$ is nef and $\Delta$ is a $\Q$-divisor. Let $\pi\colon Y\to X$ be a log resolution of $(X,\Delta)$ and write 
$$K_Y+\Delta_Y\sim_\Q\pi^*(K_X+\Delta)+E,$$
where $\Delta_Y$ and $E$ are effective $\Q$-divisors without common components. Let $A$ be an ample $\R$-divisor on $Y$, and assume that there exist an effective divisor $D$ on $Y$ and a sequence of positive integers $\{m_\ell\}_{\ell\in\N_{>0}}$ such that $m_\ell\to\infty$ and
$$\textstyle\mathcal I\big(\ell(K_Y+\Delta_Y+\frac{1}{m_\ell} A)\big)_{\min}\subseteq \mathcal I\big(\ell(K_Y+\Delta_Y)\big)_{\min}\otimes\OO_Y(D)\quad\text{for all }\ell.$$
Then
\begin{enumerate}[\normalfont (a)]
\item any current with minimal singularities in the class $\{K_Y+\Delta_Y\}$ is a current with generalised algebraic singularities.
\end{enumerate}
Assume additionally the existence of good minimal models for projective klt pairs in dimensions at most $n-1$. Then:
\begin{enumerate}[\normalfont (a)]
\item[{\normalfont (b)}] if $\kappa(X,K_X+\Delta)\geq0$, then $K_X+\Delta$ is semiample,
\item[{\normalfont (c)}] if $\chi(X,\OO_X)\neq0$, then $K_X+\Delta$ is semiample.
\end{enumerate}
\end{thm}

We will give the proof of Theorem \ref{thm:main1a} -- and thus of Theorem \ref{thm:main1} -- at the end of the section. It will be an easy consequence of the following main technical result of this section.

\begin{thm}\label{thm:main1b}
Let $(X,\Delta)$ be a projective klt pair of dimension $n$ such that $X$ is smooth, $\Delta$ is a $\Q$-divisor and $K_X+\Delta$ is pseudoeffective. Assume that $(X,\Delta)$ has a minimal model. Let $A\geq0$ be a big $\R$-divisor on $X$ such that the pair $(X,\Delta+A)$ is klt, and for each $\varepsilon\geq0$ let $T_{\varepsilon,\min}$ be a current with minimal singularities in $\{K_X+\Delta+\varepsilon A\}$. Assume that the sequence $\{T_{\frac1m,\min}\}_{m\in\N_{>0}}$ is an asymptotically equisingular approximation of $T_{0,\min}$.
\begin{enumerate}[\normalfont (a)]
\item Then $T_{0,\min}$ is a current with generalised algebraic singularities.
\end{enumerate}
Assume additionally the existence of good minimal models for projective klt pairs in dimensions at most $n-1$. Then:
\begin{enumerate}[\normalfont (a)]
\item[{\normalfont (b)}] if $\kappa(X,K_X+\Delta)\geq0$, then $(X,\Delta)$ has a good minimal model,
\item[{\normalfont (c)}] if $\chi(X,\OO_X)\neq0$, then $(X,\Delta)$ has a good minimal model.
\end{enumerate}
\end{thm}

\begin{proof}
We divide the proof in several steps.

\medskip

\emph{Step 1.}
By Corollary \ref{cor:samemodel}(a) there exist a positive integer $m_0$ and a birational model $\pi\colon Y\to X$ such that for each $m\geq m_0$ the current $T_{\frac1m,\min}$ descends to $Y$. For each $m\geq m_0$, let $D_m$ be the divisorial part of the Siu decomposition of $\pi^*T_{\frac1m,\min}$. Then
\begin{equation}\label{eq:6a}
\textstyle D_m=N_\sigma\big(\pi^*(K_X+\Delta+\frac1m A)\big)\quad\text{for }m\geq m_0
\end{equation}
by Proposition \ref{pro:MMPminimalmetrics}(c), and we have 
$$\textstyle \Supp N_\sigma\big(\pi^*\big(K_X+\Delta+\frac1m A\big)\big)\subseteq\Supp N_\sigma\big(\pi^*(K_X+\Delta)\big)\cup\Supp N_\sigma(\pi^*A)$$
for all $m>0$ by the convexity of Nakayama--Zariski functions. This and \eqref{eq:6a} show that
\begin{equation}\label{eq:excellent}
\text{the sequence }\{T_{\frac1m,\min}\}_{m\geq m_0}\text{ is an excellent approximation of }T_{0,\min}.
\end{equation}

\medskip

\emph{Step 2.}
By \eqref{eq:excellent} and by Theorem \ref{thm:LP18}(a) we deduce that $T_{0,\min}$ has generalised analytic singularities. By possibly replacing $Y$ by a higher birational model, we may assume that $T_{0,\min}$ descends to $Y$, and let
$$\pi^*T_{0,\min} = R_0+D_0$$
be the Siu decomposition of $T_{0,\min}$, where the residual part $R_0$ has all Lelong numbers zero and $D_0$ is the divisorial part. By Lemma \ref{lem:valuationscurrents}, there exists a sequence of positive integers $\{m_\ell\}_{\ell\in\N_{>0}}$ with $m_\ell\to\infty$, such that 
$$D_0=\lim_{\ell\to\infty}D_{m_\ell},$$
which together with \eqref{eq:6a} and \cite[Lemma III.1.7(2)]{Nak04} gives
\begin{equation}\label{eq:6b}
D_0=\lim_{\ell\to\infty}\textstyle N_\sigma\big(\pi^*\big(K_X+\Delta+\frac{1}{m_\ell} A\big)\big)=N_\sigma\big(\pi^*(K_X+\Delta)\big).
\end{equation}
Hence, $D_0$ is a rational divisor by Lemma \ref{lem:NsigmaMMP}. In other words, $T_{0,\min}$ has generalised \emph{algebraic} singularities, which proves (a).

\medskip

\emph{Step 3.}
In this step we assume that $\chi(X,\OO_X)\neq0$, and we show that $\kappa(X,K_X+\Delta)\geq0$. 

If $X$ is uniruled, then $\kappa(X,K_X+\Delta)\geq0$ by \cite[Theorem 1.1]{LM21}. Therefore, from now on we may assume that $X$ is not uniruled. We follow the arguments of \cite[Corollary 4.5]{LP18} closely. 

Assume that $\kappa(X,K_X+\Delta)={-}\infty$. By possibly replacing $Y$ by a higher birational model, we may assume that the $\Q$-divisor $D_0$ on $Y$ has simple normal crossings support. Then by Theorem \ref{thm:LP18}(b) we have
$$\chi\big(Y,\OO_Y(K_Y+\ell\pi^*(K_X+\Delta)-\lfloor \ell D_0\rfloor)\big) = 0$$
for all $\ell>0$ divisible by some positive integer $q$, and we may assume that $qD_0$ and $q(K_X+\Delta)$ are Cartier. Then Serre duality gives
\begin{equation}\label{eq:89}
\chi\big(Y,\OO_Y( \ell qD_0 - \ell q\pi^*(K_X+\Delta))\big) = 0\quad \text{for all }\ell>0.
\end{equation}
Since the Euler--Poincar\'e characteristic $\chi\big(Y,\OO_Y( \ell qD_0 - \ell q\pi^*(K_X+\Delta))\big)$ is a polynomial in $\ell$ by the Hirzebruch--Riemann--Roch theorem, \eqref{eq:89} implies that it must be identically zero, hence $\chi(Y,\OO_Y) = 0$ by setting $\ell=0$. Thus, $\chi(X,\OO_X) = 0$ as $X$ has rational singularities, a contradiction which proves that $\kappa(X,K_X+\Delta)\geq0$.

\medskip

\emph{Step 4.}
Finally, in this step we show (b) and (c) simultaneously. By Step 3, we may assume that
\begin{equation}\label{eq:89a}
\kappa(X,K_X+\Delta)\geq0.
\end{equation}
By assumption, there exists a minimal model $\varphi\colon (X,\Delta)\dashrightarrow (X',\Delta')$ of $(X,\Delta)$. By possibly replacing $Y$ by a higher birational model, we may assume that $(\pi,\pi')\colon Y\to X\times X'$ is a resolution of indeterminacies of $\varphi$ such that $Y$ is smooth.
	\begin{center}
		\begin{tikzcd}
			& Y \arrow[dl, "\pi" swap] \arrow[dr, "\pi'"] && \\
			X \arrow[rr, dashed, "\varphi" ] && X'
		\end{tikzcd}
	\end{center} 
Then as in the proof of Lemma \ref{lem:NsigmaMMP} we obtain that
$$ P_\sigma\big(\pi^*(K_X+\Delta)\big)\sim_\Q(\pi')^*(K_{X'}+\Delta'). $$
This together with \eqref{eq:6b} implies
\begin{align*}
R_0=\pi^*T_{\min}-D_0&\equiv\pi^*(K_X+\Delta)-N_\sigma\big(\pi^*(K_X+\Delta)\big)\\
&=P_\sigma\big(\pi^*(K_X+\Delta)\big)\sim_\Q(\pi')^*(K_{X'}+\Delta').
\end{align*}
Since $\kappa(X',K_{X'}+\Delta')=\kappa(X,K_X+\Delta)\geq0$ by \eqref{eq:89a} and since $R_0$ has all Lelong numbers zero, we conclude that $K_{X'}+\Delta'$ is semiample by \cite[Theorem 1.5]{GM17}. Thus, $\varphi$ is a good minimal model of $(X,\Delta)$, which concludes the proof.
\end{proof}

Finally, we have:

\begin{proof}[Proof of Theorem \ref{thm:main1a}]
We first note that the pair $(Y,\Delta_Y)$ has a minimal model by \cite[Lemma 2.14(e)]{LX23}.

For each $\varepsilon\geq0$ let $T_{\varepsilon,\min}$ be a current with minimal singularities in $\{K_Y+\Delta_Y+\varepsilon A\}$. Then by Lemma \ref{lem:descendingminimal}, for all positive integers $m$ and $\ell$ we have
$$\mathcal I(\ell T_{0,\min})\subseteq \mathcal I(\ell T_{\frac1m,\min}),$$
hence the assumptions of Theorem \ref{thm:main1a} imply that $\{T_{\frac1m,\min}\}_{m\in\N_{>0}}$ is an asymptotically equisingular approximation of $T_{0,\min}$. Then part (a) follows from Theorem \ref{thm:main1b}(a) applied to the pair $(Y,\Delta_Y)$. 

For (b) and (c), note that $\kappa(X,K_X+\Delta)=\kappa(Y,K_Y+\Delta_Y)$, as well as $\chi(X,\OO_X)=\chi(Y,\OO_Y)$ since $X$ has rational singularities. Thus, if $\kappa(X,K_X+\Delta)\geq0$ or if $\chi(X,\OO_X)\neq0$, then $(Y,\Delta_Y)$ has a good minimal model by Theorem \ref{thm:main1b}(b)(c). This implies that $(X,\Delta)$ has a good minimal model by \cite[Corollary 2.2]{Laz24}. But then $K_X+\Delta$ is semiample by the same argument as in the third paragraph of the proof of \cite[Lemma 4.1]{LM21}. This proves (b) and (c), and finishes the proof of the theorem.
\end{proof}

\section{Local behaviour}\label{sec:localPL}

In this section we prove a general result on local linearity of currents with minimal singularities in the context of the Minimal Model Program, and on the local behaviour of the asymptotic base loci. It will be one of the ingredients in the proof of Theorem \ref{thm:main2}.

\begin{thm}\label{thm:localPL}
Let $(X,\Delta)$ be a projective klt pair of dimension $n$ such that $\Delta$ is a $\Q$-divisor and $K_X+\Delta$ is pseudoeffective. Assume that $(X,\Delta)$ has a minimal model. Let $A$ be an ample $\Q$-divisor on $X$. Then there exists a rational number $0<\delta\leq1$ such that the following holds.
\begin{enumerate}[\normalfont (a)]
\item The sets $\sB_-(K_X+\Delta+\varepsilon A)$ are independent of $\varepsilon\in[0,\delta)$.
\item For each $\varepsilon\in(0,\delta)$ we have
$$\sB_-(K_X+\Delta+\varepsilon A)=\sB(K_X+\Delta+\varepsilon A)=\sB_+(K_X+\Delta+\varepsilon A).$$
\item Assume that $X$ is additionally smooth, and for each $\varepsilon\geq0$ let $T_{\varepsilon,\min}$ be a current with minimal singularities in $\{K_X+\Delta+\varepsilon A\}$. Then for any two $\varepsilon_1,\varepsilon_2\in (0,\delta]$ and for any $t\in[0,1]$ the current
$$tT_{\varepsilon_1,\min}+(1-t)T_{\varepsilon_2,\min}$$
has minimal singularities in the class $\big\{K_X+\Delta+\big(t\varepsilon_1+(1-t)\varepsilon_2\big)A\big\}$.
\end{enumerate}
\end{thm}

\begin{proof}
First note that by replacing $A$ by a sufficiently general divisor $\Q$-linearly equivalent to $A$, we may assume that the pair $(X,\Delta+A)$ is klt.

\medskip

\emph{Step 1.}
As in the proof of Corollary \ref{cor:samemodel}(a), there exists a rational number $\delta>0$ and a $(K_X+\Delta)$-non-positive birational contraction $\xi\colon X\dashrightarrow X'$ such that for each $0<\varepsilon\leq\delta$, the map $\xi$ is a $(K_X+\Delta+\varepsilon A)$-MMP. Moreover, if we set $\Delta':=\xi_*\Delta$ and $A':=\xi_*A$, then $(X',\Delta'+\varepsilon A')$ is a good minimal model of $(X,\Delta+\varepsilon A)$ for $0<\varepsilon\leq\delta$ by the Basepoint free theorem \cite[Theorem 3.3]{KM98}.

Let $(p,q)\colon Y\to X\times X'$ be a resolution of indeterminacies of $\xi$ such that $Y$ is smooth.
	\begin{center}
		\begin{tikzcd}
			& Y \arrow[dl, "p" swap] \arrow[dr, "q"] && \\
			X \arrow[rr, dashed, "\xi" ] && X'
		\end{tikzcd}
	\end{center} 
By the Negativity lemma \cite[Lemma 3.39]{KM98}, for each $\varepsilon\in[0,\delta]$ there exists an effective $q$-exceptional $\R$-divisor $E_\varepsilon$ on $Y$ such that
\begin{equation}\label{eq:simR}
p^*(K_X+\Delta+\varepsilon A)\sim_\R q^*(K_{X'}+\Delta'+\varepsilon A')+E_\varepsilon.
\end{equation}
Then
\begin{equation}\label{eq:affine1}
\text{the function }\varepsilon\mapsto E_\varepsilon\text{ is affine on }[0,\delta],
\end{equation}
since both functions $\varepsilon\mapsto p^*(K_X+\Delta+\varepsilon A)$ and $\varepsilon\mapsto q^*(K_{X'}+\Delta'+\varepsilon A')$ are.

\medskip

\emph{Step 2.}
In this step we prove (a). First note that by \eqref{eq:simR} and since each divisor $K_{X'}+\Delta'+\varepsilon A'$ is semiample for $\varepsilon\in(0,\delta]$, we have
$$\sB\big(p^*(K_X+\Delta+\varepsilon A)\big)=\Supp E_\varepsilon\quad\text{for all }\varepsilon\in(0,\delta],$$
which together with \cite[Proposition 2.8]{BBP13} and \cite[Lemma 2.3]{LMT23} implies
\begin{equation}\label{eq:Bepsilon}
\sB_-(K_X+\Delta+\varepsilon A)=\sB(K_X+\Delta+\varepsilon A)=p(\Supp E_\varepsilon)\quad\text{for all }\varepsilon\in(0,\delta].
\end{equation}
On the other hand, \eqref{eq:simR} and \cite[Lemma 2.4]{LP20a} give
\begin{equation}\label{eq:Eepsilon}
N_\sigma\big(p^*(K_X+\Delta+\varepsilon A)\big)=E_\varepsilon\quad\text{for all }\varepsilon\in[0,\delta].
\end{equation}
Moreover, since $A$ is ample, by the convexity of Nakayama--Zariski functions we have for each $0\leq\xi_1\leq\xi_2$:
\begin{align*}
N_\sigma\big(p^*(K_X+\Delta+\xi_2 A)\big)&\leq N_\sigma\big(p^*(K_X+\Delta+\xi_1 A)\big)+N_\sigma\big((\xi_2-\xi_1)p^*A)\big)\\
&=N_\sigma\big(p^*(K_X+\Delta+\xi_1 A)\big),
\end{align*}
hence $\Supp E_{\xi_2}\subseteq \Supp E_{\xi_1}$ when $0\leq\xi_1\leq\xi_2$. This together with \eqref{eq:affine1} and \eqref{eq:Eepsilon} shows that
\begin{equation}\label{eq:independent}
\Supp E_0=\Supp E_\varepsilon\quad\text{for all }\varepsilon\in[0,\delta).
\end{equation}
Now, \eqref{eq:Bepsilon} and \eqref{eq:independent} imply that
\begin{equation}\label{eq:independent1}
\sB_-(K_X+\Delta+\varepsilon A)=p(\Supp E_0)\quad\text{for all }\varepsilon\in(0,\delta),
\end{equation}
whereas \cite[Proposition 1.19]{ELMNP} together with \eqref{eq:Bepsilon} and \eqref{eq:independent} gives
$$\sB_-(K_X+\Delta)=\bigcup_{\varepsilon\in(0,\delta)}\sB(K_X+\Delta+\varepsilon A)=p(\Supp E_0)$$
This and \eqref{eq:independent1} give (a).

\medskip

\emph{Step 3.}
For (b), fix $\varepsilon\in(0,\delta)$. Then by \cite[Proposition 1.21]{ELMNP} there exists $\xi\in(0,\varepsilon)$ such that
$$\sB_+(K_X+\Delta+\varepsilon A)=\sB_-\big(K_X+\Delta+(\varepsilon-\xi)A\big).$$
Since $\sB_-\big(K_X+\Delta+(\varepsilon-\xi)A\big)=\sB_-(K_X+\Delta+\varepsilon A)$ by (a), we conclude that
$$\sB_-(K_X+\Delta+\varepsilon A)=\sB_+(K_X+\Delta+\varepsilon A),$$
which together with \eqref{eq:inclusionsloci} proves (b).

\medskip

\emph{Step 4.}
Finally, in this step we prove (c). By \eqref{eq:simR}, for $\varepsilon\in [0,\delta]$ we have
$$ p^*T_{\varepsilon,\min}\equiv q^*(K_{X'}+\Delta'+\varepsilon A')+E_\varepsilon, $$
and set
\begin{equation}\label{eq:S_epsilon}
S_\varepsilon:=p^*T_{\varepsilon,\min}-E_\varepsilon\in\{q^*(K_{X'}+\Delta'+\varepsilon A')\}.
\end{equation}
Since $p^*T_{\varepsilon,\min}$ is a positive current with minimal singularities by Proposition \ref{pro:pullbackmincurrent}, so is also $S_\varepsilon$ by Lemma \ref{lem:Tminsigma}(c).

Fix $\varepsilon_1,\varepsilon_2\in (0,\delta]$. Then by \eqref{eq:affine1} we have
\begin{equation}\label{eq:affine}
E_{t\varepsilon_1+(1-t)\varepsilon_2}=tE_{\varepsilon_1}+(1-t)E_{\varepsilon_2}\quad\text{for each }t\in[0,1].
\end{equation}
Since for every $\varepsilon\in(0,\delta]$ the current $S_\varepsilon$ has minimal singularities and the $\R$-divisor $q^*(K_{X'}+\Delta'+\varepsilon A')$ is semiample, for each $t\in[0,1]$ the current
$$tS_{\varepsilon_1}+(1-t)S_{\varepsilon_2}\in\big\{q^*\big(K_{X'}+\Delta'+(t\varepsilon_1+(1-t)\varepsilon_2) A'\big)\big\}$$
has minimal singularities by Lemma \ref{lem:minsingeasy}(c). Thus, by Lemma \ref{lem:Tminsigma}(b) and by \eqref{eq:simR} each current
\begin{equation}\label{eq:sumofminimalandexceptional}
tS_{\varepsilon_1}+(1-t)S_{\varepsilon_2}+E_{t\varepsilon_1+(1-t)\varepsilon_2}\in\big\{p^*\big(K_X+\Delta+(t\varepsilon_1+(1-t)\varepsilon_2) A\big)\big\}
\end{equation}
has minimal singularities. Note that by \eqref{eq:S_epsilon} and \eqref{eq:affine} we have
\begin{align*}
p^*\big(tT_{\varepsilon_1,\min}+(1-t)T_{\varepsilon_2,\min}\big)
&=t(S_{\varepsilon_1}+E_{\varepsilon_1})+(1-t)(S_{\varepsilon_2}+E_{\varepsilon_2})\\
&=tS_{\varepsilon_1}+(1-t)S_{\varepsilon_2}+E_{t\varepsilon_1+(1-t)\varepsilon_2},
\end{align*}
which together with \eqref{eq:sumofminimalandexceptional} gives that for each $t\in[0,1]$ the current
$$p^*\big(tT_{\varepsilon_1,\min}+(1-t)T_{\varepsilon_2,\min}\big)\in\{p^*(K_X+\Delta+(t\varepsilon_1+(1-t)\varepsilon_2) A)\}$$
has minimal singularities. We conclude by Proposition \ref{pro:pullbackmincurrent}.
\end{proof}

We conclude this section with a few comments on the behaviour of the asymptotic base loci; the following example and proposition were obtained in discussions with Nikolaos Tsakanikas.

Recall that according to \cite{ELMNP} a pseudoeffective $\R$-Cartier $\R$-divisor $D$ on a normal projective variety $X$ is called \emph{stable} if $\sB_-(D)=\sB_+(D)$. Therefore, Theorem \ref{thm:localPL}(b) shows that the stability of adjoint divisors holds, in a certain sense, locally on a klt pair $(X,\Delta)$. Proposition \ref{pro:local}, which complements results from \cite{BBP13,TX23}, says that in a similar situation as in Theorem \ref{thm:localPL}, actually \emph{all but finitely many} divisors of the form $K_X+\Delta+\varepsilon A$ are stable. We first note that, however, one cannot conclude that \emph{all} such divisors are stable.

\begin{exa}\label{exa:notstable}
This example is a slightly modified version of \cite[Example 3.5]{TX23}, and it shows that there exists a projective klt pair $(X,\Delta)$ such that $K_X+\Delta$ is big, but
$$\sB(K_X+\Delta)\neq\sB_+(K_X+\Delta).$$
To that end, let $X$ be the blowup of $\mathbb P^2$ along three distinct points which belong to a line $L\subseteq\mathbb P^2$, and let $E_1$, $E_2$ and $E_3$ be the exceptional divisors. Then ${-}K_X\sim 3L'+2(E_1+E_2+E_3)$, where $L'$ is the strict transform of $L$ on $X$. From this it is easy to check that ${-}K_X$ is nef, but it is not ample since $K_X\cdot L'=0$. Moreover, ${-}K_X$ is big as $K_X^2=6$. By \cite[Proposition 2.61(3)]{KM98} there exists an effective $\Q$-divisor $\Delta\sim_\Q{-}2K_X$ such that the pair $(X,\Delta)$ is klt. Therefore, $K_X+\Delta\sim_\Q{-}K_X$ is nef and big, hence semiample by the Basepoint free theorem \cite[Theorem 3.3]{KM98}. Thus, $\sB(K_X+\Delta)=\emptyset$, whereas $\sB_+(K_X+\Delta)\neq\emptyset$ since $K_X+\Delta$ is not ample.
\end{exa}

\begin{prop}\label{pro:local}
Let $(X,\Delta)$ be a projective klt pair such that $\Delta$ is a $\Q$-divisor and $K_X+\Delta$ is pseudoeffective, and assume that $(X,\Delta)$ has a minimal model. Let $A$ be an ample $\Q$-divisor on $X$. Then there exist only finitely many real numbers $\varepsilon\geq0$ such that $K_X+\Delta+\varepsilon A$ is not stable, and all such $\varepsilon$ are rational.
\end{prop}

When $(X,\Delta)$ is a projective klt pair such that $K_X+\Delta$ is big, then it has a minimal model by \cite{BCHM,CL13}, hence Proposition \ref{pro:local} applies unconditionally to such pairs.

\begin{proof}
Set $n:=\dim X$.

Assume first that $K_X+\Delta$ is big. By \cite[Theorem 3.7(1)]{KM98} we have that $K_X+\Delta+mA$ is ample for all $m>2n$, and in particular each such divisor is stable. On the other hand, by Theorem \ref{thm:models} applied to the ring
$$R(X;K_X+\Delta,K_X+\Delta+2nA),$$
there exist finitely many rational numbers $0=\varepsilon_1<\varepsilon_2<\dots<\varepsilon_k=2n$ such that for each $i$ there exists a $\Q$-factorial projective variety $X_i$ and a birational contraction $\varphi_i\colon X\dashrightarrow X_i$ such that $\varphi_i$ is a minimal model for every klt pair $(X,\Delta+\xi A)$ with $\varepsilon_i<\xi<\varepsilon_{i+1}$. Then for each $\xi\in(\varepsilon_i,\varepsilon_{i+1})$ the divisor $K_X+\Delta+\xi A$ is stable, by repeating verbatim the proof of Theorem \ref{thm:localPL}(b). Therefore, if $K_X+\Delta+\varepsilon A$ is not stable for some $\varepsilon\geq0$, then $\varepsilon\in\{\varepsilon_1,\dots,\varepsilon_k\}$. This proves the proposition when $K_X+\Delta$ is big.

In the general case, by Theorem \ref{thm:localPL}(b) there exists a rational number $0<\delta\leq1$ such that for each $\varepsilon\in(0,\delta)$ the divisor
$K_X+\Delta+\varepsilon A$ is stable. On the other hand, by the first part of the proof there exist only finitely many real numbers $\varepsilon\geq\delta$ such that $K_X+\Delta+\varepsilon A$ is not stable, and they are all rational. This finishes the proof.
\end{proof}

\newpage

\part{Approximations by supercanonical currents}\label{part:approximations}

\section{A uniform bound theorem}

In this section we prove a crucial result that will be used several times in the remainder of the paper. It shows the existence of global holomorphic sections of adjoint line bundles with precise properties, which depend only on a prescribed open cover of the given compact complex manifold.

The method of the proof is to construct holomorphic sections locally by the Ohsawa-Takegoshi extension theorem, and then use smooth cut-off functions and solve a $\dbar$-equation by a version of H\"ormander's $L^2$ estimates to find global holomorphic sections satisfying similar estimates. These techniques go back at least to the proofs of \cite[Theorem 4.2.7]{Hoer07} and \cite[Proposition 3.1]{Dem92}. One of the main difficulties is to organise the proof in such a way that all the constants depend only on the starting data.

\begin{thm}\label{thm:boundedsections}
Let $X$ be a compact K\"ahler manifold with a K\"ahler form $\omega$, and fix an open covering $\mathcal U$ of $X$ by coordinate balls on which $K_X$ trivialises. Then there exist positive constants $\delta$ and $C$, depending only on $\mathcal U$, with the following property. For each line bundle $L$ on $X$ which trivialises on $\mathcal U$, for every singular metric $h$ on $L$ with $\Theta_h(L)\geq \delta\omega$, and for each $x\in X$ such that the restriction of $h$ to the fibre $L_x$ is well defined, there is a section $\sigma_x\in H^0\big(X,\OO_X(K_X)\otimes L\big)$ such that
\begin{equation}\label{eq:twoconditions}
|\sigma_x(x)|_{h,\omega}=1\quad\text{and}\quad \|\sigma_x\|_{h,\omega}\leq C.
\end{equation}
\end{thm}

\begin{proof}
Let $n:=\dim X$.

\medskip

\emph{Step 1.}
In this step we prepare an open covering of $X$ and a sequence of new metrics we will need in the next steps.

We fix a finite covering $\{V_1,\dots,V_r\}$ of $X$ by coordinate balls such that for each $1\leq i\leq r$ we have $V_i\Subset W_i\Subset U_i$ for some coordinate balls $W_i$ and $U_i$, where the covering $\{U_1,\dots,U_r\}$ is subordinate to $\mathcal U$; this is possible by the compactness of $X$. For each $1\leq i\leq r$ fix a function $\chi_i\in C_c^\infty(X)$ such that $0\leq\chi_i\leq1$, $\Supp(\chi_i)\subseteq U_i$ and $\chi_i\equiv1$ on $W_i$. 

Given a point $x\in \overline{V_i}$ for some $1\leq i\leq r$, consider the function 
$$\varphi_{i,x}(z):=n\chi_i(z)\log|z-x|\quad\text{for }z\in U_i.$$
Then the extension by zero of $\varphi_{i,x}$ defines a function $\varphi_{i,x}\colon X\to\R\cup\{{-}\infty\}$ such that $\Supp(\varphi_{i,x})\subseteq U_i$, and such that the following properties hold:
\begin{enumerate}[\normalfont (i)]
\item $dd^c\varphi_{i,x}\geq0$ on $W_i$, since $\chi_i\equiv1$ on $W_i$ and the function $z\mapsto \log|z-x|$ is plurisubharmonic,
\item $dd^c\varphi_{i,x}$ is a smooth form on $X\setminus W_i$ whose coefficients are bounded independently of $i$ and $x$, since $(x,z)\mapsto\chi_i(z)\log|z-x|$ is a smooth function on the compact set $\overline{V_i}\times\big(\Supp(\chi_i)\setminus W_i\big)$.
\end{enumerate}
Then (i), (ii) and Lemma \ref{lem:diagonalisation} imply that there exists a constant $\eta>0$ such that
\begin{equation}\label{eq:9}
dd^c\varphi_{i,x}\geq{-}\eta\omega\quad\text{for all $i$ and $x$}.
\end{equation}
Note that $\eta$ depends only on the choice of sets $V_i,W_i$ and $U_i$ and on the choice of functions $\chi_i$.

Set
\begin{equation}\label{eq:9g}
M_1:=\max_{1\leq i\leq r}\max\big\{\dbar\chi_i(z)\mid z\in X\big\}
\end{equation}
and
\begin{equation}\label{eq:9h}
M_2:=\max_{1\leq i\leq r}\sup\big\{e^{{-}2\varphi_{i,x}(z)}\mid (x,z)\in V_i\times(U_i\setminus W_i)\big\};
\end{equation}
note that $M_2$ is well defined for the same reason as in (ii) above and since $\Supp(\varphi_{i,x})\subseteq U_i$. Further, set
\begin{equation}\label{eq:9k}
M_3:=\min_{1\leq i\leq r}\inf\big\{e^{{-}2\varphi_{i,x}(z)}\mid (x,z)\in V_i\times X\big\}.
\end{equation}
Then $M_3>0$ since $\varphi_{i,x}(z)\leq n\log\big(\diam(U_i)\big)$ for $(x,z)\in V_i\times U_i$ and $\varphi_{i,x}(z)=0$ otherwise. Note that $M_1,M_2$ and $M_3$ depend only on the choice of sets $V_i,W_i$ and $U_i$ and on the choice of functions $\chi_i$.

\medskip

\emph{Step 2.}
Set $\delta:=2\eta$. In the remainder of the proof we show that there exists a constant $C>0$ such that for any singular metric $h$ on $L$ with $\Theta_h(L)\geq \delta\omega$ and for each $x\in X$ such that the restriction of $h$ to the fibre $L_x$ is well defined, there is a section $\sigma_x\in H^0(X,K_X+L)$ such that \eqref{eq:twoconditions} holds.

Fix a singular metric $h$ on $L$ with
\begin{equation}\label{eq:0d}
\Theta_h(L)\geq \delta\omega,
\end{equation}
and fix a point $x\in V_i$ for some $1\leq i\leq r$ for which the restriction of $h$ to the fibre $L_x$ is well defined. By the Ohsawa--Takegoshi extension theorem \cite[Theorem]{OT87} on $U_i$, applied successively $n$ times to a collection of $n$ hyperplanes intersecting at $x$, there exist a constant $C_1$ depending only on the cover $\{U_1,\dots,U_r\}$ (and, in particular, not on $x$ and $h$) and a section $s_x\in H^0\big(U_i,\OO_X(K_X)\otimes L\big)$ such that
\begin{equation}\label{eq:9i}
|s_x(x)|_{h,\omega}=1\quad\text{and}\quad \int_{U_i}|s_x|_{h,\omega}^2dV_\omega\leq C_1.
\end{equation}
The problem is that $s_x$ is a holomorphic section only on $U_i$ and not on the whole $X$. The strategy is to use Theorem \ref{thm:Hoermander} to rectify this.
 
Note that $\chi_i s_x$ is a smooth $L$-valued $(n,0)$-form on $X$, and set 
$$f_x := \dbar(\chi_i s_x).$$
Then $f_x$ is a smooth $L$-valued $(n,1)$-form on $X$ such that $\dbar f_x = 0$, and note that
\begin{equation}\label{eq:9f}
f_x = \dbar\chi_i\cdot s_x
\end{equation}
since $s_x$ is holomorphic, hence
\begin{equation}\label{eq:9e}
\Supp(f_x)\subseteq U_i\setminus W_i
\end{equation}
since $\chi_i\equiv1$ on $W_i$ and $\Supp(\chi_i)\subseteq U_i$.

Now, consider the singular metric $h_{i,x}:=h e^{{-}2\varphi_{i,x}}$ on $L$. Then by \eqref{eq:9} and \eqref{eq:0d} we have
$$\Theta_{h_{i,x}}(L)\geq(\delta-\eta)\omega=\eta\omega.$$
Therefore, as $\eta>0$, by Theorem \ref{thm:Hoermander} there exists an $L$-valued $(n,0)$-form $u_x$ on $X$ such that $\dbar u_x = f_x$ in the sense of currents and
\begin{equation}\label{eq:estimateeta}
\|u_x\|^2_{h_{i,x},\omega}\leq \frac{1}{2\pi\eta}\|f_x\|^2_{h_{i,x},\omega}.
\end{equation}
Moreover, observe that
\begin{align*}
\|f_x\|^2_{h_{i,x},\omega}&=\int_X |f_x|^2_{h,\omega}e^{{-}2\varphi_{i,x}}dV_\omega & \\
&=\int_{U_i\setminus W_i} |f_x|^2_{h,\omega}e^{{-}2\varphi_{i,x}}dV_\omega & \text{by \eqref{eq:9e}} \\
&\leq M_2\int_{U_i\setminus W_i} |f_x|^2_{h,\omega}dV_\omega & \text{by \eqref{eq:9h}} \\
&=M_2\int_{U_i\setminus W_i} \big|\dbar\chi_i\cdot s_x\big|^2_{h,\omega}dV_\omega & \text{by \eqref{eq:9f}} \\
&\leq C_1M_1^2M_2, & \quad\text{by \eqref{eq:9g} and \eqref{eq:9i}}
\end{align*}
and therefore, setting $C_2:=\frac{1}{2\pi\eta} C_1M_1^2M_2$, by \eqref{eq:estimateeta} we have
\begin{equation}\label{eq:9j}
\|u_x\|^2_{h_{i,x},\omega}\leq C_2.
\end{equation}

\medskip

\emph{Step 3.}
Set 
$$\sigma_x := \chi_i s_x - u_x.$$
Then $\dbar \sigma_x = \dbar(\chi_i s_x)-\dbar u_x = 0$ in the sense of currents, which implies that $\sigma_x$ is a holomorphic $L$-valued $(n,0)$-form by the regularity of the $\dbar$-operator. In particular, $u_x$ is a smooth $L$-valued $(n,0)$-form, as it is the difference of smooth forms $\chi_i s_x$ and $\sigma_x$. Since
$$|u_x|^2_{h_{i,x},\omega}=|u_x|^2_{h,\omega}|z-x|^{{-}2n\chi_i(z)},$$
and since the function $|z-x|^{{-}2n}$ is not locally integrable at $z=x$, the inequality \eqref{eq:9j} and Remark \ref{rem:vanishat0} imply that $u_x(x) = 0$. Thus, $|\sigma_x(x)|_{h,\omega} = 1$ by \eqref{eq:9i}.

On the other hand, by \eqref{eq:9k} and \eqref{eq:9j} we have
\begin{equation}\label{eq:9l}
M_3\|u_x\|^2_{h,\omega}\leq\int_X |u_x|^2_{h,\omega}e^{{-}2\varphi_{i,x}}dV_\omega=\|u_x\|^2_{h_{i,x},\omega}\leq C_2.
\end{equation}
Finally, the triangle inequality together with \eqref{eq:9i} and \eqref{eq:9l} gives
\begin{align*}
\|\sigma_x\|_{h,\omega}&\leq \|\chi_i s_x\|_{h,\omega}+\|u_x\|_{h,\omega}\\
&\leq \Big(\int_{U_i}|s_x|_{h,\omega}^2dV_\omega\Big)^{1/2}+\|u_x\|_{h,\omega}\leq \sqrt{C_1}+\sqrt{C_2/M_3}.
\end{align*}
Therefore, $C:=\sqrt{C_1}+\sqrt{C_2/M_3}$ is the desired constant.
\end{proof}

\section{Supercanonical currents on big line bundles}\label{sec:supercanbig}

In this section we analyse supercanonical currents on big $\Q$-divisors on a projective manifold in detail. The main result of the section, Theorem \ref{thm:supercanbig}, says that the corresponding supercanonical potentials depend only on the global sections of multiples of $L$ in a very precise sense. 

The main technical result of the section is Theorem \ref{thm:supercanbigrepresentatives}. The proof uses the main ideas of the proof of \cite[Proposition 5.19]{BD12}, which we occasionally follow closely and which in turn uses essentially Demailly's estimates from his regularisation results \cite{Dem92}. However, several arguments in \cite{BD12} are difficult to follow. Instead, in this paper we use crucially the uniform bounds result (Theorem \ref{thm:boundedsections}) as well as the approximation result (Corollary \ref{cor:approximatedense}) to make arguments streamlined and more precise.

\begin{thm}\label{thm:supercanbigrepresentatives}
Let $X$ be a projective manifold with a K\"ahler form $\omega$. Let $L$ be a big $\Q$-divisor on $X$ and set $N:=\{m\in\N\mid mL\text{ is Cartier}\}$. Fix a smooth metric $h$ on $L$ and denote $\alpha:=\Theta_h(L)\in\{L\}$. Let $\varphi\in\PSH(X,\alpha)$ such that
\begin{equation}\label{eq:81}
\int_X e^{2\varphi}dV_\omega\leq 1.
\end{equation}
Then there exists a sequence of sections $\sigma_m\in H^0(X,mL)$ for $m\in N$ such that
$$\int_X|\sigma_m|^{2/m}_{h^m}dV_\omega\leq1\quad\text{and}\quad \varphi=\Big(\limsup_{m\to\infty}\log|\sigma_m|^{1/m}_{h^m}\Big)^*.$$
\end{thm}

\begin{proof}
\emph{Step 1.} 
In this step we prepare several constants that will be used throughout the proof.

Fix a finite covering $\mathcal U$ of $X$ by coordinate balls on which $K_X$ and all $mL$ trivialise, for $m\in N$. Fix constants $\delta$ and $C$, depending only on $\mathcal U$, as in Theorem \ref{thm:boundedsections}. As $L$ is big, there exist a positive constant $\varepsilon$ and $\psi\in\PSH(X,\alpha)$ such that
\begin{equation}\label{eq:8b}
\alpha+dd^c\psi\geq \varepsilon\omega.
\end{equation}
Since $\psi$ is bounded from above, by subtracting a constant from $\psi$ we may assume that
\begin{equation}\label{eq:7b11}
\psi\leq0\quad\text{and}\quad\int_X e^{2\psi}dV_\omega\leq \frac12.
\end{equation}
Let $h_\omega$ be the smooth metric on $K_X$ induced by the hermitian metric on $T_X$ whose fundamental form is $\omega$. Note that by Lemma \ref{lem:diagonalisation} there exists a constant $C_\omega>0$ such that 
\begin{equation}\label{eq:step1}
{-}\Theta_{h_\omega}(K_X)+C_\omega \omega\geq0.
\end{equation}

We fix for the remainder of the proof an integer $p>1$ such that:
\begin{enumerate}[\normalfont (i)]
\item $p\varepsilon-C_\omega\geq\delta$, 
\item $C\leq 2^{(p-1)/2}$.
\end{enumerate}
Since $\frac{p}{p-1}\psi\leq\psi$ by the first inequality in \eqref{eq:7b11}, by the second inequality in \eqref{eq:7b11} we have
\begin{equation}\label{eq:psi}
\int_X e^{2\frac{p}{p-1}\psi}dV_\omega\leq \frac12.
\end{equation}

\medskip

\emph{Step 2.}
In this step we prepare several Cartier divisors on $X$ and singular metrics on them.

Set
$$N_{\geq p}:=\{n\in N\mid n\geq p\}.$$
For each $m\in N_{\geq p}$ set 
$$L_m:=mL-K_X.$$
Then
$$h_m:=e^{{-}2(m-p)\varphi-2p\psi}h^m h_\omega^{-1}$$
is a singular metric on $L_m$ with curvature current
\begin{align*}
\Theta_{h_m}(L_m)&=m\alpha+(m-p)dd^c\varphi+p\, dd^c\psi-\Theta_{h_\omega}(K_X)\\
& \geq (p\varepsilon-C_\omega)\omega
\end{align*}
by \eqref{eq:8b}, by \eqref{eq:step1} and since $\alpha+dd^c\varphi\geq0$. This together with the property (i) from Step 1 yields
\begin{equation}\label{eq:9p}
\Theta_{h_m}(L_m)\geq\delta\omega.
\end{equation}
For each $m\in N_{\geq p}$ define the singular metric 
$$g_m:=h_m h_\omega=e^{{-}2(m-p)\varphi-2p\psi}h^m$$
on $K_X+L_m=mL$.

\medskip

\emph{Step 3.}
By Theorem \ref{thm:boundedsections}, by the choices of the constants $\delta$ and $C$ in Step 1 and by \eqref{eq:9p}, for each $m\in N_{\geq p}$ and each $x\in X\setminus\{\varphi+\psi={-}\infty\}$ there is a section 
$$\sigma_{m,x}\in H^0(X,K_X+L_m)$$
such that
\begin{equation}\label{eq:9d}
|\sigma_{m,x}(x)|_{g_m}=1\quad\text{and}\quad \|\sigma_{m,x}\|_{g_m}\leq C.
\end{equation}
For $m\in N_{\geq p}$, H\"older's inequality for conjugate exponents $\frac{1}{m}+\frac{m-p}{m}+\frac{p-1}{m}=1$ gives
\begin{align*}
\int_X |\sigma_{m,x}|^{2/m}_{h^m} dV_\omega&=\int_X\big(|\sigma_{m,x}|^2_{h^m}e^{-2(m-p)\varphi-2p\psi}\big)^{\frac{1}{m}}e^{2\frac{m-p}{m}\varphi}e^{2\frac{p}{m}\psi}dV_\omega\\
& \leq \|\sigma_{m,x}\|_{g_m}^{2/m}\bigg(\int_X e^{2\varphi}dV_\omega\bigg)^{\frac{m-p}{m}}\bigg(\int_X e^{2\frac{p}{p-1}\psi}dV_\omega\bigg)^{\frac{p-1}{m}} \\
&\leq C^{\frac{2}{m}}2^{\frac{1-p}{m}},
\end{align*}
where the last inequality follows from \eqref{eq:81}, \eqref{eq:psi} and \eqref{eq:9d}. This together with the property (ii) from Step 1 gives
\begin{equation}\label{eq:9a}
\int_X|\sigma_{m,x}|^{2/m}_{h^m} dV_\omega\leq 1 \quad\text{for all }m\in N_{\geq p}.
\end{equation}
Furthermore, for $x\in X\setminus\{\varphi+\psi={-}\infty\}$, from \eqref{eq:9d} we have
$$ 1=|\sigma_{m,x}(x)|_{g_m} = |\sigma_{m,x}(x)|_{h^m}e^{-(m-p)\varphi(x)-p\psi(x)}, $$
and thus
\begin{equation}\label{eq:9b}
\log |\sigma_{m,x}(x)|^{1/m}_{h^m}= \Big(1-\frac{p}{m}\Big)\varphi(x)+\frac{p}{m}\psi(x).
\end{equation}

\medskip

\emph{Step 4.}
Set 
$$\mathcal P:=\{z\in X\mid(\varphi+\psi)(z)={-}\infty\}.$$
Then the set $\mathcal P$ is of Lebesgue measure zero in $X$ and the set $X\setminus\mathcal P$ is dense in $X$. By Corollary \ref{cor:approximatedense} there exists a countable set 
$$\mathcal D:=\{x_q\mid q\in\N\}\subseteq X\setminus\mathcal P$$
which is dense in $X$, such that for each $z\in X$ there exists a sequence $\{z_s\}$ in $\mathcal D$ with
$$\lim\limits_{s\to\infty}z_s=z\quad\text{and}\quad \lim\limits_{s\to\infty}\varphi(z_s)=\varphi(z).$$

Fix a sequence $\{q_j\}_{j\in\N_{>0}}$ of positive integers, in which each positive integer occurs infinitely many times. For each integer $m\in N_{\geq p}$ we have $x_{q_m}\in X\setminus\{\varphi+\psi={-}\infty\}$, hence we may, by Step 3, define sections
$$\sigma_m\in H^0(X,mL)=H^0(X,K_X+L_m)$$
by
\begin{equation}\label{eq:sigmas1}
\sigma_m:=\sigma_{m,x_{q_m}},
\end{equation}
and note that $\sigma_{m,x_{q_m}}$ satisfy inequalities \eqref{eq:9a}. Set
$$u:=\limsup_{m\to\infty} \log|\sigma_m|^{1/m}_{h^m}.$$
We will show that 
$$\varphi=u^*.$$

\medskip

\emph{Step 5.}
Fix a point $x\in X$. In this step we show that
\begin{equation}\label{eq:11b}
\varphi(x)\geq u^*(x).
\end{equation}
By Lemma \ref{lem:holomorphicpsh} (applied to the $\Q$-divisor $L$ and the metric $h$) there exist constants $C_2>0$ and $r_0>0$ such that for every coordinate ball $B(x,r)$ with $r\leq r_0$ and for each integer $m\in N_{\geq p}$ we have
\begin{equation}\label{eq:11a}
|\sigma_m(x)|^2_{h^m}\leq e^{2m C_2r^2}\fint_{B(x,r)}|\sigma_m|^2_{h^m}dV_\omega.
\end{equation}
Now, since the functions $\varphi$ and $\psi$ are bounded from above on $B(x,r)$, \eqref{eq:9d} gives
\begin{align*}
\int_{B(x,r)}&|\sigma_m|^2_{h^m}dV_\omega=\int_{B(x,r)}|\sigma_m|^2_{g_m}e^{2(m-p)\varphi+2p\psi}dV_\omega \\
&\leq \|\sigma_m\|^2_{g_m}\sup_{B(x,r)}e^{2(m-p)\varphi+2p\psi}\leq C^2\sup_{B(x,r)}e^{2(m-p)\varphi+2p\psi},
\end{align*}
which together with \eqref{eq:11a} implies
$$|\sigma_m(x)|^2_{h^m}\leq \frac{e^{2m C_2r^2}n!C^2}{r^{2n}\pi^n}\sup_{B(x,r)}e^{2(m-p)\varphi+2p\psi}. $$
Plugging in $r:=\frac1m$ for $m\geq\max\{p,\frac{1}{r_0}\}$, taking logarithms of both sides and dividing by $2m$, we obtain
\begin{align*}
\log |\sigma_m(x)|^{1/m}_{h^m}&\leq \frac{C_2}{m^2}+\frac{n\log m}{m}+\frac{1}{2m}\log (n!C^2/\pi^n)\\
&+\sup_{B\left(x,\frac1m\right)}\Big(1-\frac{p}{m}\Big)\varphi+\sup_{B\left(x,\frac1m\right)}\frac{p}{m}\psi.
\end{align*}
Taking $\limsup$ as $m\to\infty$, by Lemma \ref{lem:limit}(a)(b) we obtain
$$u(x)=\limsup_{m\to\infty} \log|\sigma_m(x)|^{1/m}_{h^m}\leq \varphi(x),$$
which gives \eqref{eq:11b} since $\varphi$ is upper semicontinuous.

\medskip

\emph{Step 6.}
Fix a point $x\in X$. In this step we finally show that
$$\varphi(x)\leq u^*(x).$$

To this end, recalling the construction of the set $\mathcal D$ from Step 4, we may find a strictly increasing sequence $\{q_j'\}_{j\in\N_{>0}}$ of positive integers such that $x_{q_j'}\in\mathcal D$ for all $j$ and we have
\begin{equation}\label{eq:11c}
\lim_{j\to\infty}x_{q_j'}=x\quad\text{and}\quad \lim_{j\to\infty}\varphi(x_{q_j'})=\varphi(x).
\end{equation}
By the construction in Step 4, for each fixed $j$ there is a strictly increasing sequence $\{m_\ell\}_{\ell\in\N_{>0}}$ in the set $N_{\geq p}$ such that $q_j'=q_{m_\ell}$ for all $\ell$. Then $\sigma_{m_\ell}=\sigma_{m_\ell,x_{q_{m_\ell}}}$ by \eqref{eq:sigmas1}. Hence, by \eqref{eq:9b} and since $\psi(x_{q_j'})\neq{-}\infty$ by the construction of $\mathcal D$ we have
\begin{align*}
u(x_{q_j'})&\geq \limsup_{\ell\to\infty}\log |\sigma_{m_\ell}(x_{q_j'})|^{1/m_\ell}_{h^{m_\ell}}\\
&=\limsup_{\ell\to\infty}\log |\sigma_{m_\ell,x_{q_{m_\ell}}}(x_{q_{m_\ell}})|^{1/m_\ell}_{h^{m_\ell}}= \varphi(x_{q_j'}).
\end{align*}
Then this last inequality and \eqref{eq:11c} give
\begin{align*}
u^*(x)\geq\limsup_{j\to\infty}u(x_{q_j'})\geq \limsup_{j\to\infty}\varphi(x_{q_j'})=\varphi(x),
\end{align*}
which finishes the proof.
\end{proof}

\begin{rem}\label{rem:depends}
In the proof of Theorem \ref{thm:supercanbigrepresentatives} the auxiliary quasi-psh function $\psi$ had to be introduced for two reasons: (a) to create a singular metric on each $L_m$ whose curvature current is sufficiently positive, and (b) to be able to prove inequality \eqref{eq:9a}. The positive integer $p$ in the proof of Theorem \ref{thm:supercanbigrepresentatives} does not depend on $\varphi$ nor on any integer $m$ in the proof, but it does depend on the choice of $\psi$.
\end{rem}

The following main result of this section is inspired by \cite[Remark 5.23]{BD12}.

\begin{thm}\label{thm:supercanbig}
Let $X$ be a projective manifold with a K\"ahler form $\omega$. Let $L$ be a big $\Q$-divisor on $X$ and set $N:=\{m\in\N\mid mL\text{ is Cartier}\}$. Fix a smooth metric $h$ on $L$ and denote $\alpha:=\Theta_h(L)\in\{L\}$. For each $m\in N$ set
$$\textstyle V_{h,m}:=\big\{\sigma\in H^0(X,mL)\mid \int_X|\sigma|^{2/m}_{h^m}dV_\omega\leq1\big\}$$
and
$$ \varphi_{h,m}:=\sup_{\sigma\in V_{h,m}}\log|\sigma|^{1/m}_{h^m}.$$
Then:
\begin{enumerate}[\normalfont (i)]
\item \label{enu:1} for each $m\in N$ and each $\sigma\in H^0(X,mL)$ there exists a positive real number $\lambda$ such that $\lambda\sigma\in V_{h,m}$,
\item \label{enu:2} there exists a constant $C$ such that $\log|\sigma|^{1/m}_{h^m}\leq C$ for each $m\in N$ and each $\sigma\in V_{h,m}$,
\item \label{enu:3} $V_{h,m}$ is compact in $H^0(X,mL)$ for each $m\in N$,
\item \label{enu:4} $\varphi_{h,m}=\max\limits_{\sigma\in V_{h,m}}\log|\sigma|^{1/m}_{h^m}$ for each $m\in N$,
\item \label{enu:5} $\varphi_{h,m}\in\PSH(X,\alpha)$ for each $m\in N$,
\item \label{enu:7} $\varphi_{h,km}\geq\varphi_{h,m}$ for each $m\in N$ and each positive integer $k$,
\item \label{enu:8} the supercanonical potential of $L$ associated to $\alpha$ is 
$$\varphi_{\alpha,\can}=\Big(\sup_{m\in N}\varphi_{h,m}\Big)^*,$$
\item \label{enu:9} the sequence $\{\varphi_{h,m}\}_{m\in N}$ converges to $\varphi_{\alpha,\can}$ in $L^1_\loc(X)$,
\item \label{enu:13} $\mathcal I(\varphi_{h,m})=\mathcal I(\varphi_{\alpha,\can})$ for all $m\in N$ sufficiently large,
\item \label{enu:6} $\varphi_{h,m}$ is continuous on $X\setminus\sB(L)$ for all $m\in N$ sufficiently divisible,
\item \label{enu:10} the sequence $\{\varphi_{h,m}\}_{m\in N}$ converges uniformly on compact subsets of $X\setminus\sB_+(L)$ to $\varphi_{\alpha,\can}$,
\item \label{enu:11} $\varphi_{\alpha,\can}$ is bounded on $X\setminus\sB(L)$, it is continuous on $X\setminus\sB_+(L)$, and 
$$\varphi_{\alpha,\can}=\sup\limits_{m\in N}\varphi_{h,m}\quad\text{on }X\setminus\sB_+(L).$$
\end{enumerate} 
\end{thm}

\begin{proof}
Set
$$\textstyle \mathcal S_\alpha:=\big\{\varphi\in\PSH(X,\alpha)\mid \int_X e^{2\varphi}dV_\omega\leq 1\big\},$$
and recall from Lemma \ref{lem:supercanonical} that the supercanonical potential associated to $\alpha$ was defined as
$$\varphi_{\alpha,\can}(x):=\sup_{\varphi\in\mathcal S_\alpha}\varphi(x) \quad\text{for }x\in X.$$

\emph{Step 1.}
Let $m\in N$ and $\sigma\in H^0(X,mL)$. Then $\log|\sigma|^{1/m}_{h^m}\in\PSH(X,\alpha)$ by Example \ref{exa:quasi-psh}(b). Moreover, since $|\sigma|_{h^m}$ is bounded from above on $X$, there exists a constant $C_\sigma>0$ such that 
$$\int_X|\sigma|^{2/m}_{h^m}dV_\omega\leq C_\sigma,$$
hence $C_\sigma^{{-}m/2}\sigma\in V_{h,m}$, which gives \eqref{enu:1}. If $\sigma\in V_{h,m}$, then 
$$\int_X e^{2\log|\sigma|^{1/m}_{h^m}}dV_\omega=\int_X|\sigma|^{2/m}_{h^m}dV_\omega\leq1,$$
thus
\begin{equation}\label{eq:include}
\big\{\log|\sigma|^{1/m}_{h^m}\mid\sigma\in V_{h,m}\big\}\subseteq\mathcal S_\alpha.
\end{equation}
Then \eqref{enu:2} follows from Lemma \ref{lem:supercanonical}(b).

\medskip

\emph{Step 2.}
Define the norm $\|\cdot\|_{\max}$ on $H^0(X,mL)$ by
$$\|s\|_{\max}:=\sup_X|s|_{h^m}\quad\text{ for }s\in H^0(X,mL).$$
Consider a sequence $\{\sigma_\ell\}$ in $V_{h,m}$. Since $V_{h,m}$ is bounded in $H^0(X,mL)$ with respect to the norm $\|\cdot\|_{\max}$ by \eqref{enu:2}, by passing to a subsequence we may assume that the sequence $\{\sigma_\ell\}$ converges to a section $\sigma\in H^0(X,mL)$. To prove \eqref{enu:3} it suffices to show that $\sigma\in V_{h,m}$.

Note that $\log|\sigma|_{h^m}$, as well as all the functions $\log|\sigma_\ell|_{h^m}$, belong to $\PSH(X,m\alpha)$ by Example \ref{exa:quasi-psh}(b). By \eqref{enu:2} and by Theorem \ref{thm:compactnessquasipsh}(b), after passing to a subsequence we may assume that the sequence $\{\log|\sigma_\ell|_{h^m}\}$ converges in $L^1_\loc(X)$ and almost everywhere to a function $\varphi\in\PSH(X,m\alpha)$. Thus $\varphi=\log|\sigma|_{h^m}$ almost everywhere, hence everywhere by Corollary \ref{cor:strongusc}. But then $\log|\sigma|_{h^m}\in\mathcal S_{m\alpha}$ by Fatou's lemma, hence $\sigma\in V_{h,m}$, as desired.

\medskip

\emph{Step 3.}
Now we show \eqref{enu:4}. Fix $x\in X$. Then there exists a sequence of sections $\sigma_j\in V_{h,m}$ such that $\lim\limits_{j\to\infty}\log|\sigma_j(x)|_{h^m}^{1/m}=\varphi_{h,m}(x)$. By \eqref{enu:3} and by passing to a subsequence we may assume that there exists $\sigma\in V_{h,m}$ such that $\lim\limits_{j\to\infty}\sigma_j=\sigma$. Thus $\varphi_{h,m}(x)=\log|\sigma(x)|^{1/m}_{h^m}$.

\medskip

\emph{Step 4.}
Next we prove \eqref{enu:5}. By \eqref{enu:2} and Theorem \ref{thm:compactnessquasipsh}(a) we have that $(\varphi_{h,m})^*\in\PSH(X,\alpha)$. Fix $x\in X$. As $(\varphi_{h,m})^*(x)=\limsup\limits_{z\to x}\varphi_{h,m}(z)$, by \eqref{enu:4} there exists a sequence of sections $\sigma_j\in V_{h,m}$ and a sequence of points $x_j\in X$ such that 
$$\lim\limits_{j\to\infty}x_j=x\quad\text{and}\quad \lim\limits_{j\to\infty}\log|\sigma_j(x_j)|_{h^m}^{1/m}=(\varphi_{h,m})^*(x).$$
By \eqref{enu:3} and by passing to a subsequence we may assume that there exists $\sigma\in V_{h,m}$ such that $\lim\limits_{j\to\infty}\sigma_j=\sigma$. Then Lemma \ref{lem:compactuniform}(b) gives
$$(\varphi_{h,m})^*(x)=\lim_{j\to\infty}\log|\sigma_j(x_j)|_{h^m}^{1/m}=\log|\sigma(x)|_{h^m}^{1/m}\leq\varphi_{h,m}(x),$$
which shows that $\varphi_{h,m}=(\varphi_{h,m})^*$. Thus, $\varphi_{h,m}$ is $\alpha$-psh.

\medskip

\emph{Step 5.}
For \eqref{enu:7}, fix $x\in X$, $m\in N$ and a positive integer $k$. By \eqref{enu:4} there exists $\sigma\in V_{h,m}$ such that $\varphi_{h,m}(x)=\log|\sigma(x)|^{1/m}_{h^m}$. Since $|\sigma^k|^{1/mk}_{h^{mk}}=|\sigma|^{1/m}_{h^m}$, we have $\sigma^k\in V_{h,mk}$, and hence
$$\varphi_{h,km}(x)\geq\log|\sigma^k(x)|^{1/mk}_{h^{mk}}=\log|\sigma(x)|^{1/m}_{h^m}=\varphi_{h,m}(x),$$
which was to be shown.

\medskip

\emph{Step 6.}
By \eqref{enu:2} and by Theorem \ref{thm:compactnessquasipsh}(a) we have
$$\varphi_{h,\alg}:=\Big(\sup_{m\in N}\varphi_{h,m}\Big)^*\in\PSH(X,\alpha).$$
It is immediate that $\varphi_{h,\alg}\leq\varphi_{\alpha,\can}$ by \eqref{eq:include}. For the reverse inequality, let $\varphi\in\mathcal S_\alpha$. Then by Theorem \ref{thm:supercanbigrepresentatives} there exists a sequence of sections $\tau_m\in V_{h,m}$ for $m\in N$ such that
$$\varphi=\Big(\limsup_{m\to\infty}\log|\tau_m|^{1/m}_{h^m}\Big)^*.$$
Since $\log|\tau_m|^{1/m}_{h^m}\leq\varphi_{h,m}$ for each $m\in N$ by the definition of $\varphi_{h,m}$, we obtain
$$\varphi\leq\Big(\limsup_{m\to\infty}\varphi_{h,m}\Big)^*\leq\Big(\sup_{m\in N}\varphi_{h,m}\Big)^*=\varphi_{h,\alg},$$
hence
$$\varphi_{\alpha,\can}=\sup_{\varphi\in\mathcal S_\alpha}\varphi\leq\varphi_{h,\alg}.$$
This shows \eqref{enu:8}. Part \eqref{enu:9} follows from \eqref{enu:7} and from Theorem \ref{thm:compactnessquasipsh}(d).

\medskip

\emph{Step 7.}
Part \eqref{enu:13} follows from \eqref{enu:8} and \eqref{enu:9} and from Theorem \ref{thm:GuanZhou}.

\medskip

\emph{Step 8.}
In this step we prove \eqref{enu:6}. We first show that $\varphi_{h,m}\neq{-}\infty$ away from $\sB(L)$ for all $m\in N$ sufficiently divisible. Indeed, by Remark \ref{rem:stableQ} we have $\Bs|mL|=\sB(L)$ for all $m$ sufficiently divisible. Therefore, for each point $x\in X\setminus\sB(L)$ there exists $\sigma\in H^0(X,mL)$ such that $\sigma(x)\neq0$. By \eqref{enu:1} there exists a positive real number $\lambda$ such that $\lambda\sigma\in V_{h,m}$, hence $\varphi_{h,m}(x)\geq \log|\lambda\sigma(x)|^{1/m}_{h^m}>{-}\infty$.

Fix one such sufficiently divisible $m$. Fix $x_0\in X\setminus\sB(L)$ and a sequence of points $\{x_j\}$ in $X\setminus\sB(L)$ such that $\lim\limits_{j\to\infty}x_j=x_0$. By sequential continuity it suffices to show that
\begin{equation}\label{eq:equalityinthelimit}
\lim\limits_{j\to\infty}\varphi_{h,m}(x_j)=\varphi_{h,m}(x_0).
\end{equation}
To that end, set
$$a:=\limsup\limits_{j\to\infty}\varphi_{h,m}(x_j),$$
and note that $a\neq+\infty$ since $\varphi_{h,m}$ are uniformly bounded from above by \eqref{enu:2}. By \eqref{enu:4} there exists $\sigma_0\in V_{\alpha,m}$ such that $\varphi_{\alpha,m}(x_0)=\log|\sigma_0(x_0)|_{h^m}^{1/m}$.

Note first that $\varphi_{h,m}(x_j)\geq\log|\sigma_0(x_j)|_{h^m}^{1/m}$ by the definition of $\varphi_{\alpha,m}$, hence
\begin{equation}\label{eq:liminf}
a\geq\liminf_{j\to\infty}\log|\sigma_0(x_j)|_{h^m}^{1/m}=\log|\sigma_0(x_0)|_{h^m}^{1/m}=\varphi_{h,m}(x_0).
\end{equation}
We will now show that $a\leq \varphi_{h,m}(x_0)$, which together with \eqref{eq:liminf} will then prove \eqref{eq:equalityinthelimit}. By passing to a subsequence of $\{x_j\}$ we may assume that $a=\lim\limits_{j\to\infty}\varphi_{h,m}(x_j)$. By \eqref{enu:4}, for each $j\in\N$ there exists $\sigma_j\in V_{h,m}$ such that $\varphi_{h,m}(x_j)=\log|\sigma_j(x_j)|_{h^m}^{1/m}$. By \eqref{enu:3} and by passing to a subsequence we may assume that there exists $\widetilde\sigma\in V_{h,m}$ such that $\lim\limits_{j\to\infty}\sigma_j=\widetilde\sigma$. Then Lemma \ref{lem:compactuniform}(b) gives
$$a=\lim_{j\to\infty}\log|\sigma_j(x_j)|_{h^m}^{1/m}=\log|\widetilde\sigma(x_0)|_{h^m}^{1/m}\leq\varphi_{h,m}(x_0).$$
This concludes the proof of \eqref{enu:6}.

\medskip

\emph{Step 9.}
In this step we prove \eqref{enu:10}. To this end, we use the notation from the proof of Theorem \ref{thm:supercanbigrepresentatives}. We first note that, by Corollary \ref{cor:logarithmic} we may and do choose the function $\psi$ as in the proof of Theorem \ref{thm:supercanbigrepresentatives} such that it has logarithmic poles which all lie in $\sB_+(L)$. 

Let $\varphi\in\mathcal S_\alpha$ and $m\in N$. Then by \eqref{eq:9a}, \eqref{eq:9b} and Remark \ref{rem:depends} there exist a positive integer $p$ (independent of $\varphi$ and $m$) and sections $\sigma_{m,x}\in V_{h,m}$ for each $x\in X\setminus\{\varphi+\psi={-}\infty\}$ such that
$$ \log |\sigma_{m,x}(x)|^{1/m}_{h^m}= \Big(1-\frac{p}{m}\Big)\varphi(x)+\frac{p}{m}\psi(x), $$
hence by the definition of the function $\varphi_{h,m}$, for each $x\in X\setminus\{\varphi+\psi={-}\infty\}$ we obtain
$$ \varphi_{h,m}(x)\geq \Big(1-\frac{p}{m}\Big)\varphi(x)+\frac{p}{m}\psi(x). $$
This inequality holds trivially when $\varphi(x)={-}\infty$ or $\psi(x)={-}\infty$, hence it holds for all $x\in X$. By the definition of $\varphi_{\alpha,\can}$ this then implies
\begin{equation}\label{eq:mandcan}
\varphi_{h,m}(x)\geq\Big(1-\frac{p}{m}\Big)\varphi_{\alpha,\can}(x)+\frac{p}{m}\psi(x)\quad\text{for all }x\in X.
\end{equation}
Note that $\varphi_{h,m}\neq{-}\infty$ on $X\setminus\sB(L)$ by \eqref{enu:6}, hence $\varphi_{\alpha,\can}\neq{-}\infty$ on $X\setminus\sB(L)$ by \eqref{enu:8}. As $\psi\neq{-}\infty$ on $X\setminus\sB_+(L)$ by construction, we have
$$ 0\leq \varphi_{\alpha,\can}(x)-\varphi_{h,m}(x)\leq\frac{p}{m}\big(\varphi_{\alpha,\can}(x)-\psi(x)\big) \quad\text{for }x\in X\setminus\sB_+(L) $$
by \eqref{enu:8} and \eqref{eq:mandcan}. Since $\psi$ is smooth on $X\setminus\sB_+(L)$ and $\varphi_{\alpha,\can}$ is bounded from above, we conclude that for each compact set $K\subseteq X\setminus\sB_+(L)$ there exists a constant $C_K>0$ such that
$$ 0\leq \varphi_{\alpha,\can}(x)-\varphi_{h,m}(x)\leq\frac{C_K}{m}, $$
so \eqref{enu:10} follows.

\medskip

\emph{Step 10.}
Finally, the first part of \eqref{enu:11} was already noticed in Step 9, the second part of \eqref{enu:11} follows from \eqref{enu:6} and \eqref{enu:10} by the uniform convergence theorem, and then the third part of \eqref{enu:11} follows from the second part of \eqref{enu:11} and from \eqref{enu:8}.
\end{proof}

\section{Proof of Theorem \ref{thm:main2}}\label{sec:proofofMain2}

In this section we prove the second main result of this paper, Theorem \ref{thm:main2}. The first technical result is Theorem \ref{thm:supercanpsefrepresentatives}, which is a pseudoeffective analogue of Theorem \ref{thm:supercanbigrepresentatives}. A related, but somewhat more involved statement for klt pairs was mentioned without proof in \cite[Generalization 5.24]{BD12}. The proof is similar, but somewhat more involved than that of Theorem \ref{thm:supercanbigrepresentatives}, and we provide all the details. In particular, the very precise conclusion of Theorem \ref{thm:supercanpsefrepresentatives} will be needed in the proof of Theorem \ref{thm:supercanpsef}.

\begin{thm}\label{thm:supercanpsefrepresentatives}
Let $X$ be a projective manifold. Let $L$ be a pseudoeffective $\Q$-divisor on $X$ and set $N:=\{m\in\N\mid mL\text{ is Cartier}\}$. Fix a smooth metric $h$ on $L$ and denote $\alpha:=\Theta_h(L)\in\{L\}$. Let $A$ be an ample Cartier divisor on $X$, fix a K\"ahler form $\omega\in\{A\}$, and let $h_A$ be a smooth metric on $A$ such that $\omega=\Theta_{h_A}(A)$. For each positive integer $\ell$ denote $L_\ell:=L+\frac1\ell A$ and denote by $h_\ell:=hh_A^{1/\ell}$ the smooth metric on $L_\ell$.

Then there exists a positive integer $p$ such that the following holds. Let $\varphi\in\PSH(X,\alpha)$ such that
\begin{equation}\label{eq:81a}
\int_X e^{2\varphi}dV_\omega\leq 1.
\end{equation}
Then for any sequence $\{m_\ell\}_{\ell\in\N_{>0}}$ satisfying $m_\ell\in N\cap\ell\N$, $m_\ell\geq 2p\ell$ and $\lim\limits_{\ell\to\infty}\ell/m_\ell=0$, there exists a sequence of sections $\sigma_\ell\in H^0(X,m_\ell L_\ell)$ such that
$$\int_X|\sigma_\ell|^{2/m_\ell}_{h_\ell^{m_\ell}}dV_\omega\leq1\quad\text{and}\quad\varphi=\Big(\limsup_{\ell\to\infty}\log|\sigma_\ell|^{1/m_\ell}_{h_\ell^{m_\ell}}\Big)^*.$$
\end{thm}

\begin{proof}
\emph{Step 1.} 
In this step we prepare several constants that will be used throughout the proof.

Fix a finite covering $\mathcal U$ of $X$ by coordinate balls on which $K_X$, $A$ and all $mL$ trivialise, for $m\in N$. Fix constants $\delta$ and $C$, depending only on $\mathcal U$, as in Theorem \ref{thm:boundedsections}. For every positive integer $\ell$ let 
$$\alpha_\ell:=\alpha+\frac1\ell\omega\in\{L_\ell\},$$
and note that
\begin{equation}\label{eq:8b12}
\alpha_\ell+dd^c\varphi\geq\frac{1}{\ell}\omega.
\end{equation}
Since for each $\ell$ the divisor $L_\ell$ is big, by Corollary \ref{cor:logarithmic} there exist functions $\psi_\ell\in\PSH(X,\alpha_\ell)$ with logarithmic singularities such that $\{\psi_\ell={-}\infty\}\subseteq\sB_+(L_\ell)$ and
\begin{equation}\label{eq:8b1}
\alpha_\ell+dd^c\psi_\ell\geq 0.
\end{equation}
Since each $\psi_\ell$ is bounded from above, by subtracting constants from $\psi_\ell$ we may assume that for each $\ell$ we have
\begin{equation}\label{eq:0h1}
\psi_\ell\leq0\quad\text{and}\quad\int_X e^{2\psi_\ell}dV_\omega\leq\frac12.
\end{equation}
Let $h_\omega$ be the smooth metric on $K_X$ induced by the hermitian metric on $T_X$ whose fundamental form is $\omega$. Note that by Lemma \ref{lem:diagonalisation} there exists a constant $C_\omega>0$ such that 
\begin{equation}\label{eq:step1a}
{-}\Theta_{h_\omega}(K_X)+C_\omega \omega\geq0.
\end{equation}

We fix for the remainder of the proof an integer $p>1$ such that:
\begin{enumerate}[\normalfont (i)]
\item $p-C_\omega\geq\delta$,
\item $C\leq 2^{(p-1)/2}$.
\end{enumerate}
Set
$$p_\ell:=p\ell\quad\text{for each positive integer }\ell.$$
Since $\frac{p_\ell}{p_\ell-1}\psi_\ell\leq\psi_\ell$ for each $\ell$ by the first inequality in \eqref{eq:0h1}, by the second inequality in \eqref{eq:0h1} we have
\begin{equation}\label{eq:psi_ell}
\int_X e^{2\frac{p_\ell}{p_\ell-1}\psi_\ell}dV_\omega\leq \frac12\quad\text{for all positive integers }\ell.
\end{equation}

\medskip

\emph{Step 2.}
In this step we prepare several Cartier divisors on $X$ and singular metrics on them.

Set
$$N_{\geq 2p_\ell}:=\{n\in N\cap\ell\N\mid n\geq 2p_\ell\}.$$
For each $m\in N_{\geq 2p_\ell}$ set 
$$L_{m,\ell}:=mL_\ell-K_X.$$
Then
$$h_{m,\ell}:=e^{{-}2(m-p_\ell)\varphi-2p_\ell\psi_\ell}h_\ell^m h_\omega^{-1}$$
is a singular metric on $L_{m,\ell}$ with curvature current
\begin{align*}
\Theta_{h_{m,\ell}}(L_{m,\ell})&=m\alpha_\ell+(m-p_\ell)dd^c\varphi+p_\ell\, dd^c\psi_\ell-\Theta_{h_\omega}(K_X)\\
& \geq\frac{m-p_\ell}{\ell}\omega-\Theta_{h_\omega}(K_X) \geq (p-C_\omega)\omega
\end{align*}
by \eqref{eq:8b12}, \eqref{eq:8b1} and \eqref{eq:step1a}, and since $m\geq 2p_\ell=2p\ell$. This together with the property (i) from Step 1 yields
\begin{equation}\label{eq:9p1}
\Theta_{h_{m,\ell}}(L_{m,\ell})\geq\delta\omega.
\end{equation}
For each $m\in N_{\geq 2p_\ell}$, define the singular metric 
$$g_{m,\ell}:=h_{m,\ell}h_\omega=e^{{-}2(m-p_\ell)\varphi-2p_\ell\psi_\ell}h_\ell^m$$
on $K_X+L_{m,\ell}=mL_\ell$. 

\medskip

\emph{Step 3.}
By Theorem \ref{thm:boundedsections}, by the choices of the constants $\delta$ and $C$ in Step 1 and by \eqref{eq:9p1}, for each $m\in N_{\geq 2p_\ell}$ and each $x\in X\setminus\{\varphi+\psi_\ell={-}\infty\}$ there is a section 
$$\sigma_{m,\ell,x}\in H^0(X,K_X+L_{m,\ell})$$
such that
\begin{equation}\label{eq:9d1}
|\sigma_{m,\ell,x}(x)|_{g_{m,\ell}}=1\quad\text{and}\quad \|\sigma_{m,\ell,x}\|_{g_{m,\ell}}\leq C.
\end{equation}
Similarly as in Step 3 of the proof of Theorem \ref{thm:supercanbigrepresentatives}, for $m\in N_{\geq 2p_\ell}$, by H\"older's inequality for conjugate exponents $\frac{1}{m}+\frac{m-p_\ell}{m}+\frac{p_\ell-1}{m}=1$ and by \eqref{eq:81a}, \eqref{eq:psi_ell} and \eqref{eq:9d1} we obtain
$$ \int_X |\sigma_{m,\ell,x}|^{2/m}_{h_\ell^m} dV_\omega\leq C^{\frac{2}{m}}2^{\frac{1-p_\ell}{m}}\leq C^{\frac{2}{m}}2^{\frac{1-p}{m}}. $$
This together with the property (ii) from Step 1 gives
\begin{equation}\label{eq:9a1}
\int_X|\sigma_{m,\ell,x}|^{2/m}_{h_\ell^m} dV_\omega\leq 1 \quad\text{for all }m\in N_{\geq 2p_\ell}.
\end{equation}
Furthermore, for $x\in X\setminus\{\varphi+\psi_\ell={-}\infty\}$, from \eqref{eq:9d1} we have
$$ 1=|\sigma_{m,\ell,x}(x)|_{g_{m,\ell}} = |\sigma_{m,\ell,x}(x)|_{h_\ell^m}e^{-(m-p_\ell)\varphi(x)-p_\ell\psi_\ell(x)}, $$
and thus
\begin{equation}\label{eq:9b1}
\log |\sigma_{m,\ell,x}(x)|^{1/m}_{h_\ell^m}= \Big(1-\frac{p_\ell}{m}\Big)\varphi(x)+\frac{p_\ell}{m}\psi_\ell(x).
\end{equation}

\medskip

\emph{Step 4.}
Set 
$$\mathcal P:=\{z\in X\mid\varphi(z)={-}\infty\}\cup\bigcup_{\ell\in\N_{>0}}\{z\in X\mid\psi_\ell(z)={-}\infty\},$$
and note that the function $\varphi$, as well as all the functions $\psi_\ell$, are $\alpha_1$-psh. Therefore, $\mathcal P$ is a pluripolar set by Remark \ref{rem:pluripolarcountable}, hence $\mathcal P$ is of Lebesgue measure zero in $X$ and $X\setminus\mathcal P$ is dense in $X$. By Corollary \ref{cor:approximatedense} there exists a countable set 
$$\mathcal D:=\{x_q\mid q\in\N\}\subseteq X\setminus\mathcal P$$
which is dense in $X$, such that for each $z\in X$ there exists a sequence $\{z_s\}$ in $\mathcal D$ with
$$\lim\limits_{s\to\infty}z_s=z\quad\text{and}\quad \lim\limits_{s\to\infty}\varphi(z_s)=\varphi(z).$$

Fix a sequence $\{q_j\}_{j\in\N_{>0}}$ of positive integers, in which each positive integer occurs infinitely many times. Fix an arbitrary sequence $\{m_\ell\}_{\ell\in\N_{>0}}$ such that
\begin{equation}\label{eq:rapidgrowth}
m_\ell\in N_{\geq 2p_\ell}\text{ for each }\ell,\text{ and }\lim_{\ell\to\infty}{\ell/m_\ell}=0.
\end{equation}
Since $\{\varphi+\psi_\ell={-}\infty\}\subseteq\mathcal P$, we have $x_{q_\ell}\in X\setminus\{\varphi+\psi_\ell={-}\infty\}$ for each $\ell$, hence by Step 3 we may define sections 
$$\sigma_{m_\ell,\ell}\in H^0(X,m_\ell L_\ell)=H^0(X,K_X+L_{m_\ell,\ell})$$
by
\begin{equation}\label{eq:definesections}
\sigma_{m_\ell,\ell}:=\sigma_{{m_\ell,\ell},x_{q_\ell}},
\end{equation}
and note that $\sigma_{{m_\ell,\ell},x_{q_\ell}}$ satisfy inequalities \eqref{eq:9a1}. Set
$$u:=\limsup_{\ell\to\infty} \log|\sigma_{m_\ell,\ell}|^{1/m_\ell}_{h_\ell^{m_\ell}}.$$
We will show that 
$$\varphi=u^*.$$

\medskip

\emph{Step 5.}
Fix a point $x\in X$. In this step we show that
\begin{equation}\label{eq:11b1}
\varphi(x)\geq u^*(x).
\end{equation}
By Lemma \ref{lem:holomorphicpsh} (applied to the $\Q$-divisors $L_\ell$, the smooth metrics metric $h_\ell$ and the associated curvature forms $\alpha_\ell$) there exist constants $C_2>0$ and $r_0>0$ such that for every coordinate ball $B(x,r)$ with $r\leq r_0$, for each positive integer $\ell$ and for each integer $m_\ell\in N_{\geq 2p_\ell}$ we have
\begin{align}\label{eq:11a1}
|\sigma_{m_\ell,\ell}(x)|^2_{h_\ell^{m_\ell}}\leq e^{2m_\ell C_2r^2}\fint_{B(x,r)}|\sigma_{m_\ell,\ell}|^2_{h_\ell^{m_\ell}}dV_\omega.
\end{align}
Then similarly as in Step 5 of the proof of Theorem \ref{thm:supercanbigrepresentatives}, from \eqref{eq:9d1} and \eqref{eq:11a1} we obtain
$$|\sigma_{m_\ell,\ell}(x)|^2_{h_\ell^{m_\ell}}\leq \frac{e^{2m_\ell C_2r^2}n!C^2}{r^{2n}\pi^n}\sup_{B(x,r)}e^{2(m_\ell-p_\ell)\varphi+2p_\ell\psi_\ell}. $$
Plugging in $r:=\frac{1}{m_\ell}$ for $m_\ell\geq\max\{p_\ell,\frac{1}{r_0}\}$, taking logarithms of both sides, dividing by $2m_\ell$, and taking $\limsup$ as $\ell\to\infty$, as in Step 5 of the proof of Theorem \ref{thm:supercanbigrepresentatives} we obtain
$$u(x)=\limsup_{\ell\to\infty}\log|\sigma_{m_\ell,\ell}(x)|^{1/m_\ell}_{h_\ell^{m_\ell}}\leq \varphi(x),$$
which gives \eqref{eq:11b1} since $\varphi$ is upper semicontinuous.

\medskip

\emph{Step 6.}
Fix a point $x\in X$. In this step we finally show that
$$\varphi(x)\leq u^*(x).$$

To this end, recalling the construction of the set $\mathcal D$ from Step 4, we may find a strictly increasing sequence $\{q_j'\}_{j\in\N_{>0}}$ of positive integers such that $x_{q_j'}\in\mathcal D$ for all $j$ and we have
\begin{equation}\label{eq:11c1}
\lim_{j\to\infty}x_{q_j'}=x\quad\text{and}\quad \lim_{j\to\infty}\varphi(x_{q_j'})=\varphi(x).
\end{equation}
By the construction in Step 4, for each fixed $j$ there is a strictly increasing sequence $\{\ell_s\}_{s\in\N_{>0}}$ of positive integers such that $q_j'=q_{\ell_s}$ for all $s$ and $\sigma_{m_{\ell_s},\ell_s}=\sigma_{m_{\ell_s},\ell_s,x_{q_{\ell_s}}}$ by \eqref{eq:definesections}. Hence by \eqref{eq:9b1}, since $\lim\limits_{s\to\infty}(p_{\ell_s}/m_{\ell_s})=0$ by \eqref{eq:rapidgrowth}, and since $\psi_{\ell_s}(x_{q_j'})\neq{-}\infty$ for all $s$ by the construction of $\mathcal D$, we have
\begin{align*}
u(x_{q_j'})&\geq \limsup_{s\to\infty}\frac{1}{m_{\ell_s}}\log |\sigma_{m_{\ell_s},\ell_s}(x_{q_j'})|_{h_{\ell_s}^{m_{\ell_s}}}\\
&=\limsup_{s\to\infty}\frac{1}{m_{\ell_s}}\log |\sigma_{m_{\ell_s},\ell_s,x_{q_{\ell_s}}}(x_{q_{\ell_s}})|_{h_{\ell_s}^{m_{\ell_s}}}= \varphi(x_{q_j'}).
\end{align*}
Then this last inequality and \eqref{eq:11c1} give
\begin{align*}
u^*(x)\geq\limsup_{j\to\infty}u(x_{q_j'})\geq \limsup_{j\to\infty}\varphi(x_{q_j'})=\varphi(x),
\end{align*}
which finishes the proof.
\end{proof}

The following is the main result of this section.

\begin{thm}\label{thm:supercanpsef}
Let $X$ be a projective manifold. Let $L$ be a pseudoeffective $\Q$-divisor on $X$, fix a smooth metric $h$ on $L$ and denote $\alpha:=\Theta_h(L)\in\{L\}$. Let $A$ be an ample Cartier divisor on $X$, fix a K\"ahler form $\omega\in\{A\}$, and let $h_A$ be a smooth metric on $A$ such that $\omega=\Theta_{h_A}(A)$. For each positive integer $\ell$ denote $L_\ell:=L+\frac1\ell A$, denote by $h_\ell:=hh_A^{1/\ell}$ the smooth metric on $L_\ell$, let $\alpha_\ell:=\alpha+\frac1\ell\omega\in\{L_\ell\}$ be the corresponding smooth form, and set $N_\ell:=\{m\in\N\mid mL_\ell\text{ is Cartier}\}$. For each $m\in N_\ell$ set
$$\textstyle V_{h_\ell,m}:=\big\{\sigma\in H^0(X,mL_\ell)\mid \int_X|\sigma|^{2/m}_{h_\ell^m}dV_\omega\leq1\big\}$$
and
$$ \varphi_{h_\ell,m}:=\sup_{\sigma\in V_{h_\ell,m}}\log|\sigma|^{1/m}_{h_\ell^m}.$$
Then:
\begin{enumerate}[\normalfont (i)]
\item \label{enu:a} $\varphi_{h_\ell,m}\in\PSH(X,\alpha_\ell)$ for each $\ell$ and $m\in N_\ell$, and
$$\varphi_{h_\ell,m}=\max\limits_{\sigma\in V_{h_\ell,m}}\log|\sigma|^{1/m}_{h_\ell^m},$$ 
\item \label{enu:a1} for each $\ell$ the sequence $\{\varphi_{h_\ell,m}\}_{m\in N_\ell}$ is non-decreasing,
\item \label{enu:c} for each $\ell$ the supercanonical potential $\varphi_{\alpha_\ell,\can}$ of $L_\ell$ associated to $\alpha_\ell$ is
$$\varphi_{\alpha_\ell,\can}=\Big(\sup_{m\in N_\ell}\varphi_{h_\ell,m}\Big)^*,$$
\item \label{enu:b} for all $\ell$ and all $m\in N_\ell$, the functions $\varphi_{h_\ell,m}$ are uniformly bounded from above,
\item \label{enu:d} there exists a positive integer $p$ such that for any sequence $\{m_\ell\}_{\ell\in\N_{>0}}$ with the properties that $m_\ell\in N_\ell$, $m_\ell\geq 2p\ell$ and $\lim\limits_{\ell\to\infty}\ell/m_\ell=0$, the supercanonical potential of $L$ associated to $\alpha$ is
$$\varphi_{\alpha,\can}=\Big(\limsup_{\ell\to\infty}\varphi_{h_\ell,m_\ell}\Big)^*,$$
\item \label{enu:e} for positive integers $\ell'\geq\ell$ we have
$$\varphi_{\alpha_\ell,\can}\geq\varphi_{\alpha_{\ell'},\can}\geq\varphi_{\alpha,\can},$$
\item \label{enu:f0}
for each $\ell$ the function $\varphi_{\alpha_\ell,\can}$ is bounded on $X\setminus\sB(L)$, it is continuous on $X\setminus\sB_+(L_\ell)$, and 
$$\varphi_{\alpha_\ell,\can}=\sup\limits_{m\in N_\ell}\varphi_{h_\ell,m}\quad\text{on }X\setminus\sB_+(L_\ell),$$
\item \label{enu:f} we have
$$\varphi_{\alpha,\can}=\lim_{\ell\to\infty}\varphi_{\alpha_\ell,\can}.$$
\end{enumerate} 
\end{thm}

\begin{proof}
For a smooth $(1,1)$-form $\theta$ on $X$ whose class $\{\theta\}\in H^{1,1}(X,\R)$ is pseudoeffective, set
$$\textstyle \mathcal S_\theta:=\big\{\varphi\in\PSH(X,\theta)\mid \int_X e^{2\varphi}dV_\omega\leq 1\big\},$$
and recall from Lemma \ref{lem:supercanonical} that the supercanonical potential associated to $\theta$ was defined as
$$\varphi_{\theta,\can}(x):=\sup_{\varphi\in\mathcal S_\theta}\varphi(x) \quad\text{for }x\in X.$$

\emph{Step 1.}
Part \eqref{enu:a} follows from Theorem \ref{thm:supercanbig}\eqref{enu:4}\eqref{enu:5}. Part \eqref{enu:a1} follows from Theorem \ref{thm:supercanbig}\eqref{enu:7}, and \eqref{enu:c} follows from Theorem \ref{thm:supercanbig}\eqref{enu:8}.

For part \eqref{enu:b}, first notice that for each positive integer $\ell$, for each $m\in N_\ell$ and for each $\sigma\in V_{h_\ell,m}$ we have $\log|\sigma|^{1/m}_{h_\ell^m}\in\mathcal S_{\alpha_\ell}$ as in Step 1 of the proof of Theorem \ref{thm:supercanbig}. Then we conclude by Lemma \ref{lem:supercanfirstproperties}(b).

\medskip

\emph{Step 2.}
Next we show \eqref{enu:d}. Let $p$ be a positive integer as in Theorem \ref{thm:supercanpsefrepresentatives}. Fix a sequence $\mathfrak m:=\{m_\ell\}_{\ell\in\N_{>0}}$ satisfying $m_\ell\in N_\ell$, $m_\ell\geq 2p\ell$ and $\lim\limits_{\ell\to\infty}\ell/m_\ell=0$, and set
$$\varphi_{\mathfrak m,\alg}:=\limsup_{\ell\to\infty}\varphi_{h_\ell,m_\ell}.$$
By \eqref{enu:b} we have that all the functions $\varphi_{h_\ell,m_\ell}$ are uniformly bounded from above on $X$, hence $\varphi_{\mathfrak m,\alg}$ is well defined. It suffices to prove that
\begin{equation}\label{eq:equalityalgcan}
(\varphi_{\mathfrak m,\alg})^*=\varphi_{\alpha,\can}.
\end{equation}

We first show that
\begin{equation}\label{eq:inequality1}
(\varphi_{\mathfrak m,\alg})^*\leq\varphi_{\alpha,\can}.
\end{equation}
To that end, fix $x\in X$. We may assume that $(\varphi_{\mathfrak m,\alg})^*(x)\neq{-}\infty$, since otherwise the claim is clear. Then there exists a sequence $\{x_n\}_{n\in\N}$ of points in $X$ such that $x_n\to x$ and
$$(\varphi_{\mathfrak m,\alg})^*(x)=\limsup\limits_{z\to x}\varphi_{\mathfrak m,\alg}(z)=\lim_{n\to\infty}\varphi_{\mathfrak m,\alg}(x_n),$$
hence, by the definition of $\varphi_{\mathfrak m,\alg}$, there exists an increasing sequence $\{\ell_n\}_{n\in\N}$ of positive integers such that
\begin{equation}\label{eq:lim}
(\varphi_{\mathfrak m,\alg})^*(x)=\lim_{n\to\infty}\varphi_{h_{\ell_n},m_{\ell_n}}(x_n).
\end{equation}
By \eqref{enu:a}, for each $n$ there exists a section $\sigma_n\in V_{h_{\ell_n},m_{\ell_n}}$ such that
$$\varphi_{h_{\ell_n},m_{\ell_n}}(x_n)=\frac{1}{m_{\ell_n}}\log|\sigma_n(x_n)|_{h_{\ell_n}^{m_{\ell_n}}},$$
hence \eqref{eq:lim} gives
\begin{equation}\label{eq:lim3}
(\varphi_{\mathfrak m,\alg})^*(x)=\lim_{n\to\infty}\frac{1}{m_{\ell_n}}\log|\sigma_n(x_n)|_{h_{\ell_n}^{m_{\ell_n}}}.
\end{equation}
Note that as in Step 1 we have $\frac{1}{m_{\ell_n}}\log|\sigma_n|_{h_{\ell_n}^{m_{\ell_n}}}\in\PSH(X,\alpha_{\ell_n})$ and all these functions are uniformly bounded. Therefore, by Theorem \ref{thm:compactnessquasipsh}(e) and after passing to a subsequence we may assume that the sequence of functions $\big\{\frac{1}{m_{\ell_n}}\log|\sigma_n|_{h_{\ell_n}^{m_{\ell_n}}}\big\}_{n\in\N}$ converges in $L^1_\loc(X)$ and almost everywhere to a function $\widetilde\varphi\in\PSH(X,\alpha)$, and then $\widetilde\varphi\in\mathcal S_\alpha$ by Fatou's lemma. In particular, we have 
\begin{equation}\label{eq:7d}
\widetilde\varphi(x)\leq\varphi_{\alpha,\can}(x)
\end{equation}
by the definition of $\varphi_{\alpha,\can}$. On the other hand, by Lemma \ref{lem:convergencepsh} and by \eqref{eq:lim3} we have
$$\widetilde\varphi(x)\geq\limsup_{n\to\infty}\frac{1}{m_{\ell_n}}\log|\sigma_n(x_n)|_{h_{\ell_n}^{m_{\ell_n}}}=(\varphi_{\mathfrak m,\alg})^*(x),$$
which together with \eqref{eq:7d} shows \eqref{eq:inequality1}.

For the reverse inequality, let $\varphi\in\mathcal S_\alpha$. Then by Theorem \ref{thm:supercanpsefrepresentatives} there exists a sequence of sections $\tau_\ell\in V_{h_\ell,m_\ell}$ such that
$$\varphi=\Big(\limsup_{\ell\to\infty}\log|\tau_\ell|^{1/m_\ell}_{h_\ell^{m_\ell}}\Big)^*.$$
Since $\log|\tau_\ell|^{1/m_\ell}_{h_\ell^{m_\ell}}\leq\varphi_{h_\ell,m_\ell}$ for each $\ell$ by the definition of $\varphi_{h_\ell,m_\ell}$, we obtain
$$\varphi\leq\Big(\limsup_{\ell\to\infty}\varphi_{h_\ell,m_\ell}\Big)^*=(\varphi_{\mathfrak m,\alg})^*,$$
hence
$$\varphi_{\alpha,\can}=\sup_{\varphi_\in\mathcal S_\alpha}\varphi\leq(\varphi_{\mathfrak m,\alg})^*.$$
This together with \eqref{eq:inequality1} shows \eqref{eq:equalityalgcan}.

\medskip

\emph{Step 3.}
Part \eqref{enu:e} follows from Lemma \ref{lem:supercanfirstproperties}(c), and \eqref{enu:f0} follows from Theorem \ref{thm:supercanbig}\eqref{enu:11}.

\medskip

\emph{Step 4.}
Finally, in this step we show \eqref{enu:f}. Denote
$$\widehat\varphi:=\lim\limits_{\ell\to\infty}\varphi_{\alpha_\ell,\can},$$
which is well defined by \eqref{enu:e} and we have
\begin{equation}\label{eq:inequality2}
\widehat\varphi\geq\varphi_{\alpha,\can}.
\end{equation}

In order to prove \eqref{enu:f} we need to show the reverse inequality. First note that $\widehat\varphi\in\PSH(X,\alpha)$ by \eqref{eq:inequality2}, \eqref{enu:e} and Theorem \ref{thm:compactnessquasipsh}(e), hence by Corollary \ref{cor:strongusc} it suffices to show that
\begin{equation}\label{eq:inequality3}
\widehat\varphi\leq\varphi_{\alpha,\can}\quad\text{on }X\setminus\sB_-(L),
\end{equation}
since $\sB_-(L)$ is a countable union of analytically closed subsets of $X$, hence of Lebesgue measure zero. 

To show \eqref{eq:inequality3}, fix a point $x\in X\setminus\sB_-(L)$, and let $\varepsilon>0$. Fix a sequence $\{m_\ell\}_{\ell\in\N_{>0}}$ satisfying $m_\ell\in N_\ell$, $m_\ell\geq 2p\ell$ and $\lim\limits_{\ell\to\infty}\ell/m_\ell=0$, where $p$ is a positive integer as in \eqref{enu:d}. Since $x\in X\setminus\sB_+(L_\ell)$ by Remark \ref{rem:augmented}, we have by \eqref{enu:a1} and \eqref{enu:f0} that there exists a positive integer $\ell$ such that
\begin{equation}\label{eq:epsilon}
\varphi_{\alpha_\ell,\can}(x)\leq\varphi_{h_\ell,m_\ell}(x)+\varepsilon.
\end{equation}
Therefore, taking limes superior in \eqref{eq:epsilon} as $\ell\to\infty$, and then taking the upper semicontinuous regularisation, by \eqref{enu:d} we obtain
$$\widehat\varphi(x)\leq\varphi_{\alpha,\can}(x)+\varepsilon.$$
We conclude by letting $\varepsilon\to0$.
\end{proof}

Finally, we have:

\begin{proof}[Proof of Theorem \ref{thm:main2}]
Part (a) of the theorem is an immediate consequence of Theorem \ref{thm:supercanpsef}\eqref{enu:a}\eqref{enu:c}, whereas (b) follows from Theorem \ref{thm:supercanpsef}\eqref{enu:f} and Theorem \ref{thm:compactnessquasipsh}(d).

If $K_X+\Delta$ is nef, then the pair $(Y,\Delta_Y)$ has a minimal model by \cite[Lemma 2.14(e)]{LX23}. Then (c) follows from Theorem \ref{thm:localPL}(a)(b), whereas (e) is a special case of Theorem \ref{thm:localPL}(c). Finally, (d) follows from (c) and from Theorem \ref{thm:supercanpsef}\eqref{enu:f0}.
\end{proof}

\newpage

\addtocontents{toc}{\vspace{1\baselineskip}}

\bibliographystyle{amsalpha}

\bibliography{biblio}

\newcommand{\etalchar}[1]{$^{#1}$}
\providecommand{\bysame}{\leavevmode\hbox to3em{\hrulefill}\thinspace}
\providecommand{\MR}{\relax\ifhmode\unskip\space\fi MR }
\providecommand{\MRhref}[2]{%
  \href{http://www.ams.org/mathscinet-getitem?mr=#1}{#2}
}
\providecommand{\href}[2]{#2}
\begin{thebibliography}{BCHM10}

\bibitem[BBJ21]{BBJ21}
R.~J. Berman, S.~Boucksom, and M.~Jonsson, \emph{A variational approach to the
  {Y}au-{T}ian-{D}onaldson conjecture}, J. Amer. Math. Soc. \textbf{34} (2021),
  no.~3, 605--652.

\bibitem[BBP13]{BBP13}
S.~Boucksom, A.~Broustet, and G.~Pacienza, \emph{Uniruledness of stable base
  loci of adjoint linear systems via {M}ori theory}, Math. Z. \textbf{275}
  (2013), no.~1-2, 499--507.

\bibitem[BCHM10]{BCHM}
C.~Birkar, P.~Cascini, C.~D. Hacon, and J.~M\textsuperscript{c}Kernan,
  \emph{Existence of minimal models for varieties of log general type}, J.
  Amer. Math. Soc. \textbf{23} (2010), no.~2, 405--468.

\bibitem[BD12]{BD12}
R.~Berman and J.-P. Demailly, \emph{Regularity of plurisubharmonic upper
  envelopes in big cohomology classes}, Perspectives in analysis, geometry, and
  topology, Progr. Math., vol. 296, Birkh\"{a}user/Springer, New York, 2012,
  pp.~39--66.

\bibitem[BDPP13]{BDPP}
S.~Boucksom, J.-P. Demailly, M.~P{\u{a}}un, and Th. Peternell, \emph{The
  pseudo-effective cone of a compact {K}\"ahler manifold and varieties of
  negative {K}odaira dimension}, J. Algebraic Geom. \textbf{22} (2013), no.~2,
  201--248.

\bibitem[BEGZ10]{BEGZ10}
S.~Boucksom, P.~Eyssidieux, V.~Guedj, and A.~Zeriahi, \emph{Monge-{A}mp\`ere
  equations in big cohomology classes}, Acta Math. \textbf{205} (2010), no.~2,
  199--262.

\bibitem[Ber10]{Ber10}
B.~Berndtsson, \emph{An introduction to things {$\overline\partial$}}, Analytic
  and algebraic geometry, IAS/Park City Math. Ser., vol.~17, Amer. Math. Soc.,
  Providence, RI, 2010, pp.~7--76.

\bibitem[Bou02]{Bou02}
S.~Boucksom, \emph{{C\^ones positifs des vari\'et\'es complexes compactes}},
  PhD Thesis, Universit\'e Joseph Fourier Grenoble, 2002, available at
  \url{https://theses.hal.science/tel-00002268}\setbox0=\hbox{2002}.

\bibitem[Bou04]{Bou04}
\bysame, \emph{Divisorial {Z}ariski decompositions on compact complex
  manifolds}, Ann. Sci. \'Ecole Norm. Sup. (4) \textbf{37} (2004), no.~1,
  45--76.

\bibitem[BT06]{BT06}
F.~Bracci and S.~Trapani, \emph{{Notes on pluripotential theory}}, available at
  \url{http://cmtp.uniroma2.it/~fbracci/download/pluripotential.pdf}\setbox0=\hbox{2006}.

\bibitem[Cao14]{Cao14}
J.~Cao, \emph{Numerical dimension and a {K}awamata-{V}iehweg-{N}adel-type
  vanishing theorem on compact {K}\"ahler manifolds}, Compos. Math.
  \textbf{150} (2014), no.~11, 1869--1902.

\bibitem[CL12]{CL12a}
P.~Cascini and V.~Lazi{\'c}, \emph{New outlook on the {M}inimal {M}odel
  {P}rogram, {I}}, Duke Math. J. \textbf{161} (2012), no.~12, 2415--2467.

\bibitem[CL13]{CL13}
A.~Corti and V.~Lazi{\'c}, \emph{New outlook on the {M}inimal {M}odel
  {P}rogram, {II}}, Math. Ann. \textbf{356} (2013), no.~2, 617--633.

\bibitem[DEL00]{DEL00}
J.-P. Demailly, L.~Ein, and R.~Lazarsfeld, \emph{A subadditivity property of
  multiplier ideals}, Michigan Math. J. \textbf{48} (2000), 137--156.

\bibitem[Dem82]{Dem82}
J.-P. Demailly, \emph{Estimations {$L\sp{2}$} pour l'op\'{e}rateur {$\bar
  \partial$} d'un fibr\'{e} vectoriel holomorphe semi-positif au-dessus d'une
  vari\'{e}t\'{e} k\"{a}hl\'{e}rienne compl\`ete}, Ann. Sci. \'{E}cole Norm.
  Sup. (4) \textbf{15} (1982), no.~3, 457--511.

\bibitem[Dem92a]{Dem92}
\bysame, \emph{Regularization of closed positive currents and intersection
  theory}, J. Algebraic Geom. \textbf{1} (1992), no.~3, 361--409.

\bibitem[Dem92b]{Dem92b}
\bysame, \emph{Singular {H}ermitian metrics on positive line bundles}, Complex
  algebraic varieties ({B}ayreuth, 1990), Lecture Notes in Math., vol. 1507,
  Springer, Berlin, 1992, pp.~87--104.

\bibitem[Dem01]{Dem01}
\bysame, \emph{Multiplier ideal sheaves and analytic methods in algebraic
  geometry}, School on {V}anishing {T}heorems and {E}ffective {R}esults in
  {A}lgebraic {G}eometry ({T}rieste, 2000), ICTP Lect. Notes, vol.~6, Abdus
  Salam Int. Cent. Theoret. Phys., Trieste, 2001, pp.~1--148.

\bibitem[Dem12]{Dem12}
\bysame, \emph{{Complex Analytic and Differential Geometry}}, available at
  \url{https://www-fourier.ujf-grenoble.fr/~demailly/manuscripts/agbook.pdf}\setbox0=\hbox{2012}.

\bibitem[Dem15]{Dem15}
\bysame, \emph{On the cohomology of pseudoeffective line bundles}, Complex
  geometry and dynamics, Abel Symp., vol.~10, Springer, Cham, 2015, pp.~51--99.

\bibitem[DNT21]{DT21}
E.~Di~Nezza and S.~Trapani, \emph{{The regularity of envelopes}},
  arXiv:2110.14314, to appear in Ann. Sci. \'Ecole Norm.
  Sup\setbox0=\hbox{2021}.

\bibitem[DP03]{DP03}
J.-P. Demailly and Th. Peternell, \emph{A {K}awamata-{V}iehweg vanishing
  theorem on compact {K}\"ahler manifolds}, J. Differential Geom. \textbf{63}
  (2003), no.~2, 231--277.

\bibitem[DPS01]{DPS01}
J.-P. Demailly, Th. Peternell, and M.~Schneider, \emph{{Pseudo-effective line
  bundles on compact K\"ahler manifolds}}, Int. J. Math. \textbf{12} (2001),
  no.~6, 689--741.

\bibitem[ELM{\etalchar{+}}06]{ELMNP}
L.~Ein, R.~Lazarsfeld, M.~Musta{\c{t}}{\u{a}}, M.~Nakamaye, and M.~Popa,
  \emph{Asymptotic invariants of base loci}, Ann. Inst. Fourier (Grenoble)
  \textbf{56} (2006), no.~6, 1701--1734.

\bibitem[Fav99]{Fav99}
C.~Favre, \emph{Note on pull-back and {L}elong number of currents}, Bull. Soc.
  Math. France \textbf{127} (1999), no.~3, 445--458.

\bibitem[GL13]{GL13}
Y.~Gongyo and B.~Lehmann, \emph{Reduction maps and minimal model theory},
  Compos. Math. \textbf{149} (2013), no.~2, 295--308.

\bibitem[GM17]{GM17}
Y.~Gongyo and S.-i. Matsumura, \emph{Versions of injectivity and extension
  theorems}, Ann. Sci. \'Ec. Norm. Sup\'er. (4) \textbf{50} (2017), no.~2,
  479--502.

\bibitem[GZ05]{GZ05}
V.~Guedj and A.~Zeriahi, \emph{Intrinsic capacities on compact {K}\"{a}hler
  manifolds}, J. Geom. Anal. \textbf{15} (2005), no.~4, 607--639.

\bibitem[GZ15]{GZ15}
Q.~Guan and X.~Zhou, \emph{A proof of {D}emailly's strong openness conjecture},
  Ann. of Math. (2) \textbf{182} (2015), no.~2, 605--616.

\bibitem[GZ17]{GZ17}
V.~Guedj and A.~Zeriahi, \emph{Degenerate complex {M}onge-{A}mp\`ere
  equations}, EMS Tracts in Mathematics, vol.~26, European Mathematical Society
  (EMS), Z\"urich, 2017.

\bibitem[HH20]{HH20}
K.~Hashizume and Z.-Y. Hu, \emph{On minimal model theory for log abundant lc
  pairs}, J. Reine Angew. Math. \textbf{767} (2020), 109--159.

\bibitem[Hie14]{Hie14}
P.~H. Hiep, \emph{The weighted log canonical threshold}, C. R. Math. Acad. Sci.
  Paris \textbf{352} (2014), no.~4, 283--288.

\bibitem[H{\"{o}}r65]{Hoer65}
L.~H{\"{o}}rmander, \emph{{$L\sp{2}$} estimates and existence theorems for the
  {$\bar\partial $} operator}, Acta Math. \textbf{113} (1965), 89--152.

\bibitem[H{\"{o}}r90]{Hoer90}
\bysame, \emph{An introduction to complex analysis in several variables}, third
  ed., North-Holland Mathematical Library, vol.~7, North-Holland Publishing
  Co., Amsterdam, 1990.

\bibitem[H{\"{o}}r07]{Hoer07}
\bysame, \emph{Notions of convexity}, Modern Birkh\"{a}user Classics,
  Birkh\"{a}user Boston, Inc., Boston, MA, 2007, Reprint of the 1994 edition.

\bibitem[IMM24]{IMM24}
M.~Iwai, S.-i. Matsumura, and N.~M{\"u}ller, \emph{{Abundance theorem for
  minimal projective varieties satisfying Miyaoka's equality}},
  arXiv:2404.07568\setbox0=\hbox{2024}.

\bibitem[Kaw85]{Kaw85b}
Y.~Kawamata, \emph{Minimal models and the {K}odaira dimension of algebraic
  fiber spaces}, J. Reine Angew. Math. \textbf{363} (1985), 1--46.

\bibitem[Kaw92]{Kaw92}
\bysame, \emph{Abundance theorem for minimal threefolds}, Invent. Math.
  \textbf{108} (1992), no.~2, 229--246.

\bibitem[Kis00]{Kis00}
C.~O. Kiselman, \emph{Ensembles de sous-niveau et images inverses des fonctions
  plurisousharmoniques}, Bull. Sci. Math. \textbf{124} (2000), no.~1, 75--92.

\bibitem[KKL16]{KKL16}
A.-S. Kaloghiros, A.~K\"uronya, and V.~Lazi\'c, \emph{Finite generation and
  geography of models}, Minimal {M}odels and {E}xtremal {R}ays ({K}yoto 2011),
  Adv. Stud. Pure Math., vol.~70, Mathematical Society of Japan, Tokyo, 2016,
  pp.~215--245.

\bibitem[Kli91]{Kli91}
M.~Klimek, \emph{Pluripotential theory}, London Mathematical Society
  Monographs. New Series, vol.~6, The Clarendon Press, Oxford University Press,
  New York, 1991, Oxford Science Publications.

\bibitem[KM98]{KM98}
J.~Koll{\'a}r and S.~Mori, \emph{Birational geometry of algebraic varieties},
  Cambridge Tracts in Mathematics, vol. 134, Cambridge University Press,
  Cambridge, 1998.

\bibitem[KMM94]{KMM94}
S.~Keel, K.~Matsuki, and J.~M\textsuperscript{c}Kernan, \emph{Log abundance
  theorem for threefolds}, Duke Math.\ J. \textbf{75} (1994), 99--119.

\bibitem[Laz04]{Laz04}
R.~Lazarsfeld, \emph{Positivity in algebraic geometry. {I}, {II}}, Ergebnisse
  der Mathematik und ihrer Grenzgebiete, vol. 48, 49, Springer-Verlag, Berlin,
  2004.

\bibitem[Laz24]{Laz24}
V.~Lazi\'c, \emph{{A few remarks on effectivity and good minimal models}},
  arXiv:2401.14190\setbox0=\hbox{2024}.

\bibitem[Lem17]{Lem17}
L.~Lempert, \emph{Modules of square integrable holomorphic germs}, Analysis
  meets geometry, Trends Math., Birkh\"{a}user/Springer, Cham, 2017,
  pp.~311--333.

\bibitem[LM21]{LM21}
V.~Lazi\'{c} and F.~Meng, \emph{On nonvanishing for uniruled log canonical
  pairs}, Electron. Res. Arch. \textbf{29} (2021), no.~5, 3297--3308.

\bibitem[LM23]{LM23}
H.~Liu and S.-i. Matsumura, \emph{Strictly nef divisors on {K}-trivial
  fourfolds}, Math. Ann. \textbf{387} (2023), no.~1-2, 985--1008.

\bibitem[LMT23]{LMT23}
V.~Lazi\'{c}, J.~Moraga, and N.~Tsakanikas, \emph{Special termination for log
  canonical pairs}, Asian J. Math. \textbf{27} (2023), no.~3, 423--440.

\bibitem[LP18]{LP18}
V.~Lazi\'c and Th. Peternell, \emph{Abundance for varieties with many
  differential forms}, \'Epijournal Geom. Alg\'ebrique \textbf{2} (2018),
  Article 1.

\bibitem[LP20a]{LP20a}
\bysame, \emph{{On Generalised Abundance, I}}, Publ. Res. Inst. Math. Sci.
  \textbf{56} (2020), no.~2, 353--389.

\bibitem[LP20b]{LP20b}
\bysame, \emph{{On Generalised Abundance, II}}, Peking Math. J. \textbf{3}
  (2020), no.~1, 1--46.

\bibitem[LT22]{LT22a}
V.~Lazi\'{c} and N.~Tsakanikas, \emph{Special {MMP} for log canonical
  generalised pairs (with an appendix joint with {X}iaowei {J}iang)}, Selecta
  Math. (N.S.) \textbf{28} (2022), no.~5, Paper No. 89.

\bibitem[LX23]{LX23}
V.~Lazi\'c and Z.~Xie, \emph{{Nakayama--Zariski decomposition and the
  termination of flips}}, arXiv:2305.01752\setbox0=\hbox{2023}.

\bibitem[LX24]{LX24}
\bysame, \emph{{Rigid currents in birational geometry}},
  arXiv:2402.05807\setbox0=\hbox{2024}.

\bibitem[Miy87]{Miy87}
Y.~Miyaoka, \emph{The {C}hern classes and {K}odaira dimension of a minimal
  variety}, Algebraic geometry, {S}endai, 1985, Adv. Stud. Pure Math., vol.~10,
  North-Holland, Amsterdam, 1987, pp.~449--476.

\bibitem[Miy88a]{Miy88b}
\bysame, \emph{Abundance conjecture for {$3$}-folds: case {$\nu=1$}},
  Compositio Math. \textbf{68} (1988), no.~2, 203--220.

\bibitem[Miy88b]{Miy88a}
\bysame, \emph{On the {K}odaira dimension of minimal threefolds}, Math. Ann.
  \textbf{281} (1988), no.~2, 325--332.

\bibitem[Nak04]{Nak04}
N.~Nakayama, \emph{Zariski-decomposition and abundance}, MSJ Memoirs, vol.~14,
  Mathematical Society of Japan, Tokyo, 2004.

\bibitem[NS68]{NS68}
M.~S. Narasimhan and R.~R. Simha, \emph{Manifolds with ample canonical class},
  Invent. Math. \textbf{5} (1968), 120--128.

\bibitem[OT87]{OT87}
T.~Ohsawa and K.~Takegoshi, \emph{On the extension of {$L^2$} holomorphic
  functions}, Math. Z. \textbf{195} (1987), no.~2, 197--204.

\bibitem[Sho85]{Sho85}
V.~V. Shokurov, \emph{A nonvanishing theorem}, Izv. Akad. Nauk SSSR Ser. Mat.
  \textbf{49} (1985), no.~3, 635--651.

\bibitem[Siu74]{Siu74}
Y.~T. Siu, \emph{Analyticity of sets associated to {L}elong numbers and the
  extension of closed positive currents}, Invent. Math. \textbf{27} (1974),
  53--156.

\bibitem[Tsu07]{Tsu07}
H.~Tsuji, \emph{{Canonical volume forms on compact K\"ahler manifolds}},
  arXiv:0707.0111\setbox0=\hbox{2007}.

\bibitem[Tsu11]{Tsu11}
\bysame, \emph{Canonical singular {H}ermitian metrics on relative canonical
  bundles}, Amer. J. Math. \textbf{133} (2011), no.~6, 1469--1501.

\bibitem[TX23]{TX23}
N.~Tsakanikas and Z.~Xie, \emph{{Comparison and uniruledness of asymptotic base
  loci}}, arXiv:2309.01031\setbox0=\hbox{2023}.

\bibitem[Vu19]{Vu19}
D.-V. Vu, \emph{Locally pluripolar sets are pluripolar}, Internat. J. Math.
  \textbf{30} (2019), no.~13.

\bibitem[Vu21]{Vu21}
\bysame, \emph{{An introduction to pluripotential theory}}, available at
  \url{http://www.mi.uni-koeln.de/~vuviet/Vu_lectures.chapter1-3.pdf}\setbox0=\hbox{2021}.

\end{thebibliography}
\end{document}